\documentclass[]{informs3noheader}

\OneAndAHalfSpacedXI
\usepackage{hyperref}
\usepackage[utf8]{inputenc}
\usepackage{graphicx}
\usepackage[numbers]{natbib}
\usepackage{booktabs,caption,fixltx2e,threeparttable}
\usepackage{amssymb}
\usepackage{subcaption}
\usepackage{float}
\usepackage{indentfirst}
\usepackage{textcomp}
\usepackage{mathtools}
\usepackage{tabularx}
\usepackage{amsmath}
\usepackage{amsfonts}
\usepackage{todonotes}
\usepackage{optidef}
\usepackage{bbm}
\usepackage{algpseudocode}
\usepackage[shortlabels]{enumitem}
\usepackage{algorithm}
\usepackage{algorithmicx}
\usepackage{algpseudocode}
\usepackage{bm}
\usepackage{comment}
\usepackage{multirow}

\usepackage{natbib}
\bibpunct[, ]{(}{)}{,}{a}{}{,}%
\def\bibfont{\small}%

\DeclareMathOperator{\A}{\calA}
\DeclareMathOperator{\B}{\calB}

\DeclareMathOperator{\rmp}{\texttt{RMP}}
\DeclareMathOperator{\crewsp}{\textsc{CrewSubproblem}}
\DeclareMathOperator{\firesp}{\textsc{FireSubproblem}}

\TheoremsNumberedThrough
\EquationsNumberedThrough

\allowdisplaybreaks

\MANUSCRIPTNO{}

\def\R{\mathbb{R}}
\def\Z{\mathbb{Z}}

\def\E{\mathbb{E}}

\def\calA{\mathcal{A}}
\def\calB{\mathcal{B}}

\def\calD{\mathcal{D}}
\def\calE{\mathcal{E}}
\def\calF{\mathcal{F}}
\def\calG{\mathcal{G}}
\def\calH{\mathcal{H}}

\def\calJ{\mathcal{J}}

\def\calL{\mathcal{L}}
\def\calM{\mathcal{M}}
\def\calN{\mathcal{N}}
\def\calO{\mathcal{O}}
\def\calP{\mathcal{P}}
\def\calQ{\mathcal{Q}}

\def\calS{\mathcal{S}}
\def\calT{\mathcal{T}}
\def\calU{\mathcal{U}}
\def\calV{\mathcal{V}}

\def\bX{\boldsymbol{X}}

\def\bi{\boldsymbol{i}}

\def\by{\boldsymbol{y}}
\def\bz{\boldsymbol{z}}

\def\bdelta{\boldsymbol{\delta}}

\def\bsigma{\boldsymbol{\sigma}}

\def\bpi{\boldsymbol{\pi}}

\def\brho{\boldsymbol{\rho}}

\definecolor{mygreen}    {RGB}{0,90,0}
\definecolor{myblue}     {RGB}{0,51,140}
\definecolor{myorange}   {RGB}{238,118,0}
\definecolor{myred}      {RGB}{126,0,0}
\definecolor{mygray}     {RGB}{100,100,105}
\definecolor{mygrayblue} {RGB}{0,128,128}
\definecolor{mygraygreen}{RGB}{128,128,0}
\definecolor{DarkPurple}     {RGB}{142, 36, 170}
\definecolor{LightPurple}    {RGB}{57, 130, 7}

\begin{document}

\ARTICLEAUTHORS{
    \AUTHOR{Léonard Boussioux}
    \AFF{Michael G. Foster School of Business and Paul G. Allen School of Computer Science \& Engineering,\\ University of Washington, Seattle, WA, \EMAIL{leobix@uw.edu}}
    \AUTHOR{Alexandre Jacquillat, Ryne Reger, Jacob Wachspress}
    \AFF{Operations Research Center and Sloan School of Management,\\ Massachusetts Institute of Technology, Cambridge, MA, \EMAIL{alexjacq@mit.edu}}
}

\RUNAUTHOR{Boussioux, Jacquillat, Reger, and Wachspress}

\RRHSecondLine{}
\LRHSecondLine{}

\RUNTITLE{Optimizing Wildfire Suppression}

\TITLE{Predictive and Prescriptive AI toward\\Optimizing Wildfire Suppression}

\ABSTRACT{Intense wildfire seasons require critical prioritization decisions to allocate scarce suppression resources over a dispersed geographical area. This paper develops a predictive and prescriptive approach to jointly optimize crew assignments and wildfire suppression. The problem features a discrete resource-allocation structure with endogenous wildfire demand and non-linear wildfire dynamics. We formulate an integer optimization model with crew assignments on a time-space-rest network, wildfire dynamics on a time-state network, and linking constraints between them. We develop a two-sided branch-and-price-and-cut algorithm based on: (i) a two-sided column generation scheme that generates fire suppression plans and crew routes iteratively; (ii) a new family of cuts exploiting the knapsack structure of the linking constraints; and (iii) novel branching rules to accommodate non-linear wildfire dynamics. We also propose a data-driven double machine learning approach to estimate wildfire spread as a function of covariate information and suppression efforts, mitigating observed confounding between historical crew assignments and wildfire growth. Extensive computational experiments show that the optimization algorithm scales to otherwise intractable real-world instances; and that the methodology can enhance suppression effectiveness in practice, resulting in significant reductions in area burned over a wildfire season and guiding resource sharing across wildfire jurisdictions.}


\maketitle

\vspace{-12pt}
\section{Introduction}

Wildfires are a global threat to livelihoods and ecosystems. The 2025 Palisades and Eaton fires in Los Angeles County burned 37,469 acres, destroyed or damaged 18,315 structures, led to billions of dollars in economic losses, displaced thousands of people, and caused 31 civilian fatalities \citep{calfireEaton2025,calfirePalisades2025}. Intense wildfire seasons are routine occurrences in several countries, and climate change is projected to increase extreme wildfire activity by 14\% by 2030, 30\% by 2050, and 50\% by 2100 \citep{WMO}. At the same time, limited fire management resources require difficult prioritization decisions. As early as 2021, the Chief of the US Forest Service wrote \textit{``We are in a `triage mode' where our primary focus must be on fires that threaten communities and infrastructure''}.\footnote{https://forestpolicypub.com/2021/08/05/the-chiefs-wildland-fire-direction-letter/}

As part of this challenge, this paper develops a predictive and prescriptive approach to support wildfire suppression in periods of intense, synchronous wildfire activity using artificial intelligence (AI), machine learning, and optimization. Specifically, our problem takes as input a set of wildfires with their propagation dynamics and a set of firefighting crews with their rest requirements. It then jointly optimizes the intensity and timing of suppression efforts for each wildfire and a corresponding workplan for each crew. This problem features two main technical challenges:
\begin{itemize}
    \item[--] \textit{Modeling wildfire spread.} Wildfire propagation can be modeled from historical and environmental factors via physics-based and machine learning methods (Section~\ref{sec:literature}). However, our prescriptive approach requires counterfactuals to estimate the impact of suppression efforts on wildfire propagation. This estimation problem is complicated by the endogeneity of historical crew assignments and wildfire growth---more crews were assigned to high-risk fires. Such confounding may lead to the underestimation of the impact of crews on wildfire suppression.
    \item[--] \textit{A discrete and non-linear optimization structure.} The optimization problem combines triage and routing decisions to assign scarce suppression crews across synchronous and dispersed wildfires. Compared with traditional triage operations, geographic dispersion creates a trade-off between prioritizing high-risk fires versus easily accessible ones. Compared with traditional routing operations, the problem features endogenous, non-linear demand due to dependencies between crew assignments and fire spread. Thus, the problem can be cast as a non-linear triage problem with heterogeneous switching costs, or as a routing problem where ``customer demand'' depends (non-linearly) on allocation decisions. This novel structure motivates our optimization methodology to solve it at scale and support wildfire suppression in practice.
\end{itemize}

In response, this paper develops an integrated data-driven approach to wildfire suppression, using double machine learning and large-scale optimization. We have assembled a detailed dataset that combines information on wildfire spread, historical suppression efforts, and environmental factors from meteorological data, terrain characteristics, and AI foundation models. In the predictive task, we use these data to model wildfire spread as a function of crew assignments and covariate information. In the prescriptive task, we implement our optimization methodology in a data-driven setup that reflects historical allocation decisions within the US fire management system.

The first contribution is to formulate an integer optimization model of wildfire suppression and crew assignments. We consider a general-purpose wildfire spread model of the form $S_{g,t+1}=f_{gt}(S_{gt},x_{gt})$, which governs the evolution of the state of fire $g$ from $S_{gt}$ to $S_{g,t+1}$ (e.g., burned area) as a function of suppression efforts $x_{gt}$ (e.g., crew assignments) and of contextual information encapsulated in the function $f_{gt}$ (e.g., environmental conditions). Viewed as a dynamic program, the problem faces the curse of dimensionality due to exponential growth in the state space across multiple wildfires \citep{powell2007approximate}; viewed as a mixed-integer program, it is complicated by non-linear linking constraints to model endogenous wildfire dynamics. To recover a tractable formulation, we represent wildfire propagation in a time-state network and crew operations in a time-space-rest network that captures travel times and rest requirements---an important consideration in practice due to hazardous working conditions. We propose a two-sided set partitioning formulation with path-based variables that determine a suppression plan for each fire and a routing assignment for each crew, with linking constraints to ensure consistency between fire demand and crew supply. The model can encode a broad range of non-linear fire models and suppression objectives. However, it involves an exponential number of variables, thereby requiring column generation methods.

The second contribution is to develop an exact two-sided branch-and-price-and-cut algorithm that converges to an optimal solution in a finite number of iterations, and which can handle large-scale instances arising in practice. The methodology is broken down into three components:
\begin{enumerate}
    \item \textit{A two-sided column generation algorithm} to generate fire suppression plans and crew routes iteratively. At each iteration, a restricted master problem solves a linear relaxation with a subset of the fire plans and crew routes. Pricing problems seek a new suppression plan for each fire and a new route for each crew with negative reduced costs, or prove that none exists. The algorithm alleviates the curse of dimensionality by separating the combinatorial resource-allocation complexities in the master problem and the non-linear wildfire dynamics in each pricing problem. We also propose a dual stabilization strategy to circumvent degeneracy and accelerate convergence by exploiting a common structure of wildfire spread models.
    \item \textit{A new family of augmented GUB cuts.} We develop a cutting-plane algorithm to strengthen the master problem's linear relaxation. Exploiting the knapsack structure of the fire-crew linking constraints, we derive augmented generalized-upper-bound (A-GUB) cuts that tighten the GUB cover cuts from \cite{wolsey1990valid} in our setting. We devise a cut-generating linear program to solve the separation problem efficiently within the cutting planes procedure.
    \item \textit{A novel dual-aware maximum-variance branching scheme.} We embed the two-sided column generation into a branch-and-cut tree. The traditional most-fractional branching rule does not guarantee exactness due to the non-linear wildfire dynamics. Thus, we devise a new maximum-variance branching rule and an acceleration leveraging the dual prices of the fire-crew linking constraints. We prove that these branching rules ensure the convergence of the branch-and-price-and-cut algorithm to an optimal integer solution in a finite number of iterations.
\end{enumerate}

The third contribution is a data-driven approach to estimate wildfire spread as a function of fire information and crew assignments. Covariates include the fire state (area and one-period growth, in our model) along with features such as fire characteristics, meteorological conditions, local terrain, and AlphaEarth satellite embeddings; the treatment variable specifies the number of crews assigned to the fire on each day; and the outcome variable reflects the growth of the fire by the next day. We propose a double machine learning approach to mitigate the confounding between fire characteristics and historical crew assignments \citep{chernozhukov2018double}, resulting in a smooth, decreasing relationship between crew allocations and predicted fire growth. This predictive model yields a data-driven dynamical model for each fire, with thousands of fire states in each period.

The fourth contribution is to demonstrate the performance of the methodology to guide wildfire suppression in practice. Results show that the two-sided branch-and-price-and-cut algorithm can generate high-quality solutions in large-scale instances of the problem. The methodology significantly outperforms several benchmarks, accelerating convergence by an order of magnitude in medium instances and generating provably high-quality solutions in large and otherwise intractable instances. From a practical standpoint, results suggest that the optimized solution can double the effectiveness of suppression activities---measured in the decrease in area burned from a do-nothing baseline---as compared to practical benchmarks that allocate crews based on such proxies as fire areas and access distance. We also demonstrate the potential benefits of resource pooling across jurisdictions in intense wildfire seasons, enabled by our scalable optimization algorithm. Ultimately, the proposed methodology can support the allocation of critical resources and suppression activities across synchronous wildfires---a growing problem at the core of climate change adaptation.

\section{Literature Review}
\label{sec:literature}

\paragraph{Predicting wildfire dynamics.}
Extensive research has characterized wildfire behavior, spanning ignition risk, detection, monitoring, spread, and recovery \citep[see, e.g.,][]{balbi20073d,jain2020review,finney2021wildland,crowley2023towards,illarionova2025exploration}, as well as the impact of climate change on wildfire risk \citep{bowman2020vegetation,brown2023climate}. Wildfire dynamics are typically predicted with physics-based models \citep{simeoni2011physical,cruz2013uncertainty,hoffman2016evaluating}, machine learning \citep{radke2019firecast,khanmohammadi2022prediction,shmuel2023machine}, and hybrid approaches \citep{bottero2020physics,singh2024trending}. However, these models do not build counterfactuals to estimate the impact of crew allocations on wildfire spread---a critical input into our prescriptive approach. Instead, simulation is typically used to guide wildfire suppression \citep{finney1997use,riley2018model,cigal2022sensitivity,bertsch2026calibrating}, land use \citep{ribeiro2023promoting}, evacuation \citep{siam2022interdisciplinary}, etc. \cite{kreider2024fire} show that suppression efforts can have unintended consequences on global burning patterns. \cite{marshall2022suppression} develop a random forest model to identify the drivers of fire containment.

This paper contributes a predictive model of wildfire spread that accounts for suppression intensity. We assemble a new dataset combining governmental records, weather, terrain, and satellite imagery embeddings. The problem falls into causal machine learning to estimate the (non-linear) relationships between high-dimensional covariates, treatment variables and outcome variables \citep{wager2018estimation,athey2019estimating,kaddour2022causal}. Specifically, we leverage the double machine learning approach from \cite{chernozhukov2018double,chernozhukov2022automatic} to isolate the outcome and treatment variables from the covariates in the treatment effect estimation.

\paragraph{Optimizing wildfire suppression.}

One branch of the literature optimizes preparedness through resource positioning \citep{sakellariou2020spatial,zeferino2020optimizing}, delineations for fire containment \citep{wei2018spatial,wei2021comparing}, resource sharing \citep{wei2016simulation}, fuel treatment \citep{bhuiyan2019stochastic}, and area monitoring \citep{puechwfdronebench}. Related work optimizes power systems under wildfire risk \citep{chen2026large,brun2026modeling}. Another branch optimizes suppression operations for a single fire or within a small locale \citep{donovan2003integer,hu2009integrated,wei2011toward,rodriguez2018integer,belval2019modeling,rodriguez2023application}.

Our paper tackles the intermediate tactical problem of allocating crews across fires over a multi-day horizon. \cite{pappis2010scheduling2} propose a job-scheduling model with deteriorating processing times. \cite{yang2019emergency} frame a vehicle routing problem trading off travel times and early fire suppression. \cite{wu2019resource} and \cite{wang2020multi} formulate mixed-integer linear programs. \cite{chan2021fighting} incorporate reconnaissance operations. \cite{farhadi2021traveling} study a traveling firefighter problem, a variant of the Traveling Salesperson Problem that minimizes the $\ell_2$-norm of visit times. \cite{nguyen2024optimization} optimize a fuel break network to maximize containment. \cite{suarez2024integrated} consider a two-stage stochastic program for preparedness and suppression.

To our knowledge, our paper is the first to jointly optimize crew assignments and multi-fire suppression plans under endogenous fire demand, to formulate a general-purpose model with any (non-linear) dynamical model of fire spread, and to develop an exact methodology that scales to large practical problems. The problem combines elements of triage \citep{chan2013prioritizing,sun2018patient,saghafian2018workload} and vehicle routing with time windows \citep{bard2002branch,baldacci2011new}. The flexibility to suppress each fire with one or multiple crews relates to routing with split deliveries \citep{desaulniers2010branch} or workload requirements \citep{jacquillat2024value}. The main distinction in our problem lies in interdependencies between suppression decisions and fire demand.

Methodologically, this structure falls into discrete optimization over non-linear dynamical systems, such as vaccine allocation \citep{jacquillat2024branch} and pollution management \citep{den2025optimizing}. Our two-sided set partitioning formulation relates to similar structures in airline management \citep{cordeau2001benders}, vehicle routing with consistency requirements \citep{wang2022routing}, and maritime transportation \citep{wu2022vessel}, which have been solved by Benders decomposition and column generation. This paper contributes a new branch-and-price-and-cut algorithm---with new valid inequalities and branching rules---that exploits the two-sided set partitioning structure with linking constraints between fire demand and crew supply.

\section{Data}
\label{sec:data}

\subsubsection*{Wildfire data and environmental conditions.}
We make use of the following databases:
\begin{itemize}
    \item[--] The \textit{SIT-209} database serves as the primary reporting platform for US wildfire management agencies.\footnote{Available at \url{https://www.wildfire.gov/application/sit209}.} The Situation Report (SIT) component collects daily fire activity and resource allocation at the local dispatch level. The Incident Status Summary (ICS-209) component captures comprehensive incident summaries for significant wildfires, thus providing visibility into historical wildfire activity (e.g., geospatial coordinates, burned area, fire cause, fire behavior, fuel types, containment status), and suppression activity (e.g., personnel deployment, equipment).
    \item[--] The \textit{ERA5} dataset stores re-analysis data from weather prediction models and observational data through data assimilation \citep{EAR5}. It records estimates of climate variables such as wind, temperature, humidity, precipitation, solar radiation, moisture, vegetation, etc.
    \item[--] The \textit{LandFire} dataset provides detailed information on terrain, vegetation, and fuels \citep{LandFire}. It integrates satellite imagery, field data, and ecological modeling to produce standardized, high-resolution layers used for land management and ecological research.
    \item[--] The \textit{AlphaEarth satellite embeddings} are geospatial machine learning representations derived from satellite imagery by Google DeepMind \citep{brown2025alphaearth}. These encode spatial, temporal, and spectral patterns of the Earth’s surface into compact vector representations.
\end{itemize}

\subsubsection*{Data preprocessing and temporal alignment.}
We use a daily unit of analysis, which is consistent with both the level of granularity in the data and the scope of the optimization problem---namely, tactical crew allocations rather than real-time management of firefighting resources. We conducted extensive pre-processing to ensure data consistency and data quality, including:
\begin{itemize}
    \item[--] temporal interpolation to alleviate variations in reporting frequencies, with boundary conditions to handle the first and last observations for each fire.
    \item[--] identification and elimination of duplicate same-day reports.
    \item[--] handling of ambiguous zero-growth entries. To focus on reliable fire behavior, we removed zero-growth entries, truncated each fire sequence after its first non-positive growth step, and removed sequences of 7 days or more or with fewer than 2 usable observations.
    \item[--] removal of anomalous data entries with physically implausible fire behavior, as well as entries exceeding the 85th percentile for daily growth (4,000 acres/day). This filter focuses on the common operational regime of actively growing fires, while limiting upper-tail events and reporting artifacts. Our empirical model should be interpreted within this active-fire context that drives most day-to-day dispatch decisions, as opposed to a model of tail fire behavior.
\end{itemize}

The resulting panel contains 8,336 fire-day observations over 2015–2024 (AlphaEarth features only cover 2017-2024, and we omit 2019 entirely in the analysis due to inconsistencies). We track the current area (\texttt{area}), backward-looking momentum (\texttt{area$\_$diff}, change since the previous report) and forward-looking growth (\texttt{next$\_$area}, change until the next report). In the predictive model, the latter serves as the outcome variable and the former two as features; in the prescriptive model, the former two constitute a two-dimensional state variable. The mean daily growth rate is 683 acres/day, with a median of 258 acres/day and a standard deviation of 913 acres/day. Figure~\ref{fig:wildfire_dist} shows that the distributions of daily growth, burned area and personnel deployment are heavily right-skewed, indicating significant variability in wildfire behavior and suppression efforts.

\begin{figure}[h!]
    \centering
    \subfloat[Daily burned area]{\label{subfig:daily_growth}
        \includegraphics[width=0.3\textwidth]{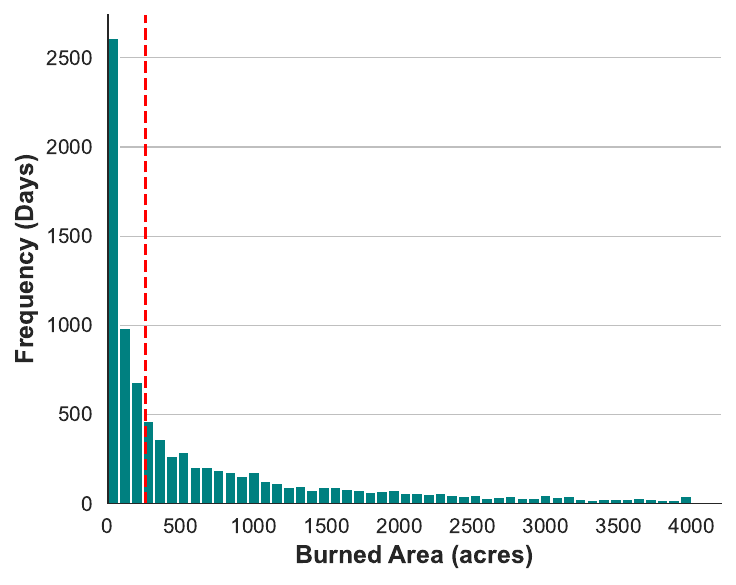}
    }
    \hfill 
    \subfloat[Total area burned]{\label{subfig:area_burned}
        \includegraphics[width=0.3\textwidth]{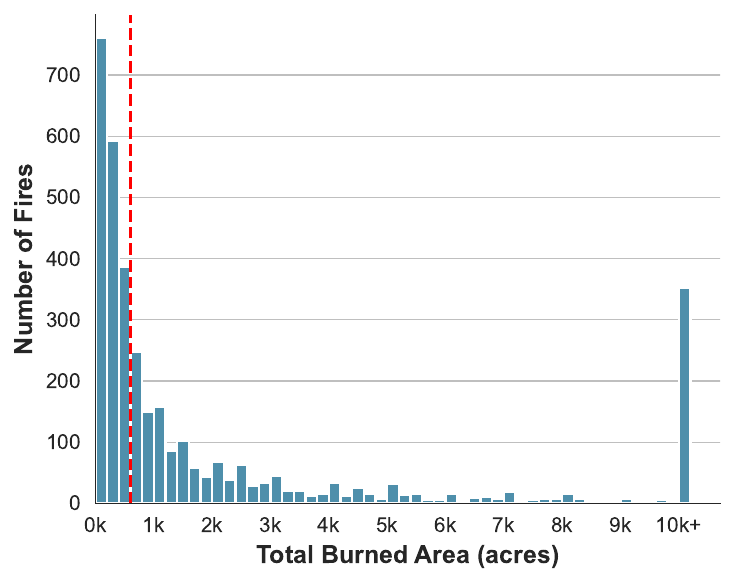}
    }
    \hfill 
    \subfloat[Peak personnel]{\label{subfig:peak_personnel}
        \includegraphics[width=0.3\textwidth]{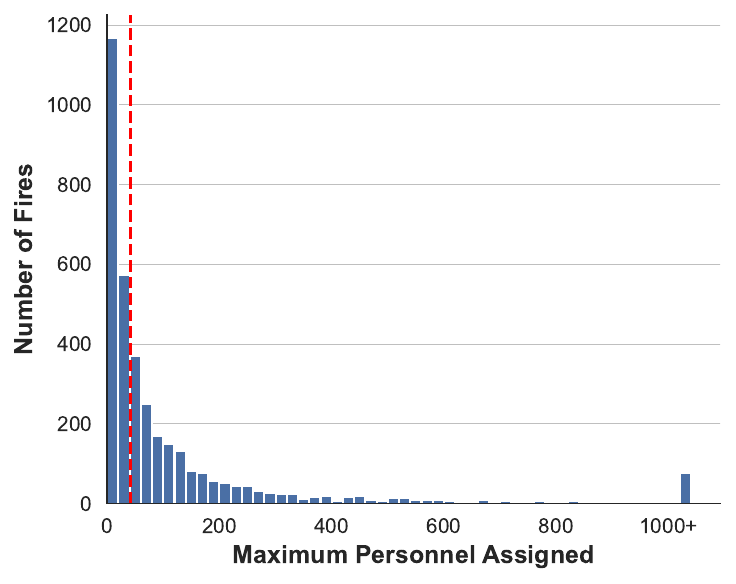}
    }
    \caption{Distribution of key wildfire characteristics from 2015-2024 (red vertical lines indicate averages).}
    \label{fig:wildfire_dist}
    \vspace{-18pt}
\end{figure}

\subsubsection*{Crew data and experimental setup.}

Throughout this paper, we focus on Interagency Hotshot Crews (IHCs), which are specialized teams of approximately 18--25 firefighters with the highest level of training and qualification \citep{ihc}. The SIT-209 data report historical allocations of IHCs to wildfires, as well as allocations of other resources (e.g., other crews, vehicles, equipment). We define an ``IHC equivalent'' variable that approximates multi-dimensional resource allocations to ensure consistency between the decision variable in our optimization problem and the treatment variable in our prediction problem. Specifically, let $P$ denote total incident personnel for a given fire on a given day, obtained from the data; we define the IHC-equivalent variable as $P_{it}/\widehat\kappa$, where $\widehat\kappa$ is the estimated equivalent of an IHC crew (\(\widehat\kappa=63\), as described in~\ref{app:personnel_scaling}). Our optimization model will determine the number $x$ of IHC-equivalent crews (referred to as ``crews’’ thereafter), which corresponds to a personnel level of $\widehat\kappa x$ in the data.

We utilize data on the location coordinates of IHC bases and IHC rosters from the USDA Forest Service. Wildfires in the contiguous United States are managed regionally by nine Geographic Area Coordination Centers (GACCs), with up to 27 IHCs per GACC. When GACC resources are strained, the National Interagency Coordination Center (NICC) can mobilize resources across GACCs \citep{thompson_how_2022}. We therefore define instances with up to 30 IHCs to guide intra-GACC resource allocation and instances with over 100 IHCs to support inter-GACC coordination.

Using wildfire and IHC data, we build an experimental setup that replicates historical suppression challenges. Figure~\ref{fig:fire_ihc_map} illustrates the geographic distribution of the 3,578 wildfires in our data and the 112 locations of IHC bases. As expected, IHC bases are spread across the contiguous United States, with a higher concentration in the Western States that are prone to intense wildfire activity.

\begin{figure}[h!]
    \centering
    \includegraphics[width=.8\textwidth]{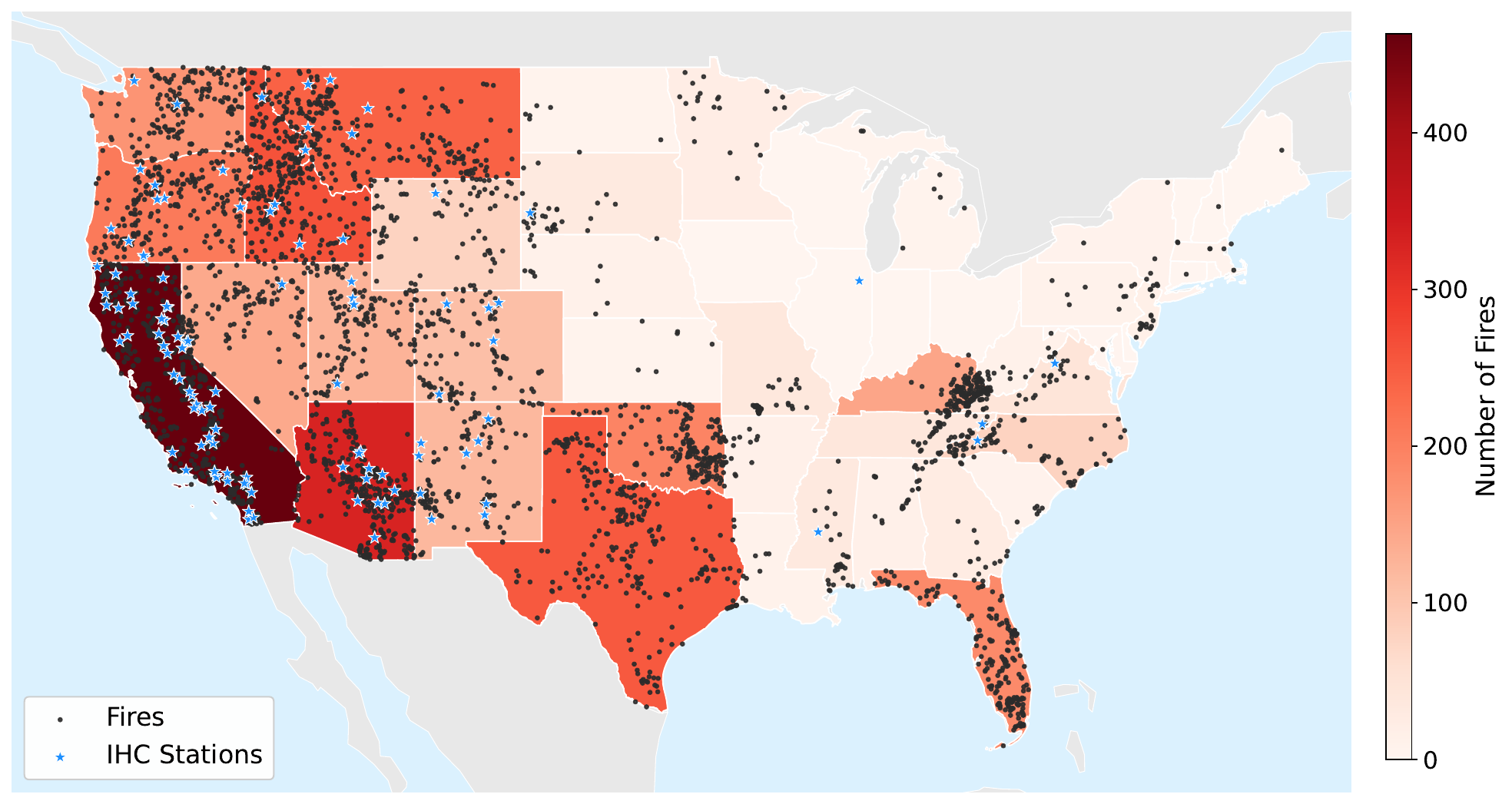}
    \caption{Geographic distribution of wildfires (2015-2024) and IHC stations across the contiguous United States.}
    \label{fig:fire_ihc_map}
\end{figure}

\section{Optimization Model Formulation}
\label{sec:model}

\subsection{Problem Statement and Structure}
\label{subsec:statement}

The problem optimizes wildfire suppression over a discretized time horizon starting in period 1, evolving through $T$ periods indexed by $t\in\calT=\{1,\cdots,T\}$, and ending in a sink period $T+1$. It takes as inputs a set $\calG$ of $G$ fires and a set $\calJ$ of $J$ IHC crews. Each crew $j\in\calJ$ is characterized by its initial location $\ell^1_j$ and the set $\calL_j$ of locations it can visit, including its home base $b_j$ and the set of fires $\calG_j$ within its jurisdiction. We denote by $tt(\ell,\ell')$ and $tc(\ell,\ell')$ the travel time and travel cost between $\ell\in\calL_j$ and $\ell'\in\calL_j$. We assume that each crew must rest every $\varphi$ periods for $\gamma$ periods.

The next input is a deterministic fire spread model. Each wildfire $g\in\calG$ constitutes a dynamical system, governed by a discrete-time state variable $S_{gt}\in\calS_g$, an initial state $S^{1}_g$, a decision $x_{gt}$, a transition function $S_{g,t+1}=f_{gt}(S_{gt},x_{gt})$, a discrete-time cost function $d_{gt}(S_{gt},x_{gt})$, and a terminal cost function $h_g(S_{g,T+1})$. The problem optimizes the assignment of each crew $j\in\calJ$ to each fire $g\in\calG$ in each period $t\in\calT$. These decisions will determine the variables $x_{gt}$ in the fire spread model.

By design, we impose no restriction on the transition and cost functions to capture generic (non-linear) wildfire models. Similarly, we consider a generic bi-objective formulation that can model fire damage---accrued through the cost functions $d_{gt}(\cdot,\cdot)$ and $h_g(\cdot)$---and travel costs, crew fatigue and operational complexity---accrued through the cost function $tc(\cdot,\cdot)$. In our implementation, we define a wildfire model with a two-dimensional state variable representing area burned (\texttt{area}) and momentum (\texttt{area$\_$diff}); the decision variable represents the number of crews assigned to fire $g$ at time $t\in\calT$; and the transition functions $f_{gt}(\cdot,\cdot)$ are learned from data in Section~\ref{sec:causalML}. We define a cost function $h_g(\cdot)$ that captures the total area burned by the end of the time horizon. Our implementation does not penalize intermediate fire states, (i.e., $d_{gt}(\cdot, \cdot) = 0$ for all $g \in \calG$ and $t \in \calT$) and does not include explicit crew travel costs (i.e., $tc(\cdot, \cdot) = 0$). Nevertheless, travel times still play a critical role to balance high-risk versus easily-accessible wildfires, trading off the triage vs. routing goals in the problem. The model can easily be adjusted to incorporate alternative wildfire propagation models as well as different suppression and prioritization objectives (e.g., impact on crew fatigue, local communities, critical infrastructure, ecologically sensitive habitat).

Altogether, the problem features a routing-scheduling structure with endogenous and non-linear wildfire demand as a function of crew assignments. The problem induces a mixed-integer non-linear formulation with ``big-M'' linking constraints coupling crew assignments to suppression operations. As a dynamic program, the problem involves an exponential state space in the number of fires and time periods, in $\calO(J^{GT})$, and an exponential decision space in the number of crews, fires and time periods, in $\calO(2^{JGT})$. Instead, we propose a time-expanded network representation and a linear set partitioning formulation, along with a branch-and-price-and-cut decomposition.

We highlight two limitations in our setup. First, the problem treats upstream strategic decisions (e.g., crew rostering, positioning) as given, and ignores downstream operational decisions (e.g., how to suppress each fire). In between, the model optimizes tactical response across widespread wildfires over a multi-day horizon. Second, due to the complexities of both the predictive and prescripive problems, we consider deterministic fire progression and ignition. In practice, the methodology is intended to operate in a rolling horizon, by re-solving the model as new information becomes available on ignitions, fire spread, and environmental conditions. As such, the model can support repeated short-term reallocations rather than a single irrevocable plan over the full horizon.

\subsection{Time-expanded network representation}
\label{subsec:network}

To capture the interdependencies between crew supply and fire demand, we represent crew assignments on time-space-rest networks and wildfire suppression operations on time-state networks.

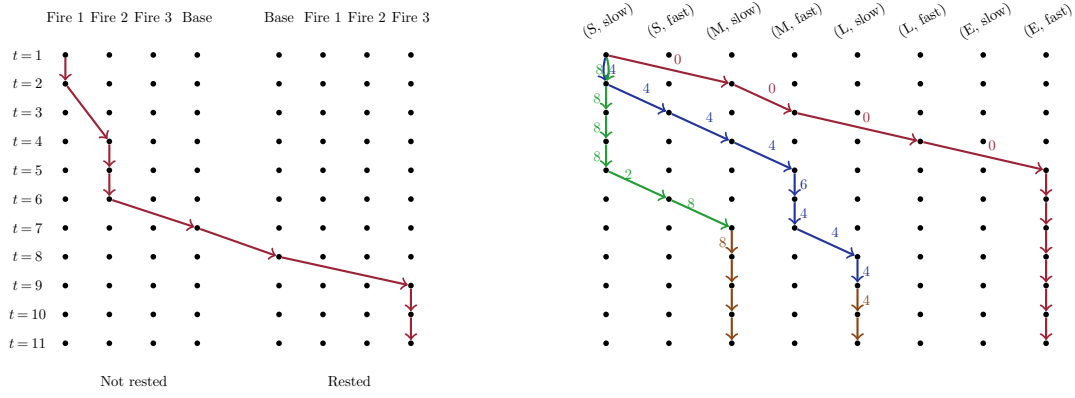
\begin{figure}[h!]
    \centering

    \begin{tikzpicture}[scale=0.5,transform shape,
           roundnode/.style={circle,fill,inner sep=1.5pt}]
            \definecolor{myred}{rgb} {0.618, 0.133, 0.208}
            \definecolor{myblue}{rgb} {0.133, 0.208, 0.618}
            \definecolor{mygreen}{rgb} {0.133, 0.618, 0.208}
            \definecolor{mybrown}{rgb}{0.55, 0.27, 0.07}

            \node[roundnode]        (100)       [] {};
            \node[roundnode]        (110)       [below=.6cm of 100] {};
            \node[roundnode]        (120)       [below=.6cm of 110] {};
            \node[roundnode]        (130)       [below=.6cm of 120] {};
            \node[roundnode]        (140)       [below=.6cm of 130] {};
            \node[roundnode]        (150)       [below=.6cm of 140] {};
            \node[roundnode]        (160)       [below=.6cm of 150] {};
            \node[roundnode]        (170)       [below=.6cm of 160] {};
            \node[roundnode]        (180)       [below=.6cm of 170] {};
            \node[roundnode]        (190)       [below=.6cm of 180] {};
            \node[roundnode]        (1100)       [below=.6cm of 190] {};

            \node[roundnode]        (200)       [right=1cm of 100] {};
            \node[roundnode]        (210)       [below=.6cm of 200] {};
            \node[roundnode]        (220)       [below=.6cm of 210] {};
            \node[roundnode]        (230)       [below=.6cm of 220] {};
            \node[roundnode]        (240)       [below=.6cm of 230] {};
            \node[roundnode]        (250)       [below=.6cm of 240] {};
            \node[roundnode]        (260)       [below=.6cm of 250] {};
            \node[roundnode]        (270)       [below=.6cm of 260] {};
            \node[roundnode]        (280)       [below=.6cm of 270] {};
            \node[roundnode]        (290)       [below=.6cm of 280] {};
            \node[roundnode]        (2100)       [below=.6cm of 290] {};

            \node[roundnode]        (300)       [right=1cm of 200] {};
            \node[roundnode]        (310)       [below=.6cm of 300] {};
            \node[roundnode]        (320)       [below=.6cm of 310] {};
            \node[roundnode]        (330)       [below=.6cm of 320] {};
            \node[roundnode]        (340)       [below=.6cm of 330] {};
            \node[roundnode]        (350)       [below=.6cm of 340] {};
            \node[roundnode]        (360)       [below=.6cm of 350] {};
            \node[roundnode]        (370)       [below=.6cm of 360] {};
            \node[roundnode]        (380)       [below=.6cm of 370] {};
            \node[roundnode]        (390)       [below=.6cm of 380] {};
            \node[roundnode]        (3100)       [below=.6cm of 390] {};

            \node[roundnode]        (400)       [right=1cm of 300] {};
            \node[roundnode]        (410)       [below=.6cm of 400] {};
            \node[roundnode]        (420)       [below=.6cm of 410] {};
            \node[roundnode]        (430)       [below=.6cm of 420] {};
            \node[roundnode]        (440)       [below=.6cm of 430] {};
            \node[roundnode]        (450)       [below=.6cm of 440] {};
            \node[roundnode]        (460)       [below=.6cm of 450] {};
            \node[roundnode]        (470)       [below=.6cm of 460] {};
            \node[roundnode]        (480)       [below=.6cm of 470] {};
            \node[roundnode]        (490)       [below=.6cm of 480] {};
            \node[roundnode]        (4100)       [below=.6cm of 490] {};

            \node[roundnode]        (101)       [right=2cm of 400] {};
            \node[roundnode]        (111)       [below=.6cm of 101] {};
            \node[roundnode]        (121)       [below=.6cm of 111] {};
            \node[roundnode]        (131)       [below=.6cm of 121] {};
            \node[roundnode]        (141)       [below=.6cm of 131] {};
            \node[roundnode]        (151)       [below=.6cm of 141] {};
            \node[roundnode]        (161)       [below=.6cm of 151] {};
            \node[roundnode]        (171)       [below=.6cm of 161] {};
            \node[roundnode]        (181)       [below=.6cm of 171] {};
            \node[roundnode]        (191)       [below=.6cm of 181] {};
            \node[roundnode]        (1101)       [below=.6cm of 191] {};

            \node[roundnode]        (201)       [right=1cm of 101] {};
            \node[roundnode]        (211)       [below=.6cm of 201] {};
            \node[roundnode]        (221)       [below=.6cm of 211] {};
            \node[roundnode]        (231)       [below=.6cm of 221] {};
            \node[roundnode]        (241)       [below=.6cm of 231] {};
            \node[roundnode]        (251)       [below=.6cm of 241] {};
            \node[roundnode]        (261)       [below=.6cm of 251] {};
            \node[roundnode]        (271)       [below=.6cm of 261] {};
            \node[roundnode]        (281)       [below=.6cm of 271] {};
            \node[roundnode]        (291)       [below=.6cm of 281] {};
            \node[roundnode]        (2101)       [below=.6cm of 291] {};

            \node[roundnode]        (301)       [right=1cm of 201] {};
            \node[roundnode]        (311)       [below=.6cm of 301] {};
            \node[roundnode]        (321)       [below=.6cm of 311] {};
            \node[roundnode]        (331)       [below=.6cm of 321] {};
            \node[roundnode]        (341)       [below=.6cm of 331] {};
            \node[roundnode]        (351)       [below=.6cm of 341] {};
            \node[roundnode]        (361)       [below=.6cm of 351] {};
            \node[roundnode]        (371)       [below=.6cm of 361] {};
            \node[roundnode]        (381)       [below=.6cm of 371] {};
            \node[roundnode]        (391)       [below=.6cm of 381] {};
            \node[roundnode]        (3101)       [below=.6cm of 391] {};

            \node[roundnode]        (401)       [right=1cm of 301] {};
            \node[roundnode]        (411)       [below=.6cm of 401] {};
            \node[roundnode]        (421)       [below=.6cm of 411] {};
            \node[roundnode]        (431)       [below=.6cm of 421] {};
            \node[roundnode]        (441)       [below=.6cm of 431] {};
            \node[roundnode]        (451)       [below=.6cm of 441] {};
            \node[roundnode]        (461)       [below=.6cm of 451] {};
            \node[roundnode]        (471)       [below=.6cm of 461] {};
            \node[roundnode]        (481)       [below=.6cm of 471] {};
            \node[roundnode]        (491)       [below=.6cm of 481] {};
            \node[roundnode]        (4101)       [below=.6cm of 491] {};

            \node[yshift=1cm] at (100) {Fire 1};
            \node[yshift=1cm] at (200) {Fire 2};
            \node[yshift=1cm] at (300) {Fire 3};
            \node[yshift=1cm] at (400) {Base};

            \node[yshift=1cm] at (201) {Fire 1};
            \node[yshift=1cm] at (301) {Fire 2};
            \node[yshift=1cm] at (401) {Fire 3};
            \node[yshift=1cm] at (101) {Base};
            \node[yshift=-1cm, xshift=0.7cm] at (2101) {Rested};
            \node[yshift=-1cm, xshift=0.65cm] at (2100) {Not rested};

            \node[xshift=-1cm] at (100) {$t=1$};
            \node[xshift=-1cm] at (110) {$t=2$};
            \node[xshift=-1cm] at (120) {$t=3$};
            \node[xshift=-1cm] at (130) {$t=4$};
            \node[xshift=-1cm] at (140) {$t=5$};
            \node[xshift=-1cm] at (150) {$t=6$};
            \node[xshift=-1cm] at (160) {$t=7$};
            \node[xshift=-1cm] at (170) {$t=8$};
            \node[xshift=-1cm] at (180) {$t=9$};
            \node[xshift=-1cm] at (190) {$t=10$};
            \node[xshift=-1cm] at (1100) {$t=11$};

            \draw[myred, thick, ->] (100) -- (110);
            \draw[myred, thick, ->] (110) -- (230);
            \draw[myred, thick, ->] (230) -- (240);
            \draw[myred, thick, ->] (240) -- (250);
            \draw[myred, thick, ->] (250) -- (460);
            \draw[myred, thick, ->] (460) -- (171);
            \draw[myred, thick, ->] (171) -- (481);
            \draw[myred, thick, ->] (481) -- (491);
            \draw[myred, thick, ->] (491) -- (4101);


            \node[roundnode] (s1t1)  [right=5cm of 401] {};
            \node[roundnode] (s2t1)  [right=1.5cm of s1t1] {};
            \node[roundnode] (s3t1)  [right=1.5cm of s2t1] {};
            \node[roundnode] (s4t1)  [right=1.5cm of s3t1] {};
            \node[roundnode] (s5t1)  [right=1.5cm of s4t1] {};
            \node[roundnode] (s6t1)  [right=1.5cm of s5t1] {};
            \node[roundnode] (s7t1)  [right=1.5cm of s6t1] {};
            \node[roundnode] (s8t1)  [right=1.5cm of s7t1] {};
            \node[roundnode] (s1t2)  [below=.6cm of s1t1] {};
            \node[roundnode] (s2t2)  [below=.6cm of s2t1] {};
            \node[roundnode] (s3t2)  [below=.6cm of s3t1] {};
            \node[roundnode] (s4t2)  [below=.6cm of s4t1] {};
            \node[roundnode] (s5t2)  [below=.6cm of s5t1] {};
            \node[roundnode] (s6t2)  [below=.6cm of s6t1] {};
            \node[roundnode] (s7t2)  [below=.6cm of s7t1] {};
            \node[roundnode] (s8t2)  [below=.6cm of s8t1] {};
            \node[roundnode] (s1t3)  [below=.6cm of s1t2] {};
            \node[roundnode] (s2t3)  [below=.6cm of s2t2] {};
            \node[roundnode] (s3t3)  [below=.6cm of s3t2] {};
            \node[roundnode] (s4t3)  [below=.6cm of s4t2] {};
            \node[roundnode] (s5t3)  [below=.6cm of s5t2] {};
            \node[roundnode] (s6t3)  [below=.6cm of s6t2] {};
            \node[roundnode] (s7t3)  [below=.6cm of s7t2] {};
            \node[roundnode] (s8t3)  [below=.6cm of s8t2] {};
            \node[roundnode] (s1t4)  [below=.6cm of s1t3] {};
            \node[roundnode] (s2t4)  [below=.6cm of s2t3] {};
            \node[roundnode] (s3t4)  [below=.6cm of s3t3] {};
            \node[roundnode] (s4t4)  [below=.6cm of s4t3] {};
            \node[roundnode] (s5t4)  [below=.6cm of s5t3] {};
            \node[roundnode] (s6t4)  [below=.6cm of s6t3] {};
            \node[roundnode] (s7t4)  [below=.6cm of s7t3] {};
            \node[roundnode] (s8t4)  [below=.6cm of s8t3] {};
            \node[roundnode] (s1t5)  [below=.6cm of s1t4] {};
            \node[roundnode] (s2t5)  [below=.6cm of s2t4] {};
            \node[roundnode] (s3t5)  [below=.6cm of s3t4] {};
            \node[roundnode] (s4t5)  [below=.6cm of s4t4] {};
            \node[roundnode] (s5t5)  [below=.6cm of s5t4] {};
            \node[roundnode] (s6t5)  [below=.6cm of s6t4] {};
            \node[roundnode] (s7t5)  [below=.6cm of s7t4] {};
            \node[roundnode] (s8t5)  [below=.6cm of s8t4] {};
            \node[roundnode] (s1t6)  [below=.6cm of s1t5] {};
            \node[roundnode] (s2t6)  [below=.6cm of s2t5] {};
            \node[roundnode] (s3t6)  [below=.6cm of s3t5] {};
            \node[roundnode] (s4t6)  [below=.6cm of s4t5] {};
            \node[roundnode] (s5t6)  [below=.6cm of s5t5] {};
            \node[roundnode] (s6t6)  [below=.6cm of s6t5] {};
            \node[roundnode] (s7t6)  [below=.6cm of s7t5] {};
            \node[roundnode] (s8t6)  [below=.6cm of s8t5] {};
            \node[roundnode] (s1t7)  [below=.6cm of s1t6] {};
            \node[roundnode] (s2t7)  [below=.6cm of s2t6] {};
            \node[roundnode] (s3t7)  [below=.6cm of s3t6] {};
            \node[roundnode] (s4t7)  [below=.6cm of s4t6] {};
            \node[roundnode] (s5t7)  [below=.6cm of s5t6] {};
            \node[roundnode] (s6t7)  [below=.6cm of s6t6] {};
            \node[roundnode] (s7t7)  [below=.6cm of s7t6] {};
            \node[roundnode] (s8t7)  [below=.6cm of s8t6] {};
            \node[roundnode] (s1t8)  [below=.6cm of s1t7] {};
            \node[roundnode] (s2t8)  [below=.6cm of s2t7] {};
            \node[roundnode] (s3t8)  [below=.6cm of s3t7] {};
            \node[roundnode] (s4t8)  [below=.6cm of s4t7] {};
            \node[roundnode] (s5t8)  [below=.6cm of s5t7] {};
            \node[roundnode] (s6t8)  [below=.6cm of s6t7] {};
            \node[roundnode] (s7t8)  [below=.6cm of s7t7] {};
            \node[roundnode] (s8t8)  [below=.6cm of s8t7] {};
            \node[roundnode] (s1t9)  [below=.6cm of s1t8] {};
            \node[roundnode] (s2t9)  [below=.6cm of s2t8] {};
            \node[roundnode] (s3t9)  [below=.6cm of s3t8] {};
            \node[roundnode] (s4t9)  [below=.6cm of s4t8] {};
            \node[roundnode] (s5t9)  [below=.6cm of s5t8] {};
            \node[roundnode] (s6t9)  [below=.6cm of s6t8] {};
            \node[roundnode] (s7t9)  [below=.6cm of s7t8] {};
            \node[roundnode] (s8t9)  [below=.6cm of s8t8] {};
            \node[roundnode] (s1t10) [below=.6cm of s1t9] {};
            \node[roundnode] (s2t10) [below=.6cm of s2t9] {};
            \node[roundnode] (s3t10) [below=.6cm of s3t9] {};
            \node[roundnode] (s4t10) [below=.6cm of s4t9] {};
            \node[roundnode] (s5t10) [below=.6cm of s5t9] {};
            \node[roundnode] (s6t10) [below=.6cm of s6t9] {};
            \node[roundnode] (s7t10) [below=.6cm of s7t9] {};
            \node[roundnode] (s8t10) [below=.6cm of s8t9] {};
            \node[roundnode] (s1t11) [below=.6cm of s1t10] {};
            \node[roundnode] (s2t11) [below=.6cm of s2t10] {};
            \node[roundnode] (s3t11) [below=.6cm of s3t10] {};
            \node[roundnode] (s4t11) [below=.6cm of s4t10] {};
            \node[roundnode] (s5t11) [below=.6cm of s5t10] {};
            \node[roundnode] (s6t11) [below=.6cm of s6t10] {};
            \node[roundnode] (s7t11) [below=.6cm of s7t10] {};
            \node[roundnode] (s8t11) [below=.6cm of s8t10] {};

            \node[rotate=30, yshift=0.75cm, xshift=0.5cm] at (s1t1) {(S, slow)};
            \node[rotate=30, yshift=0.75cm, xshift=0.5cm] at (s2t1) {(S, fast)};
            \node[rotate=30, yshift=0.75cm, xshift=0.5cm] at (s3t1) {(M, slow)};
            \node[rotate=30, yshift=0.75cm, xshift=0.5cm] at (s4t1) {(M, fast)};
            \node[rotate=30, yshift=0.75cm, xshift=0.5cm] at (s5t1) {(L, slow)};
            \node[rotate=30, yshift=0.75cm, xshift=0.5cm] at (s6t1) {(L, fast)};
            \node[rotate=30, yshift=0.75cm, xshift=0.5cm] at (s7t1) {(E, slow)};
            \node[rotate=30, yshift=0.75cm, xshift=0.5cm] at (s8t1) {(E, fast)};

            \draw[myred, thick, ->] (s1t1) -- (s3t2) node[midway, above right] {0};
            \draw[myred, thick, ->] (s3t2) -- (s4t3) node[midway, above right] {0};
            \draw[myred, thick, ->] (s4t3) -- (s6t4) node[midway, above right] {0};
            \draw[myred, thick, ->] (s6t4) -- (s8t5) node[midway, above right] {0};
            \draw[myred, thick, ->] (s8t5)  -- (s8t6);
            \draw[myred, thick, ->] (s8t6)  -- (s8t7);
            \draw[myred, thick, ->] (s8t7)  -- (s8t8);
            \draw[myred, thick, ->] (s8t8)  -- (s8t9);
            \draw[myred, thick, ->] (s8t9)  -- (s8t10);
            \draw[myred, thick, ->] (s8t10) -- (s8t11);

            \draw[myblue, thick, ->, bend right=15] (s1t1) to node[midway, right] {4} (s1t2);
            \draw[myblue, thick, ->] (s1t2) -- (s2t3) node[midway, above right] {4};
            \draw[myblue, thick, ->] (s2t3) -- (s3t4) node[midway, above right] {4};
            \draw[myblue, thick, ->] (s3t4) -- (s4t5) node[midway, above right] {4};
            \draw[myblue, thick, ->] (s4t5) -- (s4t6) node[midway, right] {6};
            \draw[myblue, thick, ->] (s4t6) -- (s4t7) node[midway, right] {4};
            \draw[myblue, thick, ->] (s4t7) -- (s5t8) node[midway, above right] {4};
            \draw[myblue, thick, ->] (s5t8) -- (s5t9) node[midway, right] {4};
            \draw[mybrown, thick, ->] (s5t9)  -- (s5t10) node[midway, right] {4};
            \draw[mybrown, thick, ->] (s5t10) -- (s5t11);

            \draw[mygreen, thick, ->, bend left=15] (s1t1) to node[midway, left] {8} (s1t2);
            \draw[mygreen, thick, ->] (s1t2) -- (s1t3) node[midway, left] {8};
            \draw[mygreen, thick, ->] (s1t3) -- (s1t4) node[midway, left] {8};
            \draw[mygreen, thick, ->] (s1t4) -- (s1t5) node[midway, left] {8};
            \draw[mygreen, thick, ->] (s1t5) -- (s2t6) node[midway, above left] {2};
            \draw[mygreen, thick, ->] (s2t6) -- (s3t7) node[midway, above left] {8};
            \draw[mybrown, thick, ->] (s3t7)  -- (s3t8) node[midway, left] {8};
            \draw[mybrown, thick, ->] (s3t8)  -- (s3t9);
            \draw[mybrown, thick, ->] (s3t9)  -- (s3t10);
            \draw[mybrown, thick, ->] (s3t10) -- (s3t11);

    \end{tikzpicture}
    \caption{Crew route on a time-space-rest network (left). Fire suppression trajectories on a time-state network (right), where columns represent (area, momentum) state pairs, arc labels indicate the number of crews assigned, and brown arcs indicate a suppressed fire, with $T=10$.}
    \label{fig:ts_networks}
    \vspace{-12pt}
\end{figure}

\subsubsection*{Time-space-rest network for crew assignments.}\label{section:crew_aasignments}

We create time-space-rest networks to track crew locations and rest status. Time-space networks are common in routing and scheduling \citep{liebchen2008first,agarwal2008ship,lee2020dynamic,jacquillat2022optimizing}; we augment them here to keep track of crews' rest requirements. For example, Figure~\ref{fig:ts_networks} shows a sample path where the crew suppresses fire $1$, travels to suppress fire $2$, rests at base, and travels to suppress fire $3$.

We assume that the planning horizon is shorter than the resting horizon (Assumption~\ref{assumption:rest}). Remark~\ref{rem:rest} clarifies that our model can relax this assumption, albeit at additional costs in terms of model size and tractability (see~\ref{app:rest}). Still, this assumption is reasonable given the 14-day resting horizon and medium-term horizon in our optimization setting.

\begin{assumption}\label{assumption:rest}
    The planning horizon is shorter than the minimum rest frequency: $T < \varphi + \gamma$.
\end{assumption}

\begin{remark}\label{rem:rest}
   Travel costs, suppression costs, and rest requirements can be captured in a time-space-rest network representation with $1 + |\calL_j|\times T\times(\varphi+1)$ nodes. Under Assumption~\ref{assumption:rest}, there exists an equivalent representation in a network with $1 + |\calL_j|\times T\times 2$ nodes.
\end{remark}

Under Assumption~\ref{assumption:rest}, we define the time-space-rest network as follows. If crew $j\in\calJ$ last completed a rest at time $r_j^0$, then the latest possible rest time is $R_j=\min(\varphi + r_j^0, T)$. A set of nodes $\calN_j = (\cal L_j \times \{2,\cdots,T+1\}\times\{0,1\}) \,\cup\, (\ell^1_j,1,0)$, indexed by $n$, tracks the location $\lambda(n)\in\cal L_j$, the time stamp $\tau(n)\in\calT\cup\{T+1\}$, and a binary indicator $\nu(n)$ indicating whether the crew has rested. Node $(\ell^1_j,1,0)$ defines the starting node. A set of directed arcs $\calA_j\subseteq \calN_j\times\calN_j$ stores the transitions and associated costs (again, these crew costs are zero in our implementation).
\begin{itemize}
    \item[--] \textit{Traveling arcs:} For all $\ell_1\neq \ell_2 \in \calL_j$ and $r\in \{0, 1\}$, an arc from $n_1=(\ell_1, t_1, r)$ to $n_2=(\ell_2, t_2, r)$ is in $\calA_j$ if a crew at location $\ell_1$ in period $t_1$ can reach $\ell_2$ by $t_2$, i.e., if $\lceil tt(\ell_1,\ell_2)\rceil_{\calT}=\tau(n_2)-\tau(n_1)$, where $\lceil t\rceil_{\calT}=\min\{\tau\in\calT:\tau\geq t\}$ denotes a ``ceiling'' function given the time discretization $\calT$. It is associated with a travel cost $c_a=tc(\ell_1,\ell_2)$ (again, this cost is zero in our implementation).
    \item[--] \textit{Working arcs:} For all fires $g\in\calG_j$, time periods $t\in\{1,\cdots,T\}$, and $r\in \{0, 1\}$, an arc linking $(g, t, r)$ to $(g, t+1, r)$ is in $\calA_j$ to characterize suppression activities. The corresponding cost would encapsulate suppression cost and crew fatigue.
    \item[--] \textit{Resting arcs:} An arc from $n_1=(b_j,t,0)$ to $n_2=(b_j,t',1)$ indicates rest if $t'- t \geq \gamma$, and an arc linking $(b_j, t, r)$ to $(b_j, t+1, r)$ indicates shorter time at base. We enforce rest requirements by removing traveling and working arcs from nodes $n\in\calN_j$ such that $\tau(n)>R_j$ and $\nu(n)=0$.
\end{itemize}

For all $n\in \calN_j$, we denote by $\delta^-(n)$ and $\delta^+(n)$ the set of incoming and outgoing arcs. We define $\calA_j^{gt}=\delta^+((g,t,0))\,\cup\,\delta^+((g,t,1))$ as the set of arcs working on fire $g\in\calG_j$ at time $t\in\calT$.

\subsubsection*{Time-state network for wildfire suppression.}

We create time-state networks to track the state of each fire in each time period as a function of crew assignments. Figure~\ref{fig:ts_networks} shows sample paths with intense, moderate, and delayed suppression. This representation provides a flexible methodology for optimizing wildfire suppression operations and linking them to crew assignments. This representation can be exact if it involves $\calO(J^{GT})$ states; due to the curse of dimensionality, we resort to approximate discretization (we use a grid-based discretization in our implementation).

Formally, the time-state network is defined by a set of nodes $\calN_g = (\calS_g \times \{2,\cdots,T+1\}) \,\cup\, (s^1_g,1)$ and a set of arcs $\calB_g \subseteq \calN_g\times\calN_g$. Each node $n$ encodes the state of the fire at time $\tau(n)$, and $(s^1_g,1)$ defines the initial state. To capture the transition function, we define an arc $a\in\calB_g$ from $(s_1, t)$ to $(s_2, t+1)$ and $x_a$ as the minimum number of crews required to traverse arc $a$, i.e., $\lceil f_{g,t}(s_1,x_a)\rceil_{\calS_g}=s_2$, where $\lceil s\rceil_{\calS_g}=\min\{\sigma\in\calS_g:\sigma\geq s\}$ denotes a ``ceiling'' function given the state discretization $\calS_g$. To capture the cost function, we associate a cost $d_a$ to each arc $a\in\calB_g$ from $(s_1, t)$ to $(s_2, t+1)$ defined such that $d_a=d_{g,t}(s_1,x_a)$ if $t<T$ and such that $d_a=d_{g,t}(s_1,x_a)+h_g(s_2)$ if $t=T$. As specified above, only the terminal cost is applied in our implementation. By construction, the network contains all the information from the dynamical system associated with fire $g\in\calG$.

Finally, we denote by $\delta^-(n)$ and $\delta^+(n)$ the set of incoming and outgoing arcs in node $n$, and by $\calB_g^t=\{a=(n_1,n_2)\in\calB_g:x_a>0,\tau(n_1)=t\}$ the set of arcs that require crew presence at time $t$.

\subsection{Two-sided set partitioning formulation}
\label{subsec:SP}

We leverage the time-expanded network representation to formulate the wildfire suppression problem. We define in~\ref{app:arc} a natural benchmark that uses arc-based variables on these networks. This model, however, involves a large number of binary variables and linking constraints.

Instead, we propose a decomposition approach that optimizes over complete paths for a fire or a crew along the corresponding network. Define the set of routes $\calP_j$ for each crew $j\in\calJ$ as the set of paths in the time-space-rest network starting in $(\ell^1_j,1,0)$ and satisfying flow balance and rest requirements thereafter. Similarly, define the set of suppression plans $\calQ_g$ for each fire $g\in\calG$ as the set of paths in the time-state network starting in $(s^1_g,1)$ and satisfying flow balance thereafter. For crew route $p\in\calP_j$, $\calA^p_j\subseteq\calA_j$ stores the set of traversed arcs, $A_{pgt}=\mathbf{1}(\calA^p_j\cap\A_{j}^{gt}\neq\emptyset)$ indicates whether the crew is assigned to fire $g\in\calG$ at time $t\in\calT$ (referred to as \textit{crew assignments}), and $c_p = \sum_{a\in \calA^p_j}c_a$ denotes its cost. For fire plan $q\in\calQ_g$, $\calB^q_g\subseteq\calB_g$ stores the set of traversed arcs, $B_{qt}=\sum_{a\in\calB^q_g\cap\calB^t_g}x_a$ indicates the number of crews required by suppression plan \(q\) for fire $g\in\calG$ at time $t\in\calT$, (referred to as \textit{fire demand}) and $d_q = \sum_{a\in \calB^q_g}d_a$ denotes its cost. We define the following decision variables:
\begin{align*}
    z_p&=\begin{cases}
        1&\text{if route $p\in\calP_j$ is selected for crew $j\in\calJ$}\\
        0&\text{otherwise.}
    \end{cases}\\
    y_q&=\begin{cases}
        1&\text{if suppression plan $q\in\calQ_g$ is selected for fire $g\in\calG$}\\
        0&\text{otherwise.}
    \end{cases}
\end{align*}

The model minimizes the total cost of crew operations and fire damage (Equation~\eqref{obj2}). Constraints~\eqref{supp_plan}--~\eqref{crew_route} ensure that each fire and crew is assigned to one suppression plan and route, respectively. The linking constraints~\eqref{link set partition} ensure that crew supply meets fire demand.
\begin{mini!}
    {}{\sum\limits_{j \in \calJ}\sum\limits_{p \in \calP_j} c_pz_{p} + \sum\limits_{g \in \calG}\sum\limits_{q \in \calQ_g} d_qy_q }{\label{set partition}}{} \label{obj2}
    \addConstraint{\sum\limits_{q \in \calQ_g}y_q = 1,\,\forall g\in \calG} \label{supp_plan}
    \addConstraint{\sum\limits_{p \in \calP_j}z_{p} = 1,\,\forall j\in \calJ}  \label{crew_route}
    \addConstraint{\sum\limits_{j\in \calJ} \sum\limits_{p \in \calP_j} A_{pgt}z_{p}  -\sum\limits_{q \in \calQ_g} B_{qt}y_q  \ge 0, \,\forall g\in \calG,\forall t\in\calT} \label{link set partition}
    \addConstraint{z_{p}  \in \{0,1\},\forall{p\in \calP_j},\ \forall{j\in \calJ}} \label{crew integrality constr}
    \addConstraint{y_q  \in \{0,1\},\forall{q\in \calQ_g}},\ \forall{g\in \calG} \label{fire integrality constr}
\end{mini!}

This approach applies set partitioning principles for each fire and each crew (Constraints~\eqref{supp_plan}--\eqref{crew_route}), resulting in a ``two-sided'' set partitioning formulation with the fire-crew linking constraints (Constraints~\eqref{link set partition}). The challenge lies in the exponential growth in the number of path-based variables with the number of fires $G$, the number of crews $J$, and the time horizon $T$.

\section{Two-Sided Branch-and-Price-and-Cut Algorithm}
\label{sec:algorithm}

In response, we propose a two-sided branch-and-price-and-cut algorithm to solve the problem. Throughout, we leverage the structure of the linking constraints~\eqref{link set partition} to develop an efficient methodology via dual stabilization, cutting planes, and variable branching.

\subsection{Two-sided Column Generation}
\label{section:CG}

\subsubsection*{Core algorithm.}

The algorithm iterates between a restricted master problem ($\rmp(\calQ_g', \calP_j')$) that solves the linear relaxation based on subsets of routes $\calP_j' \subseteq \calP_j$ and of suppression plans $\calQ_g' \subseteq \calQ_g$, and a pricing problem that adds new crew routes and fire suppression plans with negative reduced costs or proves that none exists. Crew routes and fire suppression plans are connected via the linking fire-crew constraint in the master problem and the corresponding dual costs in the pricing problem. The master problem is formulated as follows:
\begin{mini!}
    {}{\sum\limits_{j \in \calJ}\sum\limits_{p \in \calP'_j} c_pz_{p} + \sum\limits_{g \in \calG}\sum\limits_{q \in \calQ'_g} d_qy_q }{\label{RMP}}{} \label{obj2MP}
    \addConstraint{\sum\limits_{q \in \calQ'_g}y_q = 1,\,\forall g\in \calG} \label{supp_planMP}
    \addConstraint{\sum\limits_{p \in \calP'_j}z_{p} = 1,\,\forall j\in \calJ}  \label{crew_routeMP}
    \addConstraint{\sum\limits_{j\in \calJ} \sum\limits_{p \in \calP'_j} A_{pgt}z_{p}  -\sum\limits_{q \in \calQ'_g} B_{qt}y_q  \ge 0, \,\forall g\in \calG,\forall t\in\calT} \label{link set partition MP}
    \addConstraint{z_{p}\geq0,\forall{p\in \calP'_j,\ \forall j\in\calJ}} \label{crew nonneg constrMP}
    \addConstraint{y_q \geq0,\forall{q\in \calQ'_g}},\ \forall{g\in \calG} \label{fire nonneg constrMP}
\end{mini!}

Let $(\by, \bz)$ denote this problem's primal solution, and $\bsigma\in\R^{\calG}, \bpi\in\R^{\calJ}, \brho\in\R_+^{\calG\times\calT}$ the dual variables associated with Constraints~\eqref{supp_planMP},~\eqref{crew_routeMP} and~\eqref{link set partition MP}, respectively. The pricing problems seek a plan of minimal reduced cost for each fire $g\in\calG$ and a route of minimal reduced cost for each crew $j\in\calJ$:
\begin{align} 
    \text{for all fires $g\in\calG$:}\ &\min_{q\in \calQ_g} \left(d_{q} + \sum_{t\in \calT} B_{qt}\rho_{gt}-\sigma_g\right)\label{sp:fire}\\
    \text{for all crews $j\in\calJ$:}\ &\min_{p\in \calP_j} \left( c_{p} - \sum_{g\in \calG}\sum_{t\in \calT} A_{pgt}\rho_{gt}-\pi_j\right).\label{sp:crew}
\end{align}

The fire subproblem minimizes the damage cost $d_{q}$ plus the cost of bringing crews to fire $g\in\calG$, driven by the dual variables $\rho_{gt}$; the plan is added if these terms are less than the dual value $\sigma_g$ of suppressing fire $g$. The crew subproblem minimizes the routing cost $c_p$ but rewards suppression activities, again driven by the dual variables $\rho_{gt}$; the route is added if these terms are less than the dual cost $\pi_j$ of using crew $j$. Thus, the dual variables of the linking constraints incorporate crew-based information into the fire subproblems and fire-based information into the crew subproblems.

The two-sided column generation procedure is described in Algorithm~\ref{alg:2CG_verbose}. Each iteration solves the master problem, one subproblem per fire, and one subproblem per crew. If all reduced costs are non-negative, the algorithm terminates. Otherwise, the restricted sets of fire suppression plans and crew routes are expanded, and the algorithm continues. By design, the algorithm terminates in a finite number of iterations with an optimal solution of the linear relaxation of Problem~\eqref{set partition}.

\begin{algorithm} [h!]
\caption{\textsc{TwoSidedColumnGeneration}.}\small
\label{alg:2CG_verbose}
\begin{algorithmic}
\item \textbf{Initialization:} Cost $c\gets\infty$; sets $\calQ' \gets \emptyset$, $\calP' \gets \emptyset$; duals $\boldsymbol \sigma \gets M'\cdot \boldsymbol{1}$, $\boldsymbol \pi \gets \boldsymbol{0}$, and $\boldsymbol \rho \gets M\cdot \boldsymbol{1}$, for large $M,M'$.
\item Iterate over Steps 1--3.
\begin{itemize}
    \item[] \textbf{Step 1.} Solve Problem~\eqref{sp:fire} for each fire $g\in\calG$, and let $q^*_g$ be its optimal solution. Solve Problem~\eqref{sp:crew} for each crew $j\in\calJ$, and let $p^*_j$ be its optimal solution.
    \item[] \textbf{Step 2.}  If all pricing problems return solutions with non-negative reduced cost, $\texttt{STOP}$: if $c < \infty$, return $(c, \by, \bz)$; otherwise return $\texttt{INFEASIBLE}$. Otherwise, expand $\calP'_j\gets\calP'_j\cup\{p^*_j\}$ and $\calQ'_g\gets\calQ'_g\cup\{q^*_g\}$ for each crew $j\in\calJ$ and each fire $g\in\calG$ such that the reduced cost of $p^*_j$ or $q^*_g$ is negative.
    \item[] \textbf{Step 3.} Solve $\rmp(\calQ_g', \calP_j')$. If there is a feasible solution $(\by^*, \bz^*)$ with cost $c^*$, update $(\by, \bz)\gets(\by^*, \bz^*)$ and $c\gets c^*$; derive optimal dual solution $(\boldsymbol \sigma^*, \boldsymbol \pi^*, \boldsymbol \rho^*)$ and update $\boldsymbol \sigma \gets \boldsymbol \sigma^*$, $\boldsymbol \pi \gets \boldsymbol \pi^*$, and $\boldsymbol \rho \gets \boldsymbol \rho^*$. Else, derive dual extreme ray $(\boldsymbol \sigma^*, \boldsymbol \pi^*, \boldsymbol \rho^*)$ and update $\boldsymbol \sigma \gets M'\cdot\frac{\boldsymbol\sigma^*}{||\boldsymbol\sigma^*||_1}$, $\boldsymbol \pi \gets M'\cdot\frac{\boldsymbol\pi^*}{||\boldsymbol\pi^*||_1}$ and $\boldsymbol \rho \gets M'\cdot\frac{\boldsymbol\rho^*}{||\boldsymbol\rho^*||_1}$.
\end{itemize}
\end{algorithmic}
\end{algorithm}

From a discrete optimization perspective, the column generation algorithm separates the combinatorial complexities via mixed-integer linear optimization in the master problem from the non-linear system dynamics via dynamic programming in the fire-pricing problems. From a dynamic programming perspective, the decomposition scheme quells the rate of exponential growth: whereas the full problem would involve $\calO(J^{GT})$ state variables and $\calO(2^{JGT})$ decision variables, each fire pricing problem involves $\calO(J^{T})$ state variables and $\calO(J^T)$ decision variables. Moreover, in both pricing problems, this decomposition approach induces a simple shortest-path structure over the time-expanded networks, which is amenable to efficient solution algorithms.

\subsubsection*{Solving the pricing subproblems.}

Both the crew and fire pricing problems can be formulated via dynamic programming over their respective time-expanded networks. In particular, the fire pricing problem can be formulated as follows, using the cost-to-go function $J_{gt}(\cdot)$:
\begin{align*}
	J_{gt}(S_{gt})		&=	\min_{x_{gt}}\left\{d_{gt}(S_{gt},x_{gt})+x_{gt}\cdot\rho_{gt}+J_{g,t+1}(f_{gt}(S_{gt},x_{gt}))\right\},\ \quad\forall t=1,\cdots,T	\\
	J_{g,T+1}(S_{g,T+1})	&=	h_{g}(S_{g,T+1})-\sigma_g
\end{align*}

In fact, both the crew and fire subproblems can be cast as shortest-path problems in the discretized time-expanded networks, upon increasing the cost of every non-terminal arc $a \in \calB^t_g$ by $x_a\cdot \rho_{gt}$, decreasing the cost of every terminal arc $a \in \calB^T_g$ by $\sigma_g$, decreasing the cost of every non-terminal arc $a\in\calA_j^{gt}$ by $\rho_{gt}$, and subtracting $\pi_j$ from the cost of every terminal arc $a\in\calA_j^{gT}$. Furthermore, the time-expanded networks feature a directed acyclic structure since all arcs travel strictly forward in time; thus the shortest paths can be computed via a label-setting algorithm in linear time in the number of nodes and arcs. In our case, the crew pricing problem terminates in $\calO(G^2\cdot T)$ and the fire pricing problem terminates in $\calO(|\calS_g|^2\cdot T)$ (Figure~\ref{fig:ts_networks}). Note, however, that this property relies on the discretization of the state space, which grows exponentially with the time horizon. We provide details on the algorithms used to solve the pricing problems in~\ref{app:SP}.

\subsubsection*{Acceleration via dual stabilization.}

The degeneracy of set partitioning formulations can oftentimes lead to slow convergence \citep{vanderbeck2005implementing}. We leverage a simple property of fire models for dual stabilization, stating that suppression earlier is no worse than later.

\begin{definition}\label{def:deferral}
    A set of plans $\calQ_g$ is \textit{deferral-proof} if for all $q_1, q_2 \in \calQ_g$ such that $B_{q_1t} = B_{q_2t},\forall t\in\calT\setminus\{t_1, t_2\}$ and $B_{q_1t_1} - B_{q_2t_1}  =  B_{q_2t_2} - B_{q_1t_2} \ge 0$ for some $t_1 < t_2$, then $d_{q_1} \le d_{q_2}$.
\end{definition}

This definition motivates a minor reformulation of the problem, where we relax the linking constraint by allowing deferred suppression. We define decision variables $\delta_{gt}\geq0$ tracking the number of crews that are ``saved'' in fire $g\in\calG$ from period $t\in\calT$ for future use, with $\delta_{g0}=0$.

\begin{proposition} \label{prop:deferral}
    If $\calQ_g$ is deferral-proof for all $g\in\calG$, Problem~\eqref{set partition} is equivalent to minimizing Objective~\eqref{obj2} subject to Constraints~\eqref{supp_plan}, \eqref{crew_route}, \eqref{crew integrality constr}, \eqref{fire integrality constr}, and:
    \begin{equation}\label{link set partition new}
        \sum\limits_{j\in \calJ} \sum\limits_{p \in \calP_j} A_{pgt}z_{p}  -\sum\limits_{q \in \calQ_g} B_{qt}y_q   + \delta_{g,t-1} - \delta_{gt} \ge 0, \,\forall g\in \calG,\forall t\in\calT\quad\text{and}\quad\delta_{gt} \ge 0,\forall g\in\calG,\forall t\in\calT
    \end{equation}
\end{proposition}

By relaxing the primal constraint, this change adds the following dual stabilization constraint.
$$\rho_{g(t+1)} - \rho_{gt} \le 0,\,\forall g\in \calG, \forall t\in \{1,\ldots, T-1\}$$

Intuitively, the monotonicity of the dual prices provides a dual view of the deferral-proof property: crews are never more valuable later than earlier on at a fire. The dual stabilization constraints smooths the solution by avoiding oscillations between low and high numbers of crews for each fire. Whereas deferral-proofness is a natural property satisfied in optimization models of wildfire suppression \citep{donovan2003integer,wei2015chance}, it may not be exactly satisfied in our data-driven model in Section~\ref{sec:causalML}. We can then solve the pricing problem with dual stabilization for acceleration purposes in early iterations, and then relax it to ensure exactness.

\subsection{Cutting planes: augmented GUB inequalities} \label{section:cutting planes}

Our two-sided set partitioning formulation may exhibit a weak linear relaxation due to the non-linear dynamics---namely, the cost of a convex combination of plans may differ from the convex combination of costs. Consider a simple two-period setting with plan $q_1$ assigning 10 crews in period 1 and 0 in period 2 (cost 100); plan $q_2$ assigning no crew in period 1 and 10 crews in period 2 (cost 200); and plan $q_3$ assigning 5 crews in periods 1 and 2 (cost 160). Plan $q_3$ is critical to the integer solution problem but the relaxation may rely on plans $q_1$ and $q_2$ with weights $y_1=y_2=0.5$. This challenge is somewhat uncommon in binary set partitioning formulation, and arises due to the integer structure of wildfire suppression operations in our setting.

We incorporate valid inequalities by adapting, strengthening, and augmenting the generalized-upper-bound (GUB) cover cuts from \cite{wolsey1990valid}. We refer to $\calF$ as the feasible region of the integer optimization problem (Constraints~\eqref{supp_plan}--\eqref{fire integrality constr}). We exploit the fact that the linking constraints~\eqref{link set partition} along with the set partitioning constraints~\eqref{crew_route} induce a 0-1 knapsack structure at any time $t\in\calT$---namely, fire demand plus idle crews must not exceed available crews:
\begin{equation}\label{knapsack}
    \sum_{g\in \calG}\sum_{q\in \calQ_g} B_{qt}y_q + \sum_{j\in \calJ}\sum_{p\in \calP_j}\mathbbm 1\left[\sum_{g\in \calG}A_{pgt} = 0\right]z_p \le J.
\end{equation}

Proposition~\ref{prop:robust} identifies a broad family of ``robust'' cuts, i.e., valid inequalities that do not impact the shortest path structure of the pricing problems. Indeed, the dual variables associated with Equation~\eqref{robust_inequality} can be incorporated by adjusting arc costs in the time-expanded networks.

\begin{proposition}[Robust inequalities]\label{prop:robust}
The following set of valid inequalities $\calU$ is robust:
\begin{equation}\label{robust_inequality}
    \sum_{g\in \calG}\sum_{q\in \calQ_g}\sum_{d=0}^{J} \mathbbm 1[B_{qt_u} = d]\delta_{ugd}y_q + \sum_{j\in \calJ}\sum_{p\in \calP_j}\sum_{g\in \calG}\sum_{d=0}^1\gamma_{ugjd}\mathbbm 1[A_{pgt_u} = d]z_p \le R_u,\ \forall u\in\calU.
\end{equation}
\end{proposition}

\subsubsection*{GUB cover cuts.}

Theorem~\ref{thm::GUB cover} adapts the GUB cover cuts from \cite{wolsey1990valid} to our problem.

\begin{theorem} \label{thm::GUB cover}
    Let $t\in \calT$, $\calG_u \subseteq \calG$, $\calJ_u \subseteq \calJ$, and  $\{D_{ug}\} \in \{0,\ldots, J\}^{\calG_u}$ be such that $\sum\limits_{g\in \calG_u} D_{ug} + |\calJ_u|> J$. The following GUB cover inequality is valid for Problem~\eqref{set partition} and robust:
\begin{equation}\label{gub_cover}
    \sum_{g\in \calG_u}\sum_{q\in \calQ_g} \mathbbm 1[B_{qt} \ge D_{ug}]y_q + \sum_{j\in \calJ_u}\sum_{p\in \calP_j}\mathbbm 1\left[\sum_{g\in \calG_u}A_{pgt} = 0\right]z_p \le |\calG_u| + |\calJ_u| - 1.
\end{equation}
\end{theorem}

We can interpret $D_{ug}$ as a target demand for fire $g\in\calG_u$. The cut states that, if the total target plus the number of crews in $\calJ_u$ exceeds the total number of crews, at least one fire in $\calG_u$ will not meet its target demand or at least one crew in $\calJ_u$ will have to suppress one fire in $\calG_u$ at time $t\in\calT$.

\paragraph{Example with $J=10$:} Assume that Fire 1 is suppressed by 2 crews at time $t$, and Fire 2 is assigned 4-crew, 8-crew, and 10-crew plans with weights 0.6, 0.2, and 0.2. The following cut ($\calG_u=\{1,2\}$ and $\calJ_u=\emptyset$) eliminates the solution by expressing that we cannot leverage 12 crews:
    \begin{align*}
        & \sum\limits_{q\in \calQ_1} \mathbbm 1[B_{qt} \ge 2]y_q  + \sum\limits_{q\in \calQ_2} \mathbbm 1[B_{qt} \ge 10]y_q  \le 1.
    \end{align*}

We provide additional details on these cuts in~\ref{app:GUBdetails}. We first elicit an enumerative scheme to separate the GUB cut. We then derive a procedure to strengthen the GUB cuts in our column generation context---that is, to tighten the cuts obtained in a lower-dimensional polyhedron into the full-dimensional polyhedron by reintroducing omitted variables. We also exploit a matroid structure to devise an efficient greedy algorithm in the strengthening procedure.

\subsubsection*{Augmented GUB (A-GUB) cuts.} Theorem~\ref{thm:general_gub_aware_cut} identifies a broader family of robust cuts.

\begin{theorem} \label{thm:general_gub_aware_cut}

Let $t\in \calT$, $\calG_u \subseteq \calG$, $\calJ_u \subseteq \calJ$, $\{\delta_{ugd}\} \in [0,1]^{\mathcal{G}_u\times\{0, \ldots, J\}}$, $K_u\in \mathbb R$. Let us define:
$$\calH=\left\{S\subseteq\mathcal{G}_u\times\{0, \ldots, J\}:\sum\limits_{(g, d)\in S} d \le J - |\calJ_u|,\ \text{and}\ |\{(g', d) \in S : g' = g\}| \le 1,\ \forall g\in \mathcal{G}_u\right\}$$
If $\sum\limits_{(g, d) \in S} \delta_{ugd} \le K_u$ for all subsets $S\in\calH$, then the following inequality is valid for $\text{conv}(\calF)$:
\begin{equation}\label{general_gub_ineq}
    \sum_{g\in \calG_u}\sum_{q\in \calQ_g}\sum_{d=0}^{J} \mathbbm 1[B_{qt} = d]\delta_{ugd}y_q + \sum_{j\in \calJ_u}\sum_{p\in \calP_j}\mathbbm 1\left[\sum_{g\in \calG_u}A_{pgt} = 0\right]z_p \le |\calJ_u| + K_u.
\end{equation}
\end{theorem}

The set $\calH$ stores all sets of target demands for fires in $\calG_u$ that could be met without using any crew in $\calJ_u$. The inequality states that, if we choose weighted fire demands for $\calG_u$ that exceed a budget $K_u$ induced by $\calH$, then at least one of the crews in $\calJ_u$ must suppress one of the fires in $\calG_u$. In particular, Equation~\eqref{general_gub_ineq} adds a degree of freedom to Equation~\eqref{gub_cover} by allowing fractional weights $\delta_{ugd}$. Proposition~\ref{prop:CGLP superset} shows that, in turn, this cut family tightens the GUB linear relaxation.

\begin{proposition} \label{prop:CGLP superset}
    Let $\calF_1$ and $\calF_2$ be the solutions to Problem~\eqref{set partition} that satisfy all inequalities from Theorem~\ref{thm::GUB cover} and from Theorem~\ref{thm:general_gub_aware_cut}, respectively. Then $\calF_2 \subseteq \calF_1$.
\end{proposition}

\paragraph{Example with $J=10$:} Assume that Fires 1 and 2 are assigned 0-crew and 2-crew plans with weights of 0.1 and 0.9; and that Fire 3 is assigned 6-crew, 7-crew and 9-crew plans with weights 0.8, 0.1 and 0.1. The following cut ($\calG_u=\{1,2,3\}$, $\calJ_u=\emptyset$) eliminates the solution by expressing that, if we serve a target demand of $D_{ug}=9$ for fire 3, then we cannot serve a target demand of 2 for fire 1 \textit{or} fire 2; and that, if we serve a target demand of 7 or 8 for fire 3, then we cannot serve the target demand of 2 for fire 1 \textit{and} fire 2 simultaneously.
\begin{align*}
        \frac{1}{2}\sum\limits_{q\in \calQ_1} \mathbbm 1[B_{qt} \ge 2]y_q  + &\frac{1}{2}\sum\limits_{q\in \calQ_2} \mathbbm 1[B_{qt} \ge 2]y_q + \frac{1}{2}\sum\limits_{q\in \calQ_3} \mathbbm 1[7 \le B_{qt} \le 8]y_q + \sum\limits_{q\in \calQ_3} \mathbbm 1[B_{qt} \ge 9]y_q  \le 1
    \end{align*}

\paragraph{Separation.} Problem~\eqref{general_gub_cglp} provides a cut generating linear program (CGLP) by maximizing the violation from a fractional solution $(\mathbf y^*, \mathbf z^*)$, given unused crews $\calJ_u = \bigg\{j \in \calJ : \sum\limits_{p\in \calP_j} \sum\limits_{g\in \calG} A_{pgt}z^*_p = 0\bigg\}$.
    \begin{maxi!}
    {\bdelta,K}{\left(\sum_{g\in \calG}\sum\limits_{q \in \calQ_g}\sum_{d=0}^{J} \mathbbm{1}[B_{qt} = d]y^*_q\delta_{gd}\right) - K }{\label{general_gub_cglp}}{} 
    \addConstraint{\sum_{(g, d)\in S}\delta_{gd} \le K,\, \forall S\in \calH} \label{exponential contraint}
    \addConstraint{0\leq \delta_{gd} \le 1,\, \forall g\in \calG, \forall d\in \{0, \ldots , J\}}
    \end{maxi!}
    
\subsection{Branching with dual-aware variable selection}
\label{section:branching}

We embed the two-sided column generation and cutting planes into a branch-and-price-and-cut algorithm. We devise a branching tree by branching on fire demand and crew assignments:
\begin{align}
    & \left(\sum\limits_{q \in \calQ_g}\mathbbm{1}[B_{qt} > d]y_q = 0\right)  \lor \left(\sum\limits_{q \in \calQ_g}\mathbbm{1}[B_{qt} > d]y_q = 1\right)\ \text{for $g\in\calG$, $t\in\calT$, and $d \in \{0, \dots, J - 1\}$} \label{fire_demand_b_rule}\\
    & \left(\sum\limits_{p \in \calP_j }\mathbbm{1}[A_{pgt} = 1]z_p = 0\right)  \lor \left(\sum\limits_{p \in \calP_j }\mathbbm{1}[A_{pgt} = 1]z_p = 1\right)\ \text{for $j\in\calJ$, $g\in\calG$, and $t\in\calT$} \label{crew_supply_b_rule}
\end{align}

These disjunctions define two child nodes that preserve all integer solutions. We incorporate the disjunctive constraints in the master problem by replacing the set partitioning constraints (Equation~\eqref{supp_plan} or~\eqref{crew_route}) with inequalities of the form $\sum\limits_{q \in \calQ_g}\mathbbm{1}[d_1 < B_{qt} \le d_2]y_q = 1$ and $\sum\limits_{p \in \calP_j}\mathbbm{1}[A_{pgt} = k]z_p = 1$, and in the pricing problems by removing arcs violating fire demands or crew assignments.

\subsubsection*{Variable selection.}

When selecting the variable to branch on at each node, the common most-fractional branching rule is insufficient to guarantee the exactness of the algorithm in our problem---again, due to the non-linear fire dynamics. For example, most-fractional branching may terminate with a solution in which fire $g$ is served by a 0-crew plan and a 10-crew plan with $0.5$ weights. If the fire damage cost were linear, this solution could be brought into an ``average'' 5-crew plan. However, the cost of the average plan may not be equal to the average cost of the constituent plans, so this transformation is no longer guaranteed to lead to an optimal solution.

\paragraph{Maximum-variance branching.}
Using a probabilistic interpretation of $(\mathbf y^*, \mathbf z^*)$, we propose instead a maximum-variance rule that branches on a variable with the largest variance, defined as follows:
\begin{align*}
    v(g, t) &= \left(\sum_{q\in \calQ_g} B_{qt}^2y^*_q\right) - \left(\sum_{q\in \calQ_g} B_{qt}y^*_q\right)^2\ \text{and}\ 
    w(j, g, t) = \left(\sum_{p\in \calP_j} A_{pgt}^2z^*_p\right) - \left(\sum_{p\in \calP_j} A_{pgt}z^*_p\right)^2
\end{align*}

Proposition~\ref{prop:min_var} shows that this rule eliminates the incumbent fractional solution whenever it branches on a positive-variance variable. Vice versa, if all variables have zero variance, fire demands and crew assignments are identical across all columns of positive support, and the fractional solution can be converted into an integer solution of same cost. Therefore, unlike most-fractional branching, maximum-variance branching yields an exact branch-and-price-and-cut algorithm.

\begin{proposition} \label{prop:min_var}
    If $v(g,t)>0$ for $g\in\calG,t\in\calT$, or if $w(j,g,t)>0$ for some $j\in\calJ,g\in\calG,t\in\calT$, then $(\mathbf y^*, \mathbf z^*)$ is infeasible in each child node defined by Equations~\eqref{fire_demand_b_rule}--\eqref{crew_supply_b_rule}. Vice versa, $v(g,t)=0$ if and only if there exists $\beta_{gt}\in\Z_+$ such that $B_{qt} = \beta_{gt}$ for all $q\in\calQ_g$ such that $y^*_q>0$; and, $w(j,g,t)=0$ if and only if there exists $\alpha_{jgt}\in\{0,1\}$ such that $A_{pgt} = \alpha_{jgt}$ for all $p\in\calP_j$ such that $z^*_p>0$.
\end{proposition}

\paragraph{Dual-aware maximum-variance branching.}

We use dual prices to prioritize high-impact fires in the maximum-variance branching rule. We modify $v(g,t)$ and $w(j,g,t)$ as follows:
\begin{align*}
    v(g, t) &= \left(\sum_{q\in \calQ_g} (B_{qt}\rho^*_{gt})^2y^*_q\right) - \left(\sum_{q\in \calQ_g} B_{qt}\rho^*_{gt}y^*_q\right)^2\ \text{and}\ 
    w(j, g, t) = \left(\sum_{p\in \calP_j} (A_{pgt}\rho^*_{gt})^2z^*_p\right) - \left(\sum_{p\in \calP_j} A_{pgt}\rho^*_{gt}z^*_p\right)^2
\end{align*}

Proposition~\ref{prop:min_var} continues to hold under this branching rule, as long as $\boldsymbol\rho^* > \boldsymbol 0$. In practice, we may perturb $\boldsymbol \rho^*$ by a small positive $\boldsymbol \varepsilon$ to retain guarantees of exactness and finite convergence.

\subsection{Upper-bounding heuristic.}
\label{sec:heuristic}

We embed a heuristic into the branch-and-price-and-cut scheme (every 120 seconds, in our experiments) by restoring integrality with increasingly large demands (see Algorithm~\ref{alg:heuristic} in~\ref{app:heuristic}).

\subsection{Full branch-and-price-and-cut algorithm}

Algorithm~\ref{alg:branch-and-price-and-cut} summarizes the branch-and-price-and-cut algorithm, leveraging two-sided column generation (Section~\ref{section:CG}, Step 1.a), cutting planes (Section~\ref{section:cutting planes}, Step 1.b), branching rules (Section~\ref{section:branching}, Step 4), and the heuristic (Section~\ref{sec:heuristic}, Step 5). Theorem~\ref{thm:validity of branch-and-price-and-cut} establishes its exactness and finite convergence, enabled by upper and lower bounds, and by the branching rules (Proposition~\ref{prop:min_var}).

\begin{theorem} \label{thm:validity of branch-and-price-and-cut}
    Algorithm~\ref{alg:branch-and-price-and-cut} returns an optimal solution of Problem~\eqref{set partition} in finitely many iterations.
\end{theorem}

\begin{algorithm} [h!]
\caption{\textsc{BranchPriceAndCut}.}\small
\label{alg:branch-and-price-and-cut}
\begin{algorithmic}
\item \textbf{Initialization:} node set $\calN = \{1\}$, counter $n \gets 0$, node $i \gets 1$, upper bound $\texttt{UB} \gets \infty$, lower bound $\texttt{LB} \gets 0$.
\item Iterate over Steps 1-6 until termination criterion.
\begin{itemize}
    \item[] \textbf{Step 1.} Increment $n\gets n+1$, and iterate between Steps 1.a and 1.b.
    \begin{itemize}
        \item[] \textbf{Step 1.a.} Run \textsc{TwoSidedColumnGeneration} (Algorithm~\ref{alg:2CG_verbose}). If $i=1$, initialize with $\calQ' = \emptyset$, $\calP' = \emptyset$, and $\calU = \emptyset$. Otherwise, initialize with $\calQ'$ and $\calP'$ from the parent node, and $\calU$ as the subset of cuts with positive dual price in parent solution. Store objective $c$ and solution $(\by, \bz, \boldsymbol \sigma, \boldsymbol \pi, \boldsymbol \rho)$.
        \item[] \textbf{Step 1.b.} For each time $t\in\calT$, add violated cuts to $\calU$, if any (Section~\ref{section:cutting planes}).
    \end{itemize}
    \item[] \textbf{Step 2.} Bounding: Update $\texttt{LB}$ as the lower bound among all unexplored leaves.
    \item[] \textbf{Step 3.} Pruning: (i) If $(\by, \bz)$ is integer and $c < \texttt{UB}$, update $\texttt{UB}\gets c$, store best solution $(\by, \bz)$, and go to step 6; (ii) if $c \geq \texttt{UB}$, prune node, go to step 5.
    \item[] \textbf{Step 4.} Branching: define two children nodes and append them to $\calN$ (Section~\ref{section:branching}).
    \item[] \textbf{Step 5.} [\texttt{OPTIONAL}] Run heuristic (Algorithm~\ref{alg:heuristic}). If cost is less than $\texttt{UB}$, update $\texttt{UB}$ and solution $(\by, \bz)$.
    \item[] \textbf{Step 6.} If $|\calN| = n$, \texttt{STOP}: return $(\by, \bz)$. Otherwise, set $i$ to an unexplored node, and go to Step 1.
\end{itemize}
\end{algorithmic}
\end{algorithm}

\section{A Data-driven Fire Spread Model}
\label{sec:causalML}

As mentioned earlier, the optimization model relies on a fire spread model, which requires us to define the state variable $S_{gt}$, the transition functions $S_{g,t+1}=f_{gt}(S_{gt},x_{gt})$ and the cost functions $d_{gt}(S_{gt},x_{gt})$ and $h_g(S_{g,T+1})$. We develop a double machine learning model to learn these functions from historical wildfire data (described in Section~\ref{sec:data}). To cast the problem in statistical learning terms, we introduce new notation in this section that supersedes earlier notation.

\subsection{Variable Definition and Feature Engineering}

Our dataset defines one observation per fire-day pair, indexed by $i=1,\cdots,n$ (i.e., a $(g,t)$ pair using earlier notation). By design, all variables only make use of information available in the current day (period $t$) to predict the following day ($t+1$).

\paragraph{Outcome variable:} forward-looking fire growth in acres (\texttt{next$\_$area}), defined as the difference between the area burned on day $t+1$ and the one on day $t$. This measure serves as a proxy for wildfire damage, while retaining a simple one-dimensional structure both in the predictive model (to mitigate variance) and in the prescriptive model (to avoid higher-dimensional time-state networks).

\paragraph{Treatment variable:} IHC-equivalent suppression effort per fire (as described in Section~\ref{sec:data}).

\paragraph{Covariates:} We define the following multi-dimensional feature representation to capture the characteristics of wildfire growth. Details and rationales can be found in Appendix~\ref{app:features}.
\begin{itemize}
    \item[--] State variable: current area (\texttt{area}) and momentum (\texttt{area$\_$diff}). These covariates map into the two-dimensional state variable $S_{gt}$ in the fire spread model.
    \item[--] Base data: time (days since ignition, month); ignition cause (4 categories: human, lightning, other, undetermined); suppression method (5 categories: monitor/manage, confine, full suppression, point/zone protection, multiple management strategy); point-of-origin fuel type (13 categories from the NFDRS fire-behavior fuel models, including timber, brush, grass, and logging-slash variants); fire behavior (4 categories: minimal, moderate, active, extreme); and granular fire behavior (16 attributes including creeping, smoldering, running, crowning, spotting variants, torching variants, backing, flanking, and uphill/wind-driven runs).
    \item[--] ERA5 data: daily statistics from 10 weather-related satellite channels characterizing wind, temperature, moisture, precipitation, radiation, evaporation, soil, and vegetation.
    \item[--] LandFire data: 24 geospatial variables, including 18 fuel indicators, 3 terrain variables, 2 vegetation cover fractions, and canopy base height mean. Unlike point-of-origin fuel types, these remote sensing data reflect the dynamic and heterogeneous fuel and terrain structure.
    \item[--] AlphaEarth embeddings: 64-dimensional vector providing a high-resolution spatial representation of land surface from multi-spectral satellite data (terrain complexity, vegetation, land use, surface properties, etc.) without explicit feature engineering.
\end{itemize}

\subsection{Double Machine Learning for Counterfactual Fire-Growth Prediction} 
\label{subsec:DML}

We seek to estimate a counterfactual treatment-response curve $\mu(x,\bX)=\E[Y(x)\mid \bX]$, where $Y$ is the outcome variable (\texttt{next\_area}), $Y(x)$ is the potential outcome under allocation $x$, $T$ is the observed treatment (crew assignments), and $\bX$ is the vector of covariates. This curve can be interpreted causally from observational data under three standard conditions: (i) consistency (the observed outcome equals the potential outcome under the assigned allocation); (ii) conditional ignorability (upon conditioning on $\bX$, decisions are not driven by unobserved factors that also affect the outcome); and (iii) overlap (each candidate allocation lies in the historical data support). The second condition motivates the inclusion of the extensive set of covariates, defined in collaboration with the US Forest Service. Under these conditions, $\mu(x,\bX)=\E[Y\mid \bX, T=x]$.

Because these assumptions are not directly testable, we interpret $\widehat{\mu}(x,\bX)$ as a counterfactual within the optimization model based on observational data. Recall that the central estimation challenge lies in the confounding between treatment and outcome variables---more crews were assigned to higher-risk fires historically. Thus, a naive machine learning model may lead to biased estimates of the impact of crews and may even---erroneously---learn that more crews lead to larger growth. We proceed via double machine learning (DML) to mitigate this observed confounding.

The DML approach orthogonalizes outcome and treatment variables with respect to observed covariates before estimating the residual treatment-response curve (see \cite{robinson1988root} in partially linear regression and \cite{chernozhukov2018double} in high-dimensional, non-linear settings). Formally, we let $g_0(\bX)=\E[Y\mid \bX]$ and $m_0(\bX)=\E[T\mid \bX]$, and define residuals $\Delta_Y=Y-g_0(\bX)$ and $\Delta_T=T-m_0(\bX)$. We model the relationship between residual treatment and outcome variables as:
\[
\Delta_Y=f_0(\Delta_T,\bX)+U, 
\qquad 
f_0(\Delta_T,\bX)=\E[\Delta_Y\mid \bX,\Delta_T],
\qquad 
\E[U\mid \bX,\Delta_T]=0.
\]

The DML procedure proceeds in two stages using cross-fitting for out-of-fold predictions, namely:

\begin{enumerate}

\item \textbf{Nuisance models:} We first estimate the nuisance functions $g_0(\bX)$ and $m_0(\bX)$ to isolate the effects of the covariates $\bX$ on the outcome $Y$ and the treatment $T$. This yields estimates $\widehat{g}_0(\bX)$ and $\widehat{m}_0(\bX)$ of the conditional expectations $\E[Y\mid\bX]$ and $\E[T\mid\bX]$. For both, we employ LightGBM models \citep{ke2017lightgbm} and implement 3-fold cross-fitting using GroupKFold partitioning by fire incident to prevent information leakage across days of the same fire. We then compute residuals $\Delta_Y = Y - \widehat{g}_0(\bX)$ and $\Delta_T = T - \widehat{m}_0(\bX)$ from out-of-fold predictions, which remove components of the outcome and treatment explained by observed covariates.

\item \textbf{Residual model:} The second-stage model estimates a function $\widehat{f}(\Delta_T,\bX)$ by regressing the outcome residual $\Delta_Y$ on the treatment residual $\Delta_T$ and the covariates $\bX$. This models the residual relationship between suppression effort and fire growth in our context. We employ LightGBM while ensuring that $\widehat{f}$ remains non-increasing in $\Delta_T$; all else equal, additional suppression effort shall not increase
next-day fire growth. One challenge is that, with high-dimensional covariates, the LightGBM trees may not split on the treatment residual, leading to a flat response curve. We mitigate this practical degeneracy by building a four-model ensemble with skip-zero averaging to retrieve an effective response curve (see~\ref{app:ensemble} for details).
\end{enumerate}

The outcome nuisance model captures expected next-day fire growth under historical dispatch patterns conditional on $\bX$, whereas the treatment nuisance model defines the historical-policy reference allocation. For each observation $i$, i.e., each fire-day pair in our context, and each possible decision $x$, i.e., each candidate crew allocation, we construct the counterfactual prediction
\[
\widehat{\mu}_i(x)=\widehat{g}_0(\bX_i)+\widehat{f}\bigl(x-\widehat{m}_0(\bX_i),\bX_i\bigr).
\]
These prediction curves provide the data-driven fire-growth transitions used to build the time-state networks in the optimization model.

\subsection{Predictive results}

We split the data into a training set (2015-2022) and a holdout test set (2023-2024) to mimic real-world deployment and prevent data leakage. All hyperparameters were tuned using 3-fold cross-validation. Tables~\ref{tab:nuisanceY}--~\ref{tab:secondStage} report performance on the test set. We vary the feature sets, starting with the base features from the SIT-209 data and then adding the ERA5 data, the LandFire data, and the AlphaEarth embeddings; for each, we test the model with area burned as the only endogenous variable and with area burned and momentum. As mentioned above, we use LGBM for the nuisance models; we report additional results in~\ref{app:predictiveresults} supporting that choice.

\begin{table}[h!]
\centering
\caption{First-stage performance for predicting the outcome variable (unit: acres).}
\label{tab:nuisanceY}
\resizebox{\textwidth}{!}{
\begin{tabular}{cccccccrrrr}
\toprule
\multicolumn{6}{c}{Feature Sources} & \multirow{2}{*}{Model} & \multirow{2}{*}{RMSE} & \multirow{2}{*}{MAE} & \multirow{2}{*}{$R^2$} & \multirow{2}{*}{CV RMSE} \\
\cmidrule(lr){1-6}
\#Feat & SIT209 & Momentum & ERA5 & LandFire & AlphaEarth &  &  &  &  &  \\
\midrule
45 & \checkmark &  &  &  &  & LGBM & 559 & 400 & 0.204 & 604 \\
46 & \checkmark & \checkmark &  &  &  & LGBM & 536 & 379 & 0.268 & 590 \\
\midrule
55 & \checkmark &  & \checkmark &  &  & LGBM & 558 & 393 & 0.208 & 604 \\
56 & \checkmark & \checkmark & \checkmark &  &  & LGBM & 534 & 375 & 0.272 & 588 \\
\midrule
79 & \checkmark &  & \checkmark & \checkmark &  & LGBM & 550 & 383 & 0.230 & 598 \\
80 & \checkmark & \checkmark & \checkmark & \checkmark &  & LGBM & 530 & 370 & 0.284 & 585 \\
\midrule
143 & \checkmark &  & \checkmark & \checkmark & \checkmark & LGBM & 549 & 387 & 0.231 & 597 \\
144 & \checkmark & \checkmark & \checkmark & \checkmark & \checkmark & LGBM & 526 & 367 & 0.295 & 582 \\
\bottomrule
\end{tabular}
}
\end{table}

\begin{table}[h!]
\centering
\caption{First-stage performance for predicting the treatment variable (unit: IHC-equivalent crews)}
\label{tab:nuisanceT}
\resizebox{\textwidth}{!}{
\begin{tabular}{cccccccrrrr}
\toprule
\multicolumn{6}{c}{Feature Sources} & \multirow{2}{*}{Model} & \multirow{2}{*}{RMSE} & \multirow{2}{*}{MAE} & \multirow{2}{*}{$R^2$} & \multirow{2}{*}{CV RMSE} \\
\cmidrule(lr){1-6}
\#Feat & SIT209 & Momentum & ERA5 & LandFire & AlphaEarth &  &  &  &  &  \\
\midrule
45 & \checkmark &  &  &  &  & LGBM & 6.13 & 2.40 & 0.457 & 5.90 \\
46 & \checkmark & \checkmark &  &  &  & LGBM & 6.10 & 2.33 & 0.462 & 5.90 \\
\midrule
55 & \checkmark &  & \checkmark &  &  & LGBM & 4.84 & 2.06 & 0.661 & 5.81 \\
56 & \checkmark & \checkmark & \checkmark &  &  & LGBM & 4.95 & 2.06 & 0.646 & 5.76 \\
\midrule
79 & \checkmark &  & \checkmark & \checkmark &  & LGBM & 5.21 & 2.25 & 0.608 & 6.08 \\
80 & \checkmark & \checkmark & \checkmark & \checkmark &  & LGBM & 4.90 & 2.19 & 0.653 & 6.13 \\
\midrule
143 & \checkmark &  & \checkmark & \checkmark & \checkmark & LGBM & 4.79 & 2.08 & 0.668 & 5.94 \\
144 & \checkmark & \checkmark & \checkmark & \checkmark & \checkmark & LGBM & 4.94 & 2.17 & 0.649 & 5.97 \\
\bottomrule
\end{tabular}
}
\end{table}

\begin{table}[h!]
\centering
\caption{Second-stage DML performance for predicting the treatment effect (unit: acres).}
\label{tab:secondStage}
\resizebox{\textwidth}{!}{
\begin{tabular}{cccccccrrrr}
\toprule
\multicolumn{6}{c}{Feature Sources} & \multirow{2}{*}{Model} & \multirow{2}{*}{RMSE} & \multirow{2}{*}{MAE} & \multirow{2}{*}{$R^2$} & \multirow{2}{*}{CV RMSE} \\
\cmidrule(lr){1-6}
\#Feat & SIT209 & Momentum & ERA5 & LandFire & AlphaEarth &  &  &  &  &  \\
\midrule
45 & \checkmark &  &  &  &  & LGBM & 563 & 402 & 0.192 & 624 \\
46 & \checkmark & \checkmark &  &  &  & LGBM & 539 & 380 & 0.260 & 603 \\
\midrule
55 & \checkmark &  & \checkmark &  &  & LGBM & 562 & 396 & 0.196 & 618 \\
56 & \checkmark & \checkmark & \checkmark &  &  & LGBM & 536 & 375 & 0.269 & 601 \\
\midrule
79 & \checkmark &  & \checkmark & \checkmark &  & LGBM & 553 & 385 & 0.222 & 611 \\
80 & \checkmark & \checkmark & \checkmark & \checkmark &  & LGBM & 533 & 372 & 0.277 & 596 \\
\midrule
143 & \checkmark &  & \checkmark & \checkmark & \checkmark & LGBM & 549 & 383 & 0.232 & 609 \\
144 & \checkmark & \checkmark & \checkmark & \checkmark & \checkmark & LGBM & 529 & 368 & 0.286 & 594 \\
\bottomrule
\end{tabular}
}
\end{table}

In terms of magnitude, the $R^2$ is lower for the outcome variable than for the treatment variable, reflecting the complexity of capturing wildfire behavior. For the outcome variable, predictive performance improves as features are added, but the gains remain modest, indicating the inherent difficulty of predicting fire growth. In comparison, the momentum variable results in larger gains. These results justify the augmentation of the fire state variable to include both area and momentum, despite the added computational complexity in the optimization problem. For the treatment variable, the nuisance model benefits from weather data but again improvements become smaller thereafter. Moreover, the model benefits mildly from the momentum variable. These results suggest that dispatch decisions were primarily based on the current fire area and weather forecasts, which we will use to define a benchmark in Section~\ref{sec:results}. These observations are corroborated by the grouped SHAP plots in Figure~\ref{fig:SHAPgrouped}; full SHAP plots are relegated to~\ref{app:predictiveresults} \citep{lundberg2017unified}.

\begin{figure}[h!]
    \centering
    \begin{subfigure}[t]{0.48\linewidth}
        \centering
        \includegraphics[width=\linewidth]{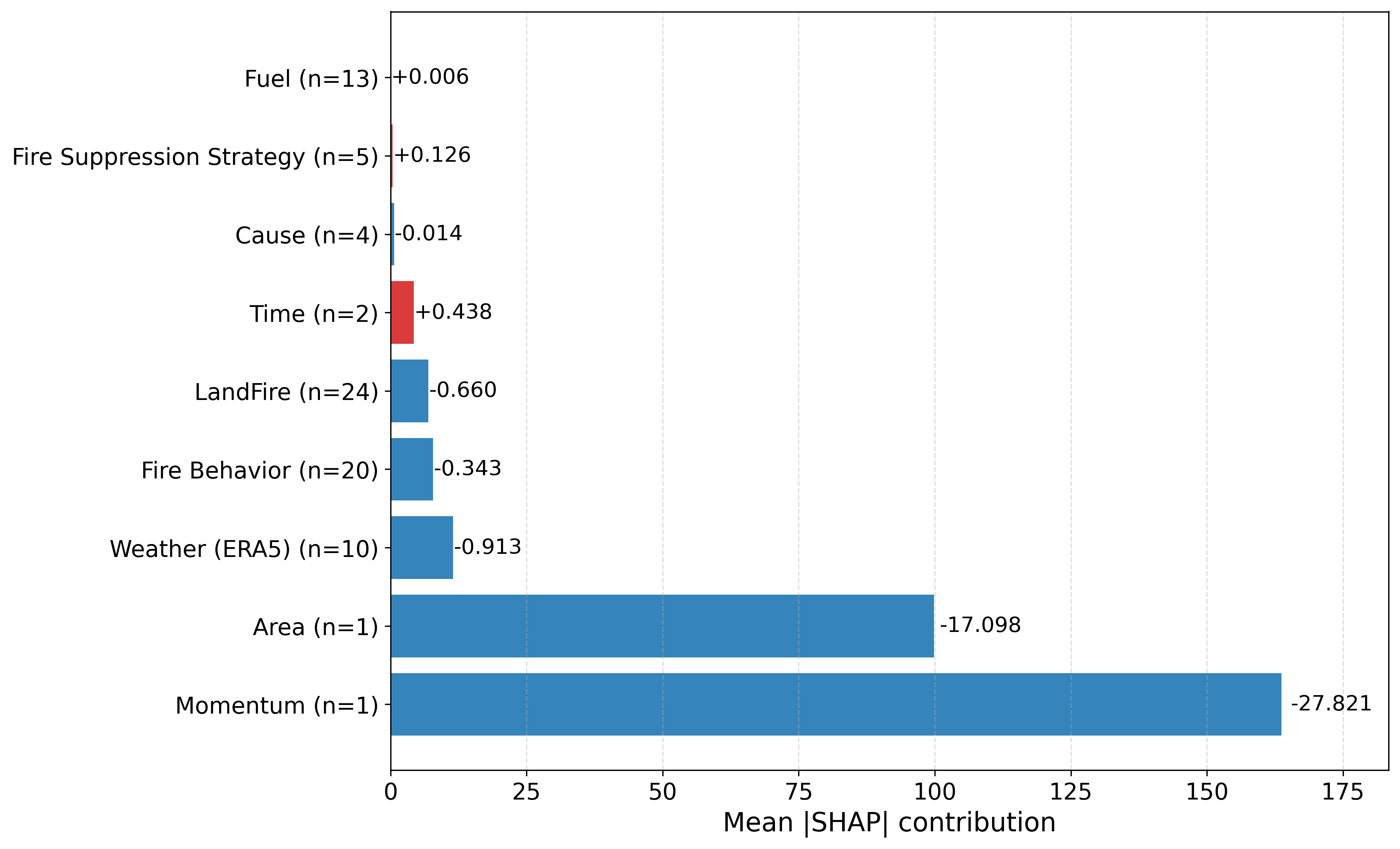}
        \caption{Model predicting outcome variable.}
        \label{fig:SHAPgroupedY}
    \end{subfigure}
    \hfill
    \begin{subfigure}[t]{0.48\linewidth}
        \centering
        \includegraphics[width=\linewidth]{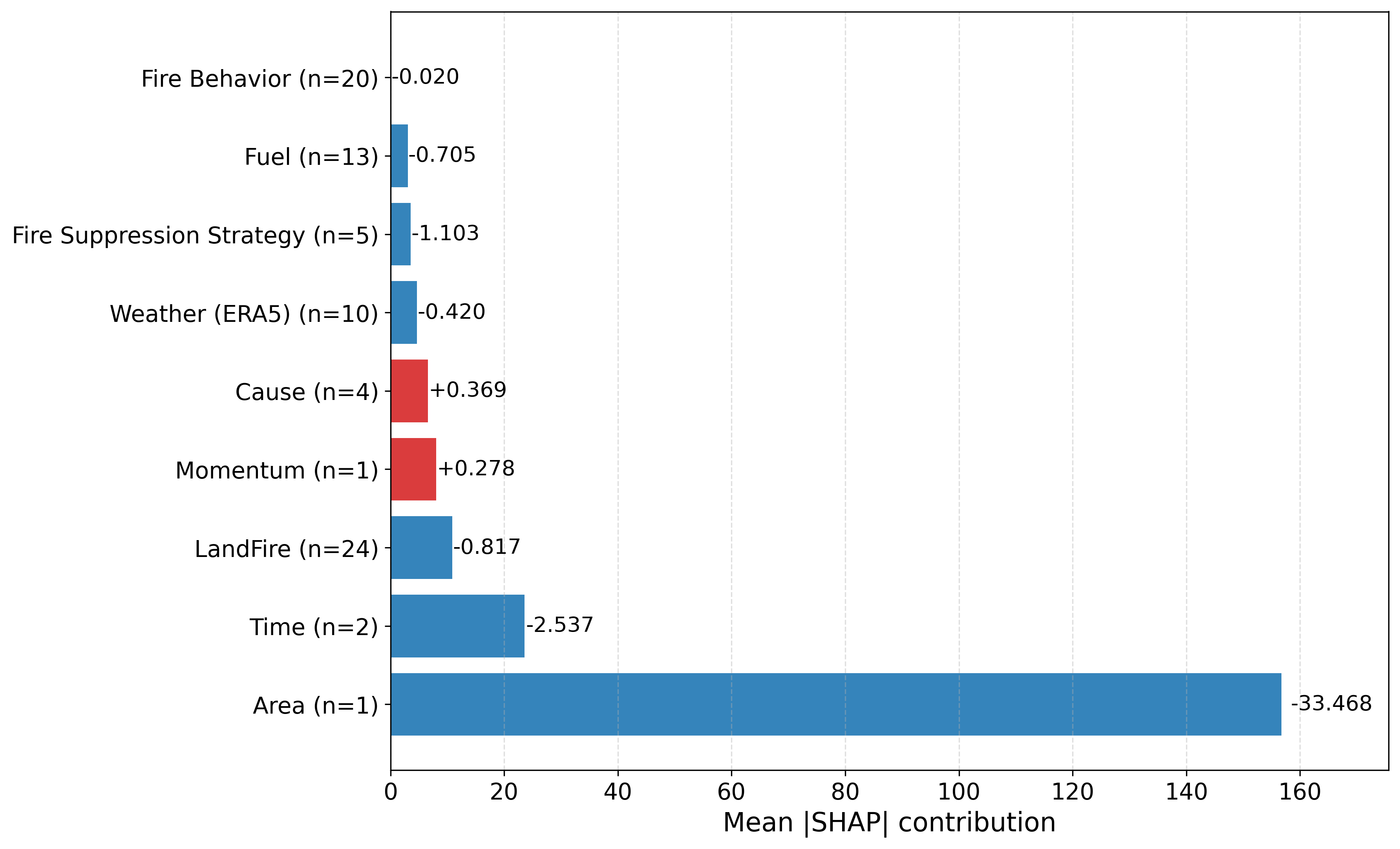}
        \caption{Model predicting treatment variable.}
        \label{fig:SHAPgroupedT}
    \end{subfigure}
    \caption{Grouped SHAP plots (SHapley Additive exPlanations) for the first-stage models in the DML trained on SIT-209, ERA5, LandFire, and AlphaEarth.}
    \label{fig:SHAPgrouped}
    \vspace{-18pt}
\end{figure}

Figure~\ref{fig:cate_curve} visualizes the average relationship between the treatment variable (IHC-equivalent crews) and outcome variable (daily fire growth). These results are consistent with the presence of confounding bias in baseline machine learning models. Ordinary least squares regression (OLS) infers a positive relationship between crew assignments and fire growth. The single-stage LightGBM model learns a non-monotonic response, where again, more crews can lead to larger fires. In contrast, the DML learns a smooth, decreasing relationship with diminishing returns, consistent with operational intuition. Ultimately, the DML model exhibits strong out-of-sample performance (the quality of fit in Table~\ref{tab:secondStage} is consistent with that of single-stage ML models) and yields a treatment-response curve that retains a positive impact of additional crews on wildfire suppression, with diminishing returns. Note that some bias may remain because the orthogonalization procedure may not eliminate all unobserved confounding, measurement error, or interference across fires. Accordingly, the treatment-response curves should be interpreted as an observational estimate of suppression effects used to construct counterfactual wildfire propagation dynamics.

\begin{figure}[h!]
    \centering
    \includegraphics[width=0.6\textwidth]{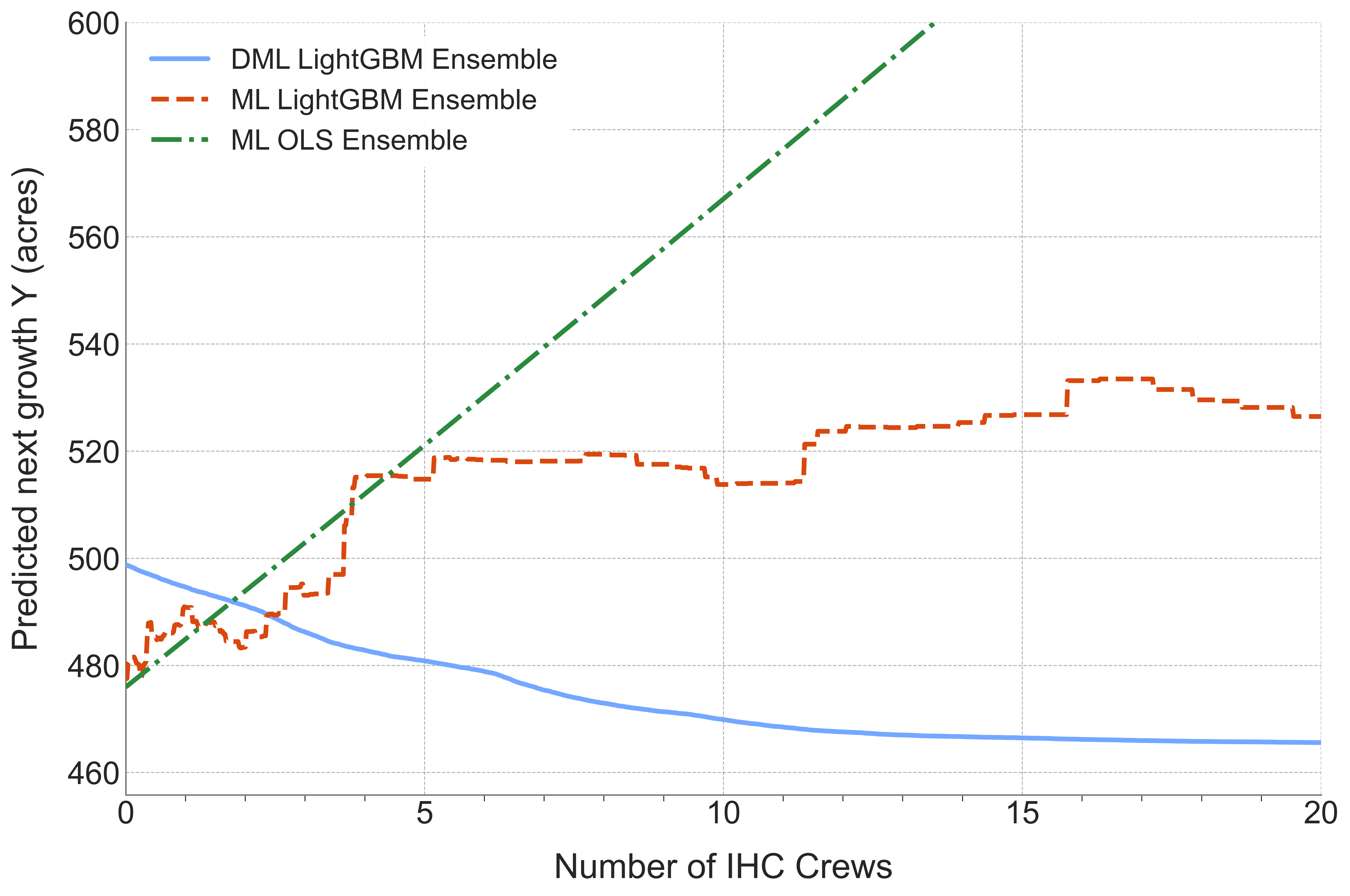}
    \caption{Average treatment-response curve with machine learning (ML) and double machine learning (DML).}
    \label{fig:cate_curve}
\end{figure}

\subsubsection*{Fire spread model.} 

We use these predictions to build the fire spread model and the corresponding time-state network, illustrated in Figure~\ref{fig:fire_state_network} for one fire. Recall that the two-dimensional state variable comprises area burned and momentum. We discretize the state space using a piecewise uniform grid to provide higher resolution for the smaller---and more common---fires. We then iterate through each discretized state and each possible crew assignment, and derive the next state using the predictive model. This defines the transition arcs $a\in\calB_g$, the corresponding labels $x_a$, and the cost parameters $d_a$. Details on this procedure are provided in~\ref{app:firenetwork}. Altogether, the predictive model yields a high-dimensional fire-state model: the worst-case number of possible discretized states is over 2 billion (51,031 possible values both for area and momentum), and the average number of states in our data-driven networks is 5,690.

\begin{figure}[h!]
    \centering
    \includegraphics[width=.8\textwidth]{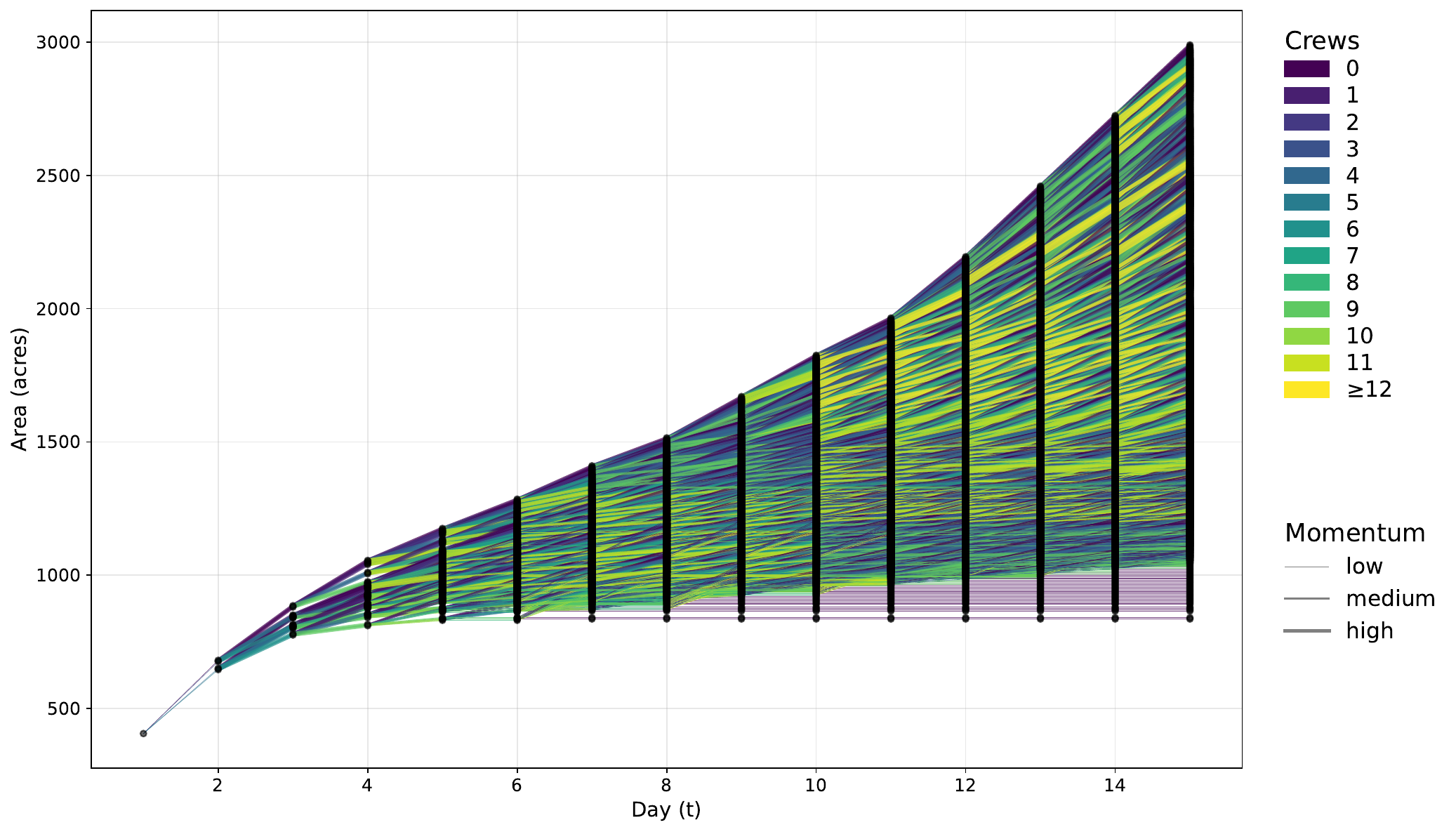}
    \caption{Example time-state network for the Stead fire in the Southern GACC on August 15, 2023. y-axis: one dimension of the fire state (area); x-axis: 14-day horizon; color gradient: number of crews required for each state transition; line width: second dimension of the fire state (momentum).}
    \label{fig:fire_state_network}
\end{figure}

\section{Optimization results}
\label{sec:results}

We evaluate our algorithm on real-world data. We build time-space-rest networks by geocoding crew stations using Google Maps and computing travel times using OpenStreetMap. We use the fire spread model obtained with our DML methodology (Section~\ref{sec:causalML}, Figure~\ref{fig:fire_state_network}). Our experiments span all fires in the contiguous United States during the Summers of 2023 and 2024, over six 14-day time periods from June through August starting on the 1st and 15th of each month. We also use the simple linear model of wildfire spread from \cite{wei2015chance} (see~\eqref{subsec:erinmodel}) in some of our experiments to establish the scalability of the methodology in a controlled environment from the literature. All models are solved with Gurobi v13.0.1 \citep{gurobi2026} using JuMP v1.30.0 \citep{dunning2017jump} in Julia v1.11.9  on a laptop with an Apple M5 Pro (15-core CPU, 24 GB unified memory). All instances and code are available online to enable replication.\footnote{Link blinded for double-blind review.}

\subsection{Scalability of the optimization methodology}
\label{subsec:scalability}

To evaluate the branch-and-price-and-cut algorithm, Table~\ref{table:value_cuts_branching_heuristic} considers two instances and incorporates each algorithmic component one at a time---cutting planes, the heuristic, and the branching rules. It reports computational times (up to a 20-minute limit), the upper and lower bounds (normalized to the best found solution), the number of nodes, and the number of columns generated.

\begin{table}[h!]
\caption{Performance of the branch-and-price-and-cut components (fire model from \cite{wei2015chance}).}\label{table:value_cuts_branching_heuristic}
\centering\renewcommand{\arraystretch}{1.0}
\small
\begin{tabular}{ccc ccccc ccccc}
\toprule
\multicolumn{3}{c}{Algorithm} & \multicolumn{5}{c}{Medium instance: 20 crews, 6 fires} & \multicolumn{5}{c}{Large instance: 30 crews, 9 fires}\\
\cmidrule(lr){1-3}\cmidrule(lr){4-8}\cmidrule(lr){9-13}
Cuts & Heur. & Branch. & CPU & Opt. gap & Sol. & Nodes & Columns & CPU & Opt. gap & Sol. & Nodes & Columns\\
\midrule
No & No & MF & 1,200 & --- & --- & 4,757 & 600,982 & 1,200 & --- & --- & 1,832 & 410,675 \\
 &  & MV & 1,200 & --- & --- & 4,804 & 606,250 & 1,200 & --- & --- & 1,778 & 400,229 \\
  &  & D-MV & 288 & 0.00\% & 1.00 & 1,439 & 152,280 & 1,200 & --- & --- & 1,902 & 493,830 \\
  & Yes & MF & 1,200 & 1.08\% & 1.01 & 4,712 & 623,146 & 1,200 & 4.58\% & 1.05 & 1,681 & 378,326 \\
 &  & MV & 1,200 & 0.98\% & 1.01 & 4,573 & 617,885 & 1,200 & 4.67\% & 1.05 & 1,663 & 376,342 \\
 &  & D-MV & 313 & 0.00\% & 1.00 & 1,437 & 170,650 & 1,200 & 2.84\% & 1.04 & 1,682 & 433,931 \\
Yes & No & MF & 1,200 & --- & --- & 660 & 134,921 & 1,200 & --- & --- & 101 & 55,575 \\
 &  & MV & 1,200 & --- & --- & 635 & 126,611 & 1,200 & --- & --- & 98 & 53,639 \\
 &  & D-MV & 229 & 0.00\% & 1.00 & 83 & 24,118 & 1,200 & --- & --- & 135 & 66,109 \\
 & Yes & MF & 1,200 & 0.65\% & 1.01 & 642 & 137,101 & 1,200 & 3.57\% & 1.04 & 52 & 58,285 \\
 &  & MV & 1,200 & 0.24\% & 1.01 & 683 & 136,763 & 1,200 & 3.95\% & 1.04 & 50 & 56,577 \\
 &  & D-MV & 179 & 0.00\% & 1.00 & 49 & 17,415 & 1,200 & 2.59\% & 1.04 & 62 & 55,285 \\
\bottomrule
\end{tabular}
\begin{tablenotes}
\item\scriptsize Branching rules: MF = most fractional, MV = maximum variance, D-MV = dual-aware maximum variance
\end{tablenotes}
\end{table}

Note the significant impact of our branch-and-price-and-cut methodology. Whereas a baseline branch-and-price approach fails to even return feasible solutions, the proposed algorithm returns the optimal solution in minutes in medium instances and derives high-quality solutions in larger instances. These improvements are enabled by all three components of the algorithm:
\begin{itemize}
    \item[--] Cutting planes enable significant speedups and solution gains. In the medium instance, they reduce computational times by 43\%, from 313 to 179 seconds. In the large one, they reduce the optimality gap from 2.84\% to 2.59\%; the difference is even larger with weaker branching strategies. We report results on the bound improvements from the A-GUB cuts in~\ref{app:cuts}.
    \item[--] The upper-bounding heuristic is critical to derive a feasible solution and accelerate the branch-and-price-and-cut algorithm. In the largest example, the heuristic enables the algorithm to derive a solution within 2.59\% of the optimum, whereas the algorithm never finds a feasible solution otherwise. These benefits come at moderate computational costs, since the heuristic is applied sporadically (every 120 seconds, in our experiments).
    \item[--] The dual-aware maximum-variance branching strategy considerably enhances the scalability of the branch-and-price-and-cut algorithm. In the medium instance, it enables convergence to the optimal solution, whereas the other two branching strategies leave an optimality gap. In the largest instance, it reduces the optimality gap from 3.5--4.7\% to 2.6--2.8\%. These results underscore the benefit of leveraging dual information in the branching rules.
\end{itemize}

Finally, Tables~\ref{tab:scalability} and~\ref{tab:bpc_comparison} assess the scalability of the branch-and-price-and-cut algorithm against (i) a polynomial formulation on the time-expanded networks (Problem~\eqref{model::full natural}); (ii) a ``2CG + IP'' benchmark that runs two-sided column generation with A-GUB cuts at the root node (Step 1 in Algorithm~\ref{alg:branch-and-price-and-cut}) and solves the master problem with integrality constraints; and (iii) a ``root node + heuristic'' that adds the upper-bounding heuristic (Steps 1 and 5 in Algorithm~\ref{alg:branch-and-price-and-cut}). Table~\ref{tab:scalability} uses the controlled setup from \cite{wei2015chance}, whereas Table~\ref{tab:bpc_comparison} uses the data-driven setup from Section~\ref{sec:causalML}. Recall that resource allocation is typically managed within each GACC, with up to 20-30 IHCs; at the same time, inter-GACC coordination can be instrumental to enable nation-wide resource sharing during the most intense fire seasons. Thus, Table~\ref{tab:bpc_comparison} tests the algorithm's scalability to the full contiguous United States (potentially divided into two halves for tractability purposes).

\begin{table}[h!]
\centering
  \caption{Performance of branch-and-price-and-cut algorithm (fire model from \cite{wei2015chance}).}
\small
\begin{tabular}{cccccccccccccc}
\toprule
  && \multicolumn{3}{c}{Polynomial} & \multicolumn{3}{c}{2CG + IP} & \multicolumn{3}{c}{Root + Heuristic} & \multicolumn{3}{c}{Branch-Price-Cut} \\\cmidrule(lr){3-5}\cmidrule(lr){6-8}\cmidrule(lr){9-11}\cmidrule(lr){12-14}
Crews & Fires & UB & CPU & Gap & UB & CPU & Gap & UB & CPU & Gap & UB & CPU & Gap \\
\hline
10 & 3  & 1.44   & 3    & 0.00\%   & 1.44   & 9    & 0.12\%   & 1.44 & 9   & 0.09\%  & 1.44 & 11   & 0.00\%  \\
20 & 6  & 2.30   & 1201 & 1.09\%   & 2.95   & 45   & 30.95\%  & 2.32 & 34  & 2.92\%  & 2.29 & 305  & 0.00\%  \\
30 & 9  & 4.98   & 1205 & 9.24\%   & 9.05   & 1200 & 102.27\% & 4.73 & 127 & 5.68\%  & 4.69 & 1200 & 2.29\%  \\
40 & 12 & 6.19   & 1208 & 29.07\%  & 10.69  & 1200 & 125.03\% & 5.02 & 180 & 5.77\%  & 4.99 & 1200 & 3.85\%  \\
50 & 15 & 7.85   & 1217 & 35.55\%  & 12.21  & 1200 & 112.53\% & 6.42 & 228 & 11.67\% & 5.96 & 1200 & 2.76\%  \\
70 & 21 & --- & 1213 & --- & --- & 1200 & --- & 7.87 & 444 & 12.83\% & 7.43 & 1200 & 6.09\%  \\
\bottomrule
\end{tabular}
\label{tab:scalability}
\end{table}

\begin{table}[ht]
  \centering
  \small
  \caption{Performance of branch-and-price-and-cut algorithm (data-driven fire model).}
  \label{tab:bpc_comparison}
  \setlength{\tabcolsep}{4pt}
  \begin{tabular}{llrrrrrrrr}
  \toprule
  &&&&\multicolumn{2}{c}{Polynomial}&\multicolumn{2}{c}{Root + Heuristic}&\multicolumn{2}{c}{Branch-Price-Cut}\\\cmidrule(lr){5-6}\cmidrule(lr){7-8}\cmidrule(lr){9-10}
  Date & GACCs & Crews & Fires & Sol. & Gap & Sol. & Gap & Sol. & Gap\\
  \midrule
  06/01/2023 & All      & 112 &  42 & 225,915 & 0.30\% & 226,229 & 0.50\% & 226,139 & 0.46\% \\
  06/15/2023 & All      & 112 &  54 & --- & --- & 333,218 & 1.02\% & 331,818 & 0.60\% \\
  07/01/2023 & East     &  41 &  16 & 34,290 & 0.18\% & 34,430 & 0.59\% & 34,385 & 0.46\% \\
  07/01/2023 & West     &  71 &  55 & --- & --- & 263,297 & 0.62\% & 262,927 & 0.48\% \\
  07/15/2023 & East     &  41 &  28 & 124,883 & 0.34\% & 125,513 & 1.23\% & 125,283 & 1.05\% \\
  07/15/2023 & West     &  71 &  89 & --- & --- & 459,633 & 1.72\% & 456,968 & 1.15\% \\
  08/01/2023 & East     &  41 &  52 & 243,523 & 0.19\% & 243,963 & 0.36\% & 243,673 & 0.24\% \\
  08/01/2023 & West     &  71 & 113 & --- & --- & 668,210 & 0.66\% & 666,720 & 0.44\% \\
  08/15/2023 & East     &  41 &  59 & 317,481 & 0.34\% & 319,136 & 0.85\% & 318,896 & 0.78\% \\
  08/15/2023 & West     &  71 & 109 & --- & --- & 585,025 & 0.92\% & 583,305 & 0.62\% \\
  \midrule
  06/01/2024 & All      & 112 &  50 & 299,945 & 0.28\% & 300,815 & 0.60\% & 300,220 & 0.40\% \\
  06/15/2024 & All      & 112 &  75 & 469,112 & 0.62\% & 470,017 & 1.26\% & 469,067 & 1.06\% \\
  07/01/2024 & East     &  41 &  43 & 272,815 & 0.30\% & 273,620 & 0.65\% & 273,095 & 0.46\% \\
  07/01/2024 & West     &  71 &  73 & 412,623 & 0.47\% & 412,824 & 0.60\% & 412,056 & 0.41\% \\
  07/15/2024 & East     &  41 & 112 & 994,915 & 0.14\% & 999,355 & 0.58\% & 996,050 & 0.25\% \\
  07/15/2024 & West     &  71 & 108 & 705,149 & 0.39\% & 706,169 & 0.58\% & 705,004 & 0.41\% \\
  08/01/2024 & East     &  41 &  95 & 2,041,898 & 0.09\% & 2,044,534 & 0.22\% & 2,042,372 & 0.11\% \\
  08/01/2024 & West     &  71 & 105 & 829,768 & 0.46\% & 830,446 & 0.54\% & 829,571 & 0.44\% \\
  08/15/2024 & East     &  41 &  75 & 1,644,491 & 0.07\% & 1,646,108 & 0.14\% & 1,645,018 & 0.07\% \\
  08/15/2024 & West     &  71 &  97 & --- & --- & 1,018,237 & 0.30\% & 1,016,722 & 0.15\% \\
  \bottomrule
  \end{tabular}
\end{table}

The first observation is that the branch-and-price-and-cut methodology significantly outperforms the polynomial formulation. With the linear fire model of \cite{wei2015chance}, the polynomial formulation leaves a large optimality gap in medium instances and does not find a feasible solution in the largest ones. With the data-driven fire model, the arc-based polynomial benchmark fails to find a feasible solution within the one-hour time limit on 6 of the 20 instances, whereas the branch-and-price-and-cut algorithm solves all instances within an optimality gap of up to 1.2\%. Next, the branch-and-price-and-cut methodology is also critical to retrieve high-quality solutions and tight optimality gaps, as compared to the simpler column generation benchmarks. The 2CG + IP heuristic achieves poor performance due to non-binary fire demands, leading to a fractional solution that combines poor-quality integer suppression plans (as discussed in Section~\ref{section:cutting planes}). The upper-bounding heuristic enables much stronger solutions by generating higher-quality integer suppression plans. Still, the heuristic is not sufficient to achieve the highest-quality solutions, and the full branch-and-price-and-cut algorithm results in significant improvements in optimality gaps, down to 2.29--6.09\% with the linear fire model and to 0.07--1.15\% with the data-driven fire model.

These results demonstrate that our methodology quickly derives good solutions through two-sided column generation, A-GUB cuts, and the upper-bounding heuristic, while the full branch-and-price-and-cut algorithm yields high-quality solutions with strong optimality bounds. Ultimately, the methodology can scale to large-scale instances arising in practice to support intra-GACC resource allocation as well as inter-GACC coordination across the full United States.

\subsection{Performance of overall predictive and prescriptive methodology}
\label{subsec:benefits}

We now compare the performance of the optimization solution to four greedy priority-based baselines that assign available crews to wildfires based on easily accessible scoring mechanisms: (i) random assignment, which allocates crews uniformly at random; (ii) a distance-based baseline, which prioritizes closer fires based on their proximity to available crews; (iii) an area-based baseline, which prioritizes larger fires based on their current area; and (iv) an impact-based baseline, which prioritizes fires based on the expected burned area reduction from the next crew. The distance-based baseline takes a routing perspective but ignores wildfire impact. As discussed in Figure~\ref{fig:SHAPgrouped}, area is the primary determinant of historical crew assignments, which motivates the area-based baseline as a practical benchmark. Finally, the impact-based baseline incorporates the DML predictions over a single time period. In comparison, the optimization solution incorporates the longer-term dynamics of the fires from the DML predictions as well as crew routing operations. For comparability, all solutions use the same travel-time discretization, crew-rest constraints, and fire-state updates.

Table~\ref{tab:summary} shows summary results of the optimization algorithm and all baselines, across all GACCs and six 14-day time periods in each summer season (detailed results are in Tables~\ref{tab:results_2023} and~\ref{tab:results_2024}). Results are reported relative to a no-crew baseline under the learned fire-spread model, both in absolute and relative terms, and relative to the optimization solution. The table considers both the intra-GACC-allocation instances, which mirrors current operations, and the inter-GACC-coordination instances, which estimate the impact of nation-wide resource sharing at the NICC level. Further description of the experiments and baselines is available in Appendix~\ref{app:optimizationresults}.

\begin{table}[h!]
\centering
\caption{Performance assessment across all instances in 2023--2024.}
\label{tab:summary}
\small
\begin{threeparttable}
\begin{tabular}{llcccccc}
\toprule
&&\multicolumn{3}{c}{Intra-GACC allocation}&\multicolumn{3}{c}{Inter-GACC coordination}\\\cmidrule(lr){3-5}\cmidrule(lr){6-8}
Year & Method & Acres saved & \% saved & Multiplier & Acres saved & \% saved & Multiplier\\
\midrule
2023 & Random assignment & 42,569 & 1.24\% & 3.48 & 61,369 & 1.79\% & 3.00 \\
 & Distance-based baseline & 64,572 & 1.88\% & 2.29 & 90,054 & 2.62\% & 2.04 \\
 & Area-based baseline & 51,894 & 1.51\% & 2.86 & 65,559 & 1.91\% & 2.81 \\
 & Impact-based baseline & 71,620 & 2.09\% & 2.07 & 106,817 & 3.11\% & 1.72 \\
 & Optimization & \textbf{148,185} & 4.31\% & 1.00 & \textbf{184,115} & 5.36\% & 1.00 \\
\midrule
2024 & Random assignment & 52,709 & 0.59\% & 3.49 & 64,884 & 0.73\% & 3.02 \\
 & Distance-based baseline & 87,486 & 0.98\% & 2.10 & 104,233 & 1.17\% & 1.88 \\
 & Area-based baseline & 55,011 & 0.62\% & 3.34 & 58,415 & 0.66\% & 3.35 \\
 & Impact-based baseline & 98,515 & 1.11\% & 1.87 & 119,383 & 1.34\% & 1.64 \\
 & Optimization & \textbf{183,924} & 2.07\% & 1.00 & \textbf{195,976} & 2.21\% & 1.00 \\
\bottomrule
\end{tabular}
\begin{tablenotes}
\item Acres saved relative to a no-crew baseline. Multiplier: acres saved by the optimization solution divided by acres saved by the corresponding benchmark, indicating the benefits of optimization.
\end{tablenotes}
\end{threeparttable}
\end{table}

These results indicate that, under the learned fire-spread model, the optimization methodology can yield significant savings as compared to benchmark allocation policies. The optimized solutions are predicted to save an estimated $332,109$ acres combined in 2023 and 2024 relative to a no-crew baseline. This amounts to more than a twofold improvement from the distance-based and area-based benchmarks. These results confirm that wildfire suppression is not solely based on routing objectives but also on triage objectives to prioritize higher-risk fires. The results also show that current burned area is an imperfect proxy of future risk by ignoring wildfire dynamics and environmental factors. In comparison, the impact-based baseline enhances the effectiveness of suppression activities by incorporating one-step predictions from the DML model. Nonetheless, the optimization solution achieves significant improvements, estimated at 87--107\%; these benefits are also consistent across GACCs and over time (see Tables~\ref{tab:results_2023} and~\ref{tab:results_2024}). In terms of magnitude, these improvements amount to an estimated 80,000 acres saved per summer---about 1.6x Acadia National Park or 60,000 football fields. While these absolute numbers depend on the underlying fire spread model, the relative benefits suggest that the optimization of crew assignments can significantly enhance the effectiveness of suppression operations.

Figure~\ref{fig:fire_and_crew} visualizes the cumulative burned area---again, as predicted by the DML model---alongside crew assignments across the 7 wildfires within the Northern California GACC from July 1 to July 15, 2024. As expected, it is generally beneficial to engage fires early to contain their spread, as seen with Thompson and Bogus. Still, the solution illustrates the complexities of allocating crews across simultaneous fires and over time, by prioritizing some fires while others receive fewer or no crews, and by re-allocating crews to work on high-impact fires at key moments (e.g. Grande on day 8). Altogether, this visualization demonstrates how the model can aid real-world decision making by guiding which fires to prioritize and when.

\begin{figure}[ht]
    \centering
    \begin{subfigure}[t]{0.48\linewidth}
        \centering
        \includegraphics[width=\linewidth]{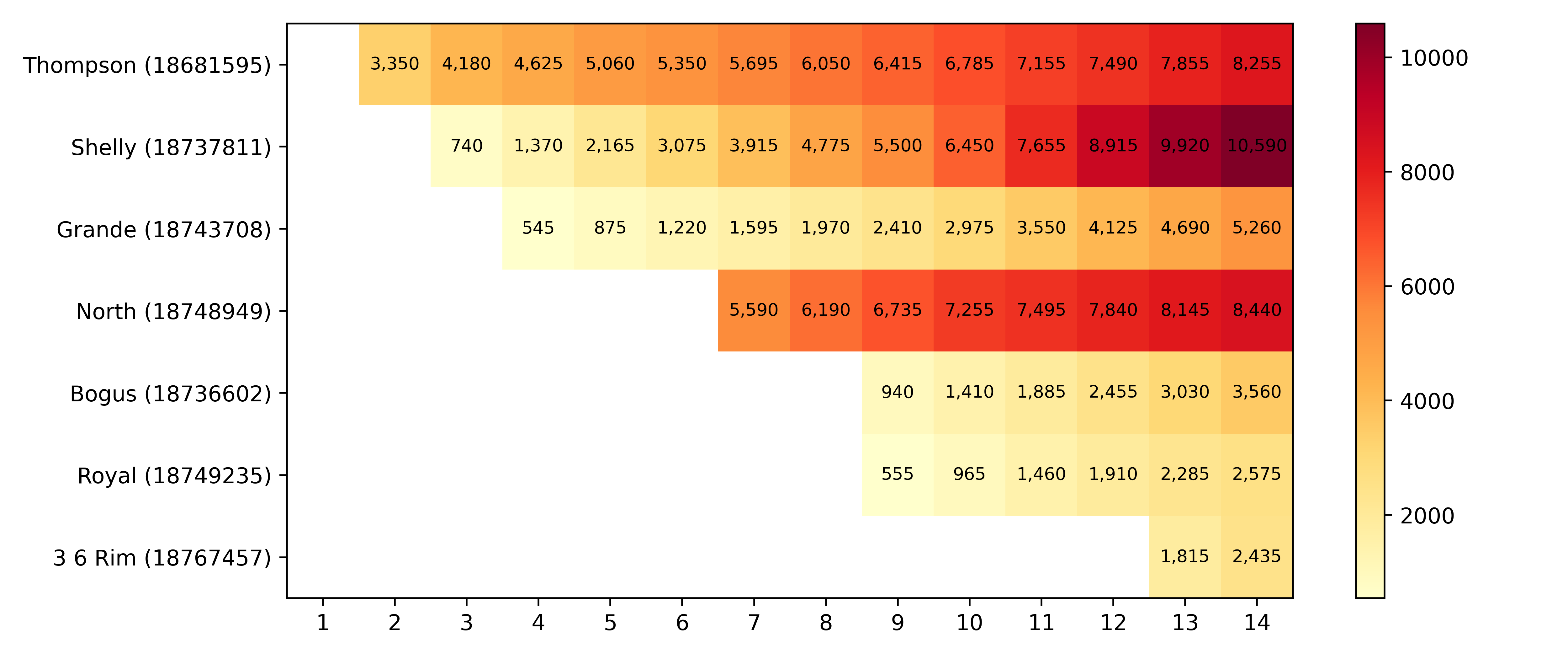}
        \caption{Burned area (acres) per fire per day.}
        \label{fig:fire_progression}
    \end{subfigure}
    \hfill
    \begin{subfigure}[t]{0.48\linewidth}
        \centering
        \includegraphics[width=\linewidth]{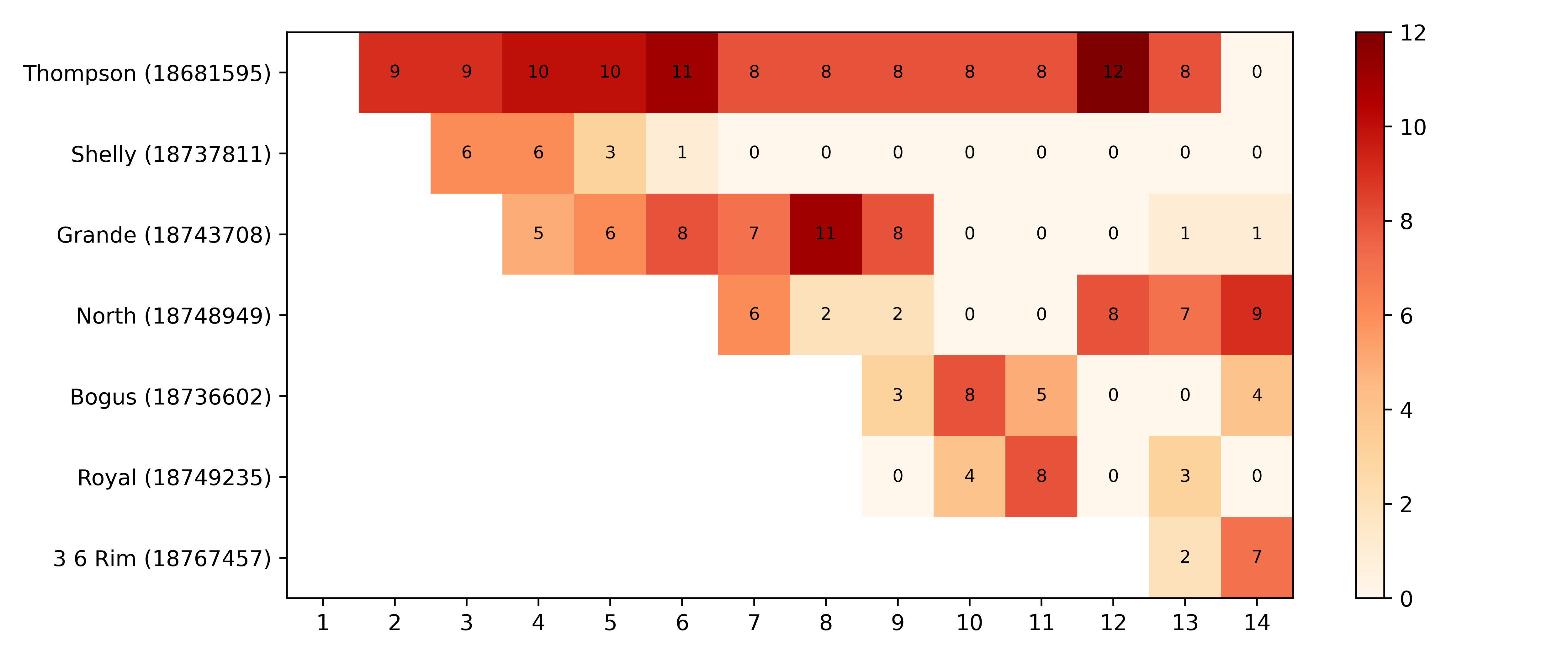}
        \caption{Crew assignments per fire per day.}
        \label{fig:crew_allocation}
    \end{subfigure}
    \caption{Fire progression (cumulative burned area) and crew allocation (number of crew assignments) for the Northern California GACC between July 1 and July 15, 2024 (7 concurrent fires, 21 available crews).}
    \label{fig:fire_and_crew}
    \vspace{-12pt}
\end{figure}

Finally, Table~\ref{tab:summary} identifies the additional benefits of inter-GACC coordination. As expected, the benefits of optimization remain robust in the presence of inter-GACC coordination; the optimized solution saves an estimated 153,891 acres across the two summers over the impact-based baseline, a 68\% improvement. Then, inter-GACC coordination results in an additional 47,982 acres saved over the intra-GACC resource allocation solution, a 15\% improvement. Whereas resource sharing across GACCs can be challenging to implement effectively in practice, these results highlight the efficiency benefits that could be achieved through centralized resource allocation. As mentioned in Section~\ref{subsec:scalability}, these benefits are enabled by the branch-and-price-and-cut methodology developed in this paper to solve large-scale instances arising from the nation-wide allocation problem.

\section{Conclusion}

Frequent and severe wildfire seasons require complex prioritization decisions to maximize the effectiveness of strained firefighting resources. In response, this paper developed an optimization methodology to support wildfire suppression operations over a dispersed geographic region. From a predictive standpoint, we developed a data-driven model of fire spread as a function of crew assignments and covariate information, using double machine learning. From a prescriptive standpoint, we formulated a mixed-integer optimization model that combines crew routing and wildfire triage, based on a general-purpose and non-linear fire spread model. The problem is formulated as a two-sided set partitioning model on time-state networks for wildfire dynamics and time-space-rest networks for crew assignments, with linking constraints between fire demand and crew supply. It is solved with an exact branch-and-price-and-cut algorithm, using (i) a two-sided column generation algorithm that generates fire suppression plans and crew routes iteratively; (ii) new ``augmented GUB cover cuts'' to tighten the problem's linear relaxation; and (iii) a novel dual-aware maximum-variance branching rule that guarantees finite convergence to an optimal integer solution.

Using real-world data, we have shown that the methodology developed in this paper can handle realistic problem instances characteristic of peak fire seasons in practice at the scale of the full contiguous United States. Notably, the branch-and-price-and-cut algorithm derives good solutions quickly and yields provably optimal or near-optimal solutions in large-scale instances when benchmark methods fail to even return a feasible solution. Combined with our data-driven fire spread model, the methodology can generate significant improvements in wildfire suppression effectiveness, roughly doubling the acres saved relative to the strongest practical benchmark over the summer seasons, and corresponding to an order of magnitude of 100,000 additional acres saved. Ultimately, this paper provides an effective prescriptive methodology for building and leveraging fire-spread models to strengthen wildfire suppression operations during intense, synchronous activity.

These promising results motivate further research in prescriptive wildfire management. Opportunities arise in identifying practical requirements from fire managers to translate the solutions of our algorithm into collaborative decision-making tools, and in leveraging recent developments in artificial intelligence to embed new predictive models of fire behavior into the proposed prescriptive optimization framework. Other interesting research questions involve adding robustness to wildfire suppression plans under uncertainty in fire spread and fire ignition. Finally, the tactical resource allocation problem can be integrated into upstream strategic planning problems or downstream operational response problems. The modeling and algorithmic approach to multi-fire suppression optimization developed in this paper provides methodological foundations to tackle this increasingly relevant class of problems at the core of climate change adaptation.

%

\section*{Acknowledgments}

This material is based upon work supported by the National Science Foundation Graduate Research Fellowship under Grant No. 2141064. The guidance from Erin Belval and Matthew Thompson from the US Forest Service is gratefully acknowledged. The authors thank Erwin Deng and Benjamin Rio for their valuable contributions to data extraction and analysis.

\bibliographystyle{informs2014}
\bibliography{sources}

\newpage

\begin{APPENDICES}

\ECSwitch
\ECHead{Electronic Companion}

\section{Discussion and proofs from Section~\ref{sec:model}}

\subsection{Table of notations}

\begin{table}[h!]
\centering\small\renewcommand{\arraystretch}{1.0}
\begin{tabular}{lll}\toprule\toprule
$\calT$ &  Set & Time periods, indexed by $t\in\calT=\{1, \ldots, T\}$, with terminal period $T+1$ \\
$\calJ$ &  Set & Crews, indexed by $j\in\calJ$ (we use $J=|\calJ|$) \\
$\calG$ &  Set & Fires, indexed by $g\in\calG$ (we use $G=|\calG|$) \\
$\calL_j$ &  Set & Locations for crew $j\in\calJ$ (fires $\calG_j$ in its jurisdiction and home base $b_j$), indexed by $\ell\in\calL_j$ \\
$\calS_g$ & Set & State space for fire $g\in\calG$ \\
$\calN_j$ &  Set & Nodes in time-space-rest network of crew $j\in\calJ$, indexed by $n\in\calN_j$ \\
$\calN_g$ &  Set & Nodes in time-state network of fire $g\in\calG$, indexed by $n\in\calN_g$ \\
$\calA_j$ &  Set & Arcs in time-space-rest network of crew $j\in\calJ$, indexed by $a\in\calA_j$ \\
$\calA_j^{gt}$ & Set & Arcs in time-space-rest network of crew $j\in\calJ$ suppressing fire $g\in\calG$ at time $t\in\calT$ \\
$\calB_g$ &  Set & Arcs in time-state network of fire $g\in\calG$, indexed by $a\in\calB_g$ \\
$\calB_g^{t}$ & Set & Arcs in time-state network of fire $g$ that require crew presence at time $t$\\
$\calP_j$ &  Set & Routes for crew $j\in\calJ$, indexed by $p\in\calP_j$ \\
$\calQ_g$ &  Set & Suppression plans for fire $g\in\calG$, indexed by $q\in\calQ_g$ \\
$d_a$ & Parameter & Cost of traversing arc $a\in\calB_g$ in time-state network of fire $g\in\calG$\\
$d_q$ & Parameter & Cost of fire suppression plan $q\in\calQ_g$, given by $d_q= \sum_{a\in \calB^q_g}d_a$\\
$c_a$ & Parameter & Cost of traversing arc $a\in\calA_j$ in time-space-rest network of crew $j\in\calJ$\\
$c_p$ & Parameter & Cost of crew route $p\in\calP_j$, given by $c_p = \sum_{a\in \calA^p_j}c_a$\\
$x_a$ & Parameter & Crews required to traverse arc $a\in\calB_g$ in time-state network of fire $g\in\calG$\\
$A_{pgt}$ & Parameter & Assignment of crew $j\in\calJ$ to fire $g\in\calG$ at time $t\in\calT$ in route $p\in\calP_j$: $A_{pgt}=\mathbf{1}(\calA^p_j\cap\A_{j}^{gt}\neq\emptyset)$\\
$B_{qt}$ & Parameter & Number of crews working on fire $g\in\calG$ at time $t\in\calT$ in plan $q\in\calQ_g$: $B_{qt}=\sum_{a\in\calB^q_g\cap\calB^t_g}x_a$\\
\bottomrule\bottomrule
\end{tabular}
\caption{Input data.}
\label{tab:full notation}
\vspace{-6pt}
\end{table}

\subsection{Discussion of Remark~\ref{rem:rest}}
\label{app:rest}

We discuss how to construct time-space-rest networks for each crew using $1 + |\calL_j|\times T\times(\varphi+1)$ nodes when Assumption~\ref{assumption:rest} does not hold.

We define a time-space-rest network similar to the one presented in Section~\ref{section:crew_aasignments}. The only change is that the third index of the node set no longer lies in $\{0,1\}$ but in $\{0,\ldots \varphi\}$. That is, we define $\calN_j = (\cal L_j \times \{2,\cdots,T+1\}\times\{0,\ldots \varphi\}) \,\cup\, (\ell^1_j,1, \min(\varphi, r^0_j))$; each node $n\in\calN_j$ track the location $\lambda(n)\in\cal L_j$, the time stamp $\tau\in\calT$, and an integer $\nu(n)$ indicating the number of time periods since the crew's last rest, capped at $\varphi$. The additional node $(\ell^1_j,1, \min(\varphi, r^0_j))$ defines the starting point of the crew. Define the set of directed arcs $\calA_j\subseteq \calN_j^2$  similarly as in Section~\ref{section:crew_aasignments}, this time tracking the number of time periods since the last rest: \begin{itemize}
    \item[--] \textit{Traveling arcs:} For all $\ell_1\neq \ell_2 \in \calL_j$ and $r\in \{0, \ldots \varphi\}$, an arc from $n_1=(\ell_1, t_1, r)$ to $n_2=(\ell_2, t_2, \min(\varphi, r + t_2 - t_1))$ is in $\calA_j$ if a crew at location $\ell_1$ in period $t_1$ can reach location $\ell_2$ by period $t_2$, i.e., if $\lceil tt(\ell_1,\ell_2)\rceil_{\calT}=\tau(n_2)-\tau(n_1)$.
    \item[--] \textit{Working arcs:} For all fire locations $g\in\calG_j$, time periods $t\in\{1,\cdots,T-1\}$, and $r\in \{0, \ldots \varphi\}$, an arc linking $(g, t, r)$ to $(g, t+1, \min(\varphi, r+1))$ is in $\calA_j$ to characterize suppression activities.
    \item[--] \textit{Resting arcs:} For all $r\in \{0, \ldots \varphi\}$, an arc from $n_1=(b_j,t,r)$ to $n_2=(b_j,t',0)$ indicates rest if $t'- t + 1\geq \gamma$ and an arc linking $(b_j, t, r)$ to $(b_j, t+1, \min(\varphi, r+1))$ indicates time at base that is too short to count as rest.
\end{itemize}

Again, we impose rest requirements by removing traveling and working arcs that violate them.

Suppose now that Assumption~\ref{assumption:rest} holds and consider any directed path $n_0n_1\cdots n_k$ in the new network with $n_0 = (\ell^1_j,0, \min(\varphi, -r^0_j))$ and $\tau(n_k) = T$. For any non-resting arc $(n_i, n_j)$, $\nu(n_j) = \min(\varphi, \nu(n_i) + \tau(n_j) - \tau(n_i))$. Therefore, for all $\ell \in \{1, \ldots, k\}$, if the crew has not completed a rest period by the time it reaches $n_\ell$, then $\nu(n_\ell) = \min(\varphi, \tau(n_\ell) -r_j^0)$. Since $R_j = \min(T, r_j^0 + \varphi)$ by definition, we have $\nu(n_\ell) \ge \varphi$ if and only if $r_j^0 - r_j^0 + \tau(n_\ell) = \tau(n_\ell) \ge R_j$, so the criterion for incurring the rest violation penalty is the same as in the network from Section~\ref{section:crew_aasignments}. In contrast, if the crew has completed a rest period by the time it enters $n_\ell$, then there is some nonnegative integer $\ell'\le \ell$ such that $\nu(n_{l'}) = 0$ and $\tau(n_{l'}) = \gamma$, so $\nu(n_\ell) = \min(\varphi, \tau(n_\ell) - \tau(n_\ell')) \le \min(\varphi, \tau(n_\ell) - \gamma) \le T - \gamma < \varphi$, and no rest violation is incurred in either network. Therefore, any path in the new network can be mapped into an equivalent path in the network from Section~\ref{section:crew_aasignments} under Assumption~\ref{assumption:rest}.


Note that if Assumption~\ref{assumption:rest} does not hold, then in the directed path $n_0n_1\cdots n_k$, if $(n_0, n_1)$ is the only rest period in the path, then $\nu(n_{k}) = \nu(n_1) + \tau(n_{k}) - \tau(n_1) = 0 + \tau(n_{k}) - \gamma = T - \gamma \ge \varphi$, so a rest violation is incurred. However, the corresponding path in the network from Section~\ref{section:crew_aasignments} has $\nu(n_k) = 1$ due to the rest period $(n_0, n_1)$ and remains feasible.

\subsection{Benchmark model: polynomial formulation in time-expanded networks}
\label{app:arc}

We formulate in this section the model in time-expanded networks with a polynomial number of arc-based variables. Note that the number of variables grows polynomially with the size of the time-expanded networks, although the time-state fire networks may themselves involve an exponential number of state---or at least a large number of states, in practice.

We define the following decision variables:
\begin{align*}
    z_a&=\begin{cases}
        1&\text{if crew $j\in\calJ$ traverses arc $a\in\calA_j$}\\
        0&\text{otherwise.}
    \end{cases}\\
 y_a&=\begin{cases}
        1&\text{if fire $g\in\calG$ traverses arc $a\in\calB_g$}\\
        0&\text{otherwise.}
    \end{cases}
\end{align*}

The problem is formulated as follows. Objective~\eqref{obj::full natural} minimizes the total cost of crew operations and fire damage. Constraints~\eqref{full natural crew flow balance}--\eqref{full natural crew start} enforce flow balance in the time-space-rest network of each crew. Constraints~\eqref{full natural fire flow balance}--\eqref{full natural fire start} enforce flow balance in the time-state network of each fire. Constraints~\eqref{full natural linking} is the fire-crew linking constraint; it ensures that at each time, a sufficient number of crews are routed to each fire on their time-space-rest networks to allow the corresponding transition in the time-state network (as specified by the parameters $x_a$). Constraints~\eqref{full crew arc int}--\eqref{full fire arc int} define the domain of all variables.
\begin{mini!}
    {}{\sum_{j\in \calJ}\sum_{a\in \A_j}c_az_a + \sum_{g\in \calG}\sum_{a\in \B_g}d_ay_a}{\label{model::full natural}}{} \label{obj::full natural}
    \addConstraint{\sum_{a\in 
\delta^+(n)}z_a - \sum_{a\in 
\delta^-(n)}z_a = 0, \, \forall j\in \calJ, \forall n\in \calN_j \text{ s.t. } 0 < \tau(n) < T} \label{full natural crew flow balance}
    \addConstraint{\sum_{a\in 
\delta^-((\ell^1_j,1,0))}z_a = 1, \, \forall j\in \calJ} \label{full natural crew start}
    \addConstraint{\sum_{a\in 
\delta_{g}^+(n)}y_{a} - \sum_{a\in 
\delta_{g}^-(n)}y_{a} = 0, \, \forall g\in \calG, \forall n\in \calN_g \text{ s.t. } 0 < \tau(n) < T} \label{full natural fire flow balance}
    \addConstraint{\sum_{a\in 
\delta^-((s^1_g,1))}y_{a} = 1, \, \forall g\in \calG} \label{full natural fire start}
    \addConstraint{\sum_{j\in \calJ}\sum_{a\in 
\A_{j}^{gt}}z_a - \sum_{a\in 
\B_{g}^t}x_ay_{a} \ge 0, \, \forall g\in \calG, \forall t\in \calT}\label{full natural linking}
\addConstraint{z_a \in \{0, 1\}, \,\forall j \in \calJ, \forall a\in \A_j} \label{full crew arc int}
\addConstraint{y_a \in \{0, 1\}, \,\forall g \in \calG, \forall a\in \calB_g.} \label{full fire arc int}
\end{mini!}

By construction, there is a one-to-one mapping between the solution space of Problems~\eqref{set partition} and~\eqref{model::full natural}, using the transformations $A_{p_jgt} = \sum\limits_{a\in \A_{j}^{gt}}z_a$, $B_{q_gt} = \sum\limits_{a\in \B_{g}^{t}}x_ay_a$, $c_p=\sum_{a\in \calA^p_j}c_a=\sum_{a\in \calA_j}c_az_a$, and $d_q = \sum_{a\in \calB^q_g}d_a = \sum_{a\in \calB_g}d_ay_a$. Therefore, both formulations yield the same integer solution and define the same linear relaxation. This is stated in the proposition below.

\begin{proposition} \label{prop:formulation equiv}
    Problem~\eqref{model::full natural} and Problem~\eqref{set partition} are equivalent.
\end{proposition}

\subsection{Properties of fire plans and crew routes}

The following lemmas identify important properties of the crew routes $\calP$ and fire suppression plans $\calQ$. Specifically, Lemma~\ref{lem crew route} shows that any crew route attends to up to one fire at a time, and Lemma~\ref{lem fire plan} shows that any fire suppression plan uses no more crews than the maximum arc requirement, which, by construction, is at most the total number of crews.

\begin{lemma} \label{lem crew route}
    For all crews $j\in\calJ$, routes $p\in \calP_j$, fires $g\in\calG$ and times $t\in\calT$, $A_{pgt} \in\{0, 1\}$. Moreover, $\sum\limits_{g\in\calG}A_{pgt}\in \{0, 1\}$ for all $p\in \calP_j$, $j\in\calJ$, and $t\in\calT$.
\end{lemma}

\paragraph{Proof of Lemma~\ref{lem crew route}.}
Consider a solution $(\by^*,\bz^*)$ to Problem~\eqref{model::full natural}. Define the directed network flow graph on $\calA_j$ with arc flows given by $z^*_a$. The flow-balance constraints~\eqref{full natural crew flow balance} and~\eqref{full natural crew start} ensure that this graph has a single connected component originating at $(\ell^1_j,1,0)$ with one unit of flow. Since no arc can exist from $(\ell_1, t_1, r_1)$ to $(\ell_2, t_2, r_2)$ if $t_2 \le t_1$, flow can only enter the set of nodes at time $t\in\calT$ once. This means that $\sum\limits_{g\in\calG}\sum\limits_{a\in\calA_{j}^{gt}}z^*_a \le 1$, and the binary constraints~\eqref{full crew arc int} give $\sum\limits_{g\in\calG}\sum\limits_{a\in\calA_{j}^{gt}}z^*_a \in \{0, 1\}$, so the corresponding route $p\in\calP_j$ has $\sum\limits_{g\in\calG}A_{pgt} \in \{0, 1\}$, as desired. Because for any fire $g^*\in\calG$, $\calA_{j}^{g^*t} \subseteq \bigcup\limits_{g\in\calG} \calA_{j}^{gt}$, a similar reasoning shows that for all $g\in\calG$, $t\in\calT$, $A_{pgt} \in\{0, 1\}$.
\hfill\Halmos

\begin{lemma} \label{lem fire plan}
    For all fires $g\in\calG$, plans $q\in \calQ_g$ and times $t\in\calT$, $B_{qt} \in\left\{0, \ldots, \max\limits_{a\in \B_{g}^t }x_a\right\}$.
\end{lemma}

\paragraph{Proof of Lemma~\ref{lem fire plan}.}
Consider a solution $(\by^*,\bz^*)$ to Problem~\eqref{model::full natural}. Define the directed network flow graph on $\calB_g$ with flows given by $y^*_a$. The flow-balance constraints~\ref{full natural fire flow balance} and~\ref{full natural fire start} ensure that this graph has a single connected component originating at $(s^1_g,1)$ with one unit of flow. Since no arc can exist from $(\ell_1, t_1)$ to $(\ell_2, t_2)$ if $t_2 \le t_1$, flow can only enter the set of nodes at time $t\in\calT$ once. This means that $\sum\limits_{a\in\calB_{g}^{t}}x_ay^*_a \le \max\limits_{a\in\calB_{g}^{t}}x_a$, so the corresponding plan $q\in\calQ_g$ has $B_{qt} \in\{0, \ldots, \max\limits_{a\in \B_{g}^t }x_a\}$, as desired.
\hfill\Halmos

\section{Algorithms and proofs from Section~\ref{sec:algorithm}}

\subsection{Shortest path algorithms for subproblems}\label{app:SP}

Algorithm~\ref{alg:sp} presents the generic label-setting algorithm for the subproblem, leveraging the shortest path structure and the topological ordering of arcs in the time-expanded networks. Algorithms~\ref{alg:crew_sp} and~\ref{alg:fire_sp} apply this procedure to the crew subproblems and the fire subproblems, respectively.

\begin{algorithm} [h!]
\caption{\textsc{TopologicalShortestPath}.}\small
\label{alg:sp}
\begin{algorithmic}
\item \textbf{Inputs:} Node set $\calV$, arc set $\calE$, arc costs $\{c(e):e\in\calE\}$, arc travel times $\{t(e):e\in\calE\}$, starting node $v_{start}$
\item \textbf{Requirement:} All arcs travel forward in time; that is, for all $v_1, v_2\in V$, if $t(v_1) \ge t(v_2)$, then $(v_1, v_2) \not\in E$.
\item \textbf{Requirement:} The starting vertex is alone at the earliest time; that is, $t(v_{start}) < t(v)$ for all $v\in \calV\setminus v_{start}$.
\item \textbf{Initialization:} length $\ell(v)$ of shortest path from $v_{start}$ to $v$: $\ell(v_{start}) \gets 0$ and $\ell(v) \gets \infty,\forall v\in V\setminus v_{start}$.
\item For each time period $t \in \{t(v) : v\in V\}$ execute the following step in ascending time order.
\begin{itemize}
    \item[] \textbf{Update step:} Consider all vertices $v$ with $t(v) = t$. For each one, update $\ell(v)=\min\{\ell(u)+c(e):e=(u,v)\in\delta^+(v)\}$; store an arc that achieves the minimum (if finite).
\end{itemize}
\item \textbf{Reconstruction step:} among all termination vertices, find $v_{min}$ with lowest cost. If the cost is $\infty$, return $\infty$. Otherwise, return $\ell(v_{min})$ and the path from $v_{start}$ using the stored arcs.
\end{algorithmic}
\end{algorithm}

\begin{algorithm} [h!]
\caption{$\crewsp$.}\small
\label{alg:crew_sp}
\begin{algorithmic}
\item \item \textbf{Inputs:} Crew $\in\calJ$, dual costs $\brho$
\item \textbf{Step 1:} Adjust arc costs: for each $g\in \calG$ and $t\in \calT$, subtract $\rho_{gt}$ from the fixed cost of each arc in $\calA_j^{gt}$.
\item \textbf{Step 2:} Find a shortest path $p$ from the starting state and its cost $c_p$ using Algorithm~\ref{alg:sp}.
\item  \textbf{Step 3:} Determine $A_{pgt}$ for each fire at each time. Return this matrix and the route cost $c_p$.
\end{algorithmic}
\end{algorithm}

\begin{algorithm} [h!]
\caption{$\firesp$.}\small
\label{alg:fire_sp}
\begin{algorithmic}
\item \item \textbf{Inputs:} Fire $\in\calG$, dual costs $\rho_g$
\item \textbf{Step 1:} Adjust arc costs: for each $t\in \calT$, add $\rho_{gt}$ to the fixed cost of each arc in $\calB^t_g$.
\item \textbf{Step 2:} Find a shortest path $q$ from the starting state and its cost $d_q$ using Algorithm~\ref{alg:sp}.
\item  \textbf{Step 3:} Determine $B_{qt}$ for each time. Return this vector and the suppression plan cost $d_q$.
\end{algorithmic}
\end{algorithm}

\subsection{Proof of Proposition~\ref{prop:deferral}.}

It suffices to show that for any feasible solution $(\boldsymbol{y}, \boldsymbol{z}, \boldsymbol{\delta})$, there is a feasible solution of equal or lesser cost with $\delta_{gt} = 0$ for all $g\in \calG, t\in \calT$. First, note that $(\boldsymbol{y}, \boldsymbol{z}, \lfloor\boldsymbol{\delta}\rfloor)$ is feasible because the integrality of $A_{pgt}$ and $B_{qt}$ imply that the inequalities defined by~\eqref{link set partition new} can be rounded down to $\sum\limits_{p \in \calP_j} A_{pgt}z_{p}  -\sum\limits_{q \in \calQ_g} B_{qt}y_q   + \lfloor\delta_{g,t-1} - \delta_{gt}\rfloor \ge 0$, and a general property of $\lfloor\cdot\rfloor$ gives $\lfloor\delta_{g,t-1} - \delta_{gt}\rfloor \le \lfloor\delta_{g,t-1}\rfloor - \lfloor\delta_{gt}\rfloor$. Since $\boldsymbol{\delta}$ does not enter into the objective function, this solution has the same cost.

Now let us fix $g\in\calG$, consider the plan $q^*\in \calQ_{g}$ such that $y_{q^*} = 1$, and define
$$x_t = \begin{cases}
    B_{q^*t} + \lfloor \delta_{gt}\rfloor - \lfloor\delta_{g,t-1}\rfloor & \forall t\in\{1, \ldots, T-1\} \\
    B_{q^*T} - \lfloor\delta_{g,T-1}\rfloor & \text{for $t=T$}
\end{cases}$$

Per the feasibility of $(\boldsymbol{y}, \boldsymbol{z}, \lfloor\boldsymbol{\delta}\rfloor)$ and Lemma~\ref{lem crew route}, we have $x_t \le \sum\limits_{j\in \calJ} \sum\limits_{p \in \calP_j} A_{pgt}z_{p} \le J$ for all $t\in\calT$. Thus, $x_t \in \{0, \ldots, J\}$, and there exists $q'\in \calQ_g$ with $B_{q',t} = x_t$ for all $t\in \calT$.

We claim that the following solution satisfies constraint~\eqref{link set partition new} with equal or lesser cost:
\begin{align*}
	\boldsymbol{z'} &= \boldsymbol{z}\\
    y'_q &= \begin{cases*}
    y_q & \text{if }$q \in \calQ\setminus\{q^*, q'\}$ \\
    1 & \text{if }$q = q'$ \\
    0 & \text{if }$q = q^*$
\end{cases*} \forall q\in\calQ\\
\delta'_{{h}t}& = \begin{cases*}
                0 & \text{if }$h = g$ \\
    \delta_{{h}t} & otherwise
    \end{cases*} \forall h\in\cal G, t\in \cal T
\end{align*}

Feasibility again follows from the feasibility of  $(\boldsymbol{y}, \boldsymbol{z}, \lfloor\boldsymbol{\delta}\rfloor)$ and the fact that $\lfloor \delta_{gT}\rfloor \ge 0$. Indeed:
\begin{itemize}
    \item[--] For $h\neq g\in\calG$ and $t\leq T$:
\begin{align*}
    \sum\limits_{j\in \calJ} \sum\limits_{p \in \calP_j} A_{pht}z'_{p}  -\sum\limits_{q \in \calQ_h} B_{qt}y'_q   + \lfloor\delta'_{h,t-1}\rfloor - \lfloor\delta'_{ht}\rfloor &= \sum\limits_{j\in \calJ} \sum\limits_{p \in \calP_j} A_{pht}z_{p}  -\sum\limits_{q \in \calQ_h} B_{qt}y_q + \lfloor\delta_{h,t-1}\rfloor - \lfloor\delta_{ht}\rfloor &\ge 0
\end{align*}
    \item[--] For $g\in\calG$ and $t\leq T-1$:
\begin{align*}
    \sum\limits_{j\in \calJ} \sum\limits_{p \in \calP_j} A_{pgt}z'_{p}  -\sum\limits_{q \in \calQ_g} B_{qt}y'_q   + \lfloor\delta'_{g,t-1}\rfloor - \lfloor\delta'_{gt}\rfloor &= \sum\limits_{j\in \calJ} \sum\limits_{p \in \calP_j} A_{pgt}z_{p}  -\sum\limits_{q \in \calQ_g} B_{qt}y_q - B_{q't} + B_{q^*t} \\
    &= \sum\limits_{j\in \calJ} \sum\limits_{p \in \calP_j} A_{pgt}z_{p}  -\sum\limits_{q \in \calQ_g} B_{qt}y_q - x_t + B_{q^*t} \\
    &= \sum\limits_{j\in \calJ} \sum\limits_{p \in \calP_j} A_{pgt}z_{p}  -\sum\limits_{q \in \calQ_g} B_{qt}y_q + \lfloor\delta_{g,t-1}\rfloor - \lfloor\delta_{gt}\rfloor \\
    &\ge 0
\end{align*}
    \item[--] For $g\in\calG$ and $t=T$:
\begin{align*}
    \sum\limits_{j\in \calJ} \sum\limits_{p \in \calP_j} A_{pgT}z'_{p}  -\sum\limits_{q \in \calQ_g} B_{qT}y'_q   + \lfloor\delta'_{g,T-1}\rfloor - \lfloor\delta'_{gT}\rfloor &= \sum\limits_{j\in \calJ} \sum\limits_{p \in \calP_j} A_{pgT}z_{p}  -\sum\limits_{q \in \calQ_g} B_{qT}y_q - B_{q'T} + B_{q^*T} \\
    &= \sum\limits_{j\in \calJ} \sum\limits_{p \in \calP_j} A_{pgT}z_{p}  -\sum\limits_{q \in \calQ_g} B_{qT}y_q - x_T + B_{q^*T} \\
    &= \sum\limits_{j\in \calJ} \sum\limits_{p \in \calP_j} A_{pgT}z_{p}  -\sum\limits_{q \in \calQ_g} B_{qT}y_q + \lfloor\delta_{g,T-1}\rfloor \\
    &\ge \lfloor\delta_{gT}\rfloor\\
    &\ge 0
\end{align*}
\end{itemize}

It remains only to prove that $d_{q'} \le d_{q^*}$. We will prove this using the deferral-proof property. For all $q\in\calQ_g$,  for all $t\in \calT$, define $\Delta(q, t) = B_{qt} - B_{q^*t}$. For all $t^* \in \{1, \ldots, T-1\}$, telescoping gives
$$\sum_{t=1}^{t^*}\Delta(q', t) = \sum_{t=1}^{t^*}(B_{q't} - B_{q^*t}) = \sum_{t=1}^{t^*}(x_t - B_{q^*t}) =\sum\limits_{t=1}^{t^*}\left( \lfloor \delta_{gt}\rfloor - \lfloor\delta_{g,t-1}\rfloor \right)  = \lfloor\delta_{gt^*}\rfloor$$
Similarly,
$$\sum_{t=1}^{T}\Delta(q', t) = \sum\limits_{t=1}^{T-1}\left( \lfloor \delta_{gt}\rfloor - \lfloor\delta_{g,t-1}\rfloor \right)-\lfloor\delta_{g,T-1}\rfloor=\lfloor\delta_{g,T-1}\rfloor-\lfloor\delta_{g,T-1}\rfloor=0$$

We proceed by induction on $\sum\limits_{t\in\calT}|\Delta(q, t)|$ to prove that $d_{q} \le d_{q^*}$ for any plan $q\in\calQ_g$ satisfying $\sum\limits_{t=1}^{\tau}\Delta(q, t) \in \mathbb Z_+$ for all $\tau\in\calT$, and $\sum\limits_{t=1}^{T}\Delta(q, t) = 0$.

Base case: $\sum\limits_{t\in\calT}|\Delta(q, t)| = 0$. Then $\Delta(q, t) = 0$ for all $t\in\calT$, so $B_{qt} = B_{q^*t}$ for all $t\in \calT$, and $d_{q} = d_{q^*}$.

Inductive step: Suppose $\sum\limits_{t\in\calT}|\Delta(q, t)| = k>0$ and the result holds for all nonnegative integers less than $k$. Let $t^*=\min\{t\in\calT:|\Delta(q,t)|>0\}$. Since $k > 0$, $t^*$ is finite; and since $\sum\limits_{t=1}^{T}\Delta(q, t)=0$, we have $t^*<T$. Furthermore,  $\Delta(q, t^*) > 0$ as $\sum\limits_{t=1}^{t^*}\Delta(q, t) =\Delta(q, t^*)$ and $\sum\limits_{t=1}^{t^*}\Delta(q, t)\geq 0$ by assumption.  Let $t^{**}=\min\{t>t^*:\Delta(q,t)<0\}$. Again, $t^{**}$ is finite because $\sum\limits_{t=1}^{t^*}\Delta(q, t) > 0$ and $\sum\limits_{t\in\calT}\Delta(q, t) = 0$. Define:
$$x_t = \begin{cases*}
    B_{qt} & \text{for all }$t\in\calT\setminus\{t^*, t^{**}\}$ \\
   B_{qt} - 1 & \text{for }$t = t^*$ \\
   B_{qt} + 1 & \text{for }$t = t^{**}$. \\
\end{cases*}$$

Since $x_t\leq\max\{B_{qt},B_{q^*t}\}$ for all $t\in\calT$, we have $x_t\leq J$. We construct $\overline{q}\in \calQ_g$ with $B_{\overline{q}t} = x_t$ for all $t\in \calT$. Applying the definition of deferral-proofness, we get $d_q \le d_{\overline{q}}$. Furthermore:
$$\Delta(\overline{q}, t) = \begin{cases*}
    \Delta(q, t) & $t\in\calT\setminus\{t^*, t^{**}\}$ \\
   \Delta(q, t) - 1 & $t = t^*$ \\
   \Delta(q, t) + 1 & $t = t^{**}$. \\
\end{cases*}$$
so 
$$\sum\limits_{t=1}^{\tau}\Delta(\overline{q}, t)  = \begin{cases*}
   \sum\limits_{t=1}^{\tau}\Delta(q, t) - 1 & if $t^* \le \tau < t^{**}$ \\
   \sum\limits_{t=1}^{\tau}\Delta(q, t) & otherwise
\end{cases*},$$
The construction of $t^*$ and $t^{**}$ ensures that $\sum\limits_{t=1}^{\tau}\Delta(q, \tau) \ge 1$ if  $t^* \le \tau < t^{**}$. Hence, $\sum\limits_{t=1}^{\tau}\Delta(\overline{q}, t) \in \mathbb Z_+$ for all $\tau\in\calT$. Furthermore, $\sum\limits_{t=1}^{T}\Delta(\overline{q}, t) = \sum\limits_{t=1}^{T}\Delta(q, t) = 0$. Lastly, $\sum\limits_{t\in\calT}|\Delta(\overline{q}, t)| \le \sum\limits_{t\in\calT}|\Delta(q, t)| - 2$, so the induction hypothesis gives $d_{\overline{q}} \le d_{q^*}$. This proves that $d_q \le d_{q^*}$, and completes the induction.
\hfill\Halmos

\subsection{Proof of Proposition~\ref{prop:robust}}

Let $\alpha_u$ denote the dual variable of constraint~\eqref{robust_inequality}. The new subproblems become:
\begin{align*}
& \min_{q\in \calQ_g} \left(d_{q} + \sum_{t\in \calT} B_{qt}\rho_{gt} + \sum_{u\in\calU}\sum_{d=0}^J\mathbbm 1[B_{qt_u} = d]\delta_{ugd}\alpha_u-\sigma_g\right) \\
& \min_{p\in \calP_j} \left( c_{p} - \sum_{g\in \calG}\sum_{t\in \calT} A_{pgt}\rho_{gt} + \sum_{u\in\calU}\sum_{g\in \calG}\sum_{d=0}^1 \mathbbm 1\left[A_{pgt_u} = d\right]\gamma_{ugjd}\alpha_u-\pi_j\right).
\end{align*}

The cost of the duals are readily introduced into the fire and crew subproblems by modifying arc costs---in particular, they require no extra labels in the subproblem algorithms. For every fire in $g\in \calG$, every arc $a \in \calB^t_g$ should have $\delta_{ugd}\alpha_u$ added to its cost if $x_a = d$, i.e., if $d$ crews are assigned to it at the corresponding time. For every crew $j\in \calJ$ and every fire $g \in\calG$, every arc in $\calA_j^{gt}$ should have $(\gamma_{ugj1} - \gamma_{ugj0})\alpha_u$ added to its cost. Then, the total subproblem cost should have $\gamma_{ugj0}\alpha_u$ added as part of the terminal arc costs. That way, $\gamma_{ugj0}\alpha_u$ is added if crew $j$ does not suppress fire $g$ at time $t$, and $\gamma_{ugj1}\alpha_u$ is added if crew $j$ suppresses fire $g$ at time $t$.

\subsection{Proof of Theorem~\ref{thm::GUB cover}.}
    Summing the set partitioning constraints (Equations~\eqref{supp_plan} and~\eqref{crew_route}) over $\calG_u$ and $\calJ_u$ gives $$ \sum_{g\in \calG_u}\sum_{q\in \calQ_g} y_q + \sum_{j\in \calJ_u}\sum_{p\in \calP_j} z_p = |\calG_u| + |\calJ_u|.$$ 

    Suppose for the sake of contradiction that $$ \sum_{g\in \calG_u}\sum_{q\in \calQ_g} \mathbbm 1[B_{qt} \ge D_{ug}]y_q + \sum_{j\in \calJ_u}\sum_{p\in \calP_j}\mathbbm 1\left[\sum_{g\in \calG_u}A_{pgt}  =  0\right]z_p \ge |\calG_u| + |\calJ_u|.$$
    Subtracting this inequality from the equality above gives $$ \sum_{g\in \calG_u}\sum_{q\in \calQ_g} \mathbbm 1[B_{qt} < D_{ug}]y_q + \sum_{j\in \calJ_u}\sum_{p\in \calP_j}\mathbbm 1\left[\sum_{g\in \calG_u}A_{pgt}  \neq  0\right]z_p \le 0.$$ This implies that $B_{qt}\geq D_{ug}$ if $y_q=1$ for all $q\in \calQ_g$ and $g\in\calG_u$, and that $A_{pgt}=0$ if $z_p=1$ for $p\in\calP_j$, for all $j\in\calJ_u$ and $g\in\calG_u$. Therefore, $\sum\limits_{q\in \calQ_g}B_{qt}y_q \ge \sum\limits_{q\in \calQ_g}D_{ug}y_q = D_{ug}$ for all $g\in \calG_u$ and  $\sum\limits_{p\in \calP_j}\sum\limits_{g\in \calG_u}A_{pgt}z_p = 0$ for all $j\in \calJ_u$. Summing the linking constraint (Equation~\eqref{link set partition}) across all fires in $\calG_u$ gives
    \begin{align*}
        0 &\le \sum_{g\in \calG_u}\sum_{j\in\calJ}\sum_{p\in \calP_j}A_{pgt}z_p - \sum_{g\in \calG_u}\sum_{q\in \calQ_g}B_{qt}y_q \\
        &= \sum_{j\in \calJ_u}\sum_{p\in \calP_j}\sum_{g\in \calG_u}A_{pgt}z_p + \sum_{j\in \calJ\setminus \calJ_u}\sum_{p\in \calP_j}\sum_{g\in \calG_u}A_{pgt}z_p - \sum_{g\in \calG_u}\sum_{q\in \calQ_g}B_{qt}y_q \\
        &\le \sum_{j\in \calJ_u}\sum_{p\in \calP_j}\sum_{g\in \calG_u}A_{pgt}z_p + \sum_{j\in \calJ\setminus \calJ_u}\sum_{p\in \calP_j}\sum_{g\in \calG}A_{pgt}z_p - \sum_{g\in \calG_u}\sum_{q\in \calQ_g}B_{qt}y_q \\
        &\le \sum_{j\in \calJ_u}\sum_{p\in \calP_j}\sum_{g\in \calG_u}A_{pgt}z_p + \sum_{j\in \calJ\setminus \calJ_u}\sum_{p\in \calP_j}z_p - \sum_{g\in \calG_u}\sum_{q\in \calQ_g}B_{qt}y_q \\
        &\le 0 + (J - |\calJ_u|) - \sum_{g\in \calG_u} D_{ug},
    \end{align*} 
    where the first inequality uses $\calG_u\subseteq\calG$, the second inequality stems from Lemma~\ref{lem crew route}, and the last inequality leverages the results obtained above. We conclude that $\sum\limits_{g\in \calG_u} D_{ug} + |\calJ_u| \le J$, which contradicts the assumption of the theorem.

    Finally, these cuts satisfy the structure of the robust inequalities, with:
    $$\delta_{ugd} =     \begin{cases*}
        1 & $g\in \calG_u \text{ and } d\ge D_{ug}$ \\
        0 & \text{ otherwise.}
    \end{cases*}$$ and $$\gamma_{ugjd} = 
    \begin{cases*}
        1 & $g\in \calG_u, d = 0$ \\
        0 & \text{ otherwise.}
    \end{cases*}$$
This completes the proof.
\hfill\Halmos

\subsection{Details on GUB cuts and strengthening procedure}
\label{app:GUBdetails}

\paragraph{Finding GUB cover cuts.} 

Given a fractional solution $(\mathbf y^*, \mathbf z^*)$, we proceed by enumeration to find a GUB cover inequality that is valid for Problem~\eqref{set partition} and violated by $(\mathbf y^*, \mathbf z^*)$, with two minor simplifications (detailed in Algorithm~\ref{alg:enumerate_gub_cuts}). The first one is to only consider crews in $\calJ_u$ that suppress no fires in the incumbent solution at time $t$. This prevents the generation of weaker cuts due to crew symmetry by leveraging signal from the incumbent solution (e.g., a crew may need to rest at this time). The second one is to consider only fire demand allotments associated with active suppression plans in the incumbent solution. For example, if $\sum\limits_{q \in \calQ_g }\mathbbm{1}[B_{qt} = d]y^*_q = 0$ and $D_{ug} = d$, then redefining $D_{ug} = d-1$ leads to a valid (and stronger) inequality.

\begin{algorithm} [h!]
\caption{Minimal GUB cover cut enumeration.}\small
\label{alg:enumerate_gub_cuts}
\begin{algorithmic}
\item \textbf{Initialization:} $\calU\gets\emptyset$; $\calJ_u=\left\{j\in\calJ:\sum\limits_{p\in \calP_j}\sum\limits_{g\in \calG} A_{pgt}z^*_p = 0\right\}$; $\calD_g \gets \left\{d : \sum\limits_{q \in \calQ_g }\mathbbm{1}[B_{qt} = d]y^*_q > 0\right\}$ for all $g\in\calG$.
\item For each time period $t\in\calT$, complete the following procedure.
\begin{itemize}
    \item[] For each subset $\calG_u\subseteq \calG$, and for each tuple in the Cartesian product of demands $D_{ug}\in\prod\limits_{g\in \mathcal{G}_u} \calD_g$
    \begin{itemize}
    	\item[] If $\sum\limits_{g\in \calG_u} D_{ug} + |\calJ_u|> J$; and if $\sum\limits_{g\in \calG_u} D_{ug} + |\calJ_u| - \min\limits_{g\in \calG_u} D_{ug}\leq J$, and if $\sum\limits_{g\in \calG_u}\sum\limits_{q\in \calQ_g} \mathbbm 1[B_{qt} \ge D_{ug}]y^*_q  > |\calG_u|  - 1$, then $\calU\gets\calU\cup\{(t, \calG_u, \calJ_u, (D_{ug})_{g\in \mathcal{G}_u})\}$.
	\end{itemize}
\end{itemize}
\end{algorithmic}
\end{algorithm}

\paragraph{Strengthened GUB cover cuts.}
In column generation, high-dimensional faces of low-dimensional polyhedra---with a subset of variables---may not be as strong in the full-dimensional polyhedron with new variables. We propose a strengthening approach by implicitly considering omitted variables. Suppose that a GUB cover cut satisfies $\sum\limits_{g\in \calG_u} D_{ug} + |\calJ_u|> J + k$ for $k\in\Z_+$. Then we can decrease some targets $D_{ug}$ to derive a stronger cut upon replacing $D_{ug}$ by $D'_{ug}\leq D_{ug}$ in Equation~\eqref{gub_cover}, by minimizing the slack between target demand and incumbent demand:
\begin{equation}\label{strengthening}
	\min\ \left\{\max_{g\in \calG_u}\left(D'_{ug} - \sum\limits_{q\in \calQ_g}B_{qt}y^*_q\right):\sum_{g\in \calG_u} D'_{ug} \ge J - |\calJ_u| + 1;D'_{ug} \le D_{ug};\ D'_{ug} \in\Z_+,\,\forall g\in \calG_u\right\}
\end{equation}

In the 10-crew example from Section~\ref{section:cutting planes} of the paper, we can derive
    \begin{align*}
	& \sum\limits_{q\in \calQ_1} \mathbbm 1[B_{qt} \ge 2]y_q  + \sum\limits_{q\in \calQ_2} \mathbbm 1[B_{qt} \ge 9]y_q  \le 1,
    \end{align*}
which strengthens the original GUB cut
    \begin{align*}
        & \sum\limits_{q\in \calQ_1} \mathbbm 1[B_{qt} \ge 2]y_q  + \sum\limits_{q\in \calQ_2} \mathbbm 1[B_{qt} \ge 10]y_q  \le 1.
    \end{align*}

Problem~\eqref{strengthening} features a matroid structure and can therefore be solved via a greedy algorithm: as long as $\sum\limits_{g\in \calG_u} D_{ug} > J - |\calJ_u| + 1$, we select the fire with the largest difference between the target and the incumbent solution, i.e., $g^*\in\argmax\big\{D_{ug} - \sum\limits_{q\in \calQ_g}B_{qt}y^*_q:g \in \calG_u\big\}$, and update $D_{ug^*} \gets D_{ug^*} - 1$.

\begin{proposition} \label{matroid}
 $(S, I)$ is a matroid, with $\calS = \{(g, z) \in \calG_u \times \mathbb Z_+ : z \le D_{ug}\}$ and $I = \{T\subseteq S : |T| \le \sum\limits_{g\in \calG_u} D_{ug} - J + |\calJ_u| - 1\}$; and the elements of $I$ map into feasible solutions of Problem~\eqref{strengthening}.
\end{proposition}

\proof{Proof of Proposition~\ref{matroid}.}
     The validity of the original cut requires $\sum\limits_{g\in\calG_u}D_{ug} - J + |\calJ_u| - 1\ge 0$, so $\emptyset\in I$. Since the independence criterion is simply a cardinality upper bound, the augmentation and subset properties are satisfied, and $(S, I)$ is a matroid. Furthermore, the function $f : I \rightarrow \Z^{\calG_u}$ defined by $$[f(T)]_{ug} = D_{ug} - |\{(g', z) \in T : g' = g\}|$$
     maps $I$ onto the feasible solutions $\calM$. To see this, note that the definition of $S$ gives $|\{(g', z) \in S : g' = g\}| = D_{ug}$, so for any $T\subseteq S$ it holds that $[f(T)]_{ug} \ge D_{ug} - D_{ug} = 0$. In addition, the inclusion criterion for $I$ guarantees that for all $T\in I$, $$
     \sum_{g\in \calG_u}[f(T)]_{ug} = \sum_{g\in \calG_u}\left(D_{ug} - |\{(g', z) \in T : g' = g\}|\right) \geq \sum_{g\in \calG_u}D_{ug} - |T| \ge J - |\calJ_u| + 1.$$ Hence, this is a feasible solution to Problem~\eqref{strengthening}. The function $f : I \rightarrow \calM$ is surjective. Indeed, for any feasible solution $\{D'_{ug}\}_{g\in \calG_u}$, the set $T = \{(g, z) \in \calG_u \times \mathbb Z_+ : D'_{ug} < z \le D_{ug}\}$ has $[f(T)]_{ug} = D_{ug}-(D_{ug}-D'_{ug})=D'_{ug}$, and $T\in I$ because $T$ has cardinality $\sum\limits_{g\in \calG_u}D_{ug} - \sum\limits_{g\in \calG_u}D'_{ug}\leq\sum\limits_{g\in \calG_u}D_{ug} - J + |\calJ_u| - 1$ per the feasibility of $\{D'_{ug}\}_{g\in \calG_u}$. Therefore, any feasible solution can be obtained from an element $T$ of the matroid $(S, I)$.
     \hfill\Halmos
     
\subsection{Proof of Theorem~\ref{thm:general_gub_aware_cut}.}
Consider a feasible solution $(\mathbf y, \mathbf z)$ to Problem~\eqref{set partition}. Let $N = \sum\limits_{g\in \calG_u}\sum\limits_{q\in \calQ_g}B_{qt}y_q$, and $M = \max(0, N - J + |\calJ_u|)\leq N$. Let $M' = \min(M, |\cal G_u|)$, and define $\calG'_u$ to be the $M'$ fires in $\calG_u$ with the largest values of $\sum\limits_{q\in \calQ_g}B_{qt}y_q$, breaking ties arbitrarily, so $|\calG'_u|=M'$.

We first prove that $\sum\limits_{g\in \calG'_u}\sum\limits_{q\in \calQ_g}B_{qt}y_q \ge M'$. Assume by contradiction that $\sum\limits_{g\in \calG'_u}\sum\limits_{q\in \calQ_g}B_{qt}y_q < |\calG'_u|$. This implies that $\sum\limits_{q\in \calQ_g}B_{qt}y_q=0$ for at least one fire $g\in\calG'_u$. By construction of $\calG'_u$, this implies that $\sum\limits_{q\in \calQ_g}B_{qt}y_q=0$ for all $g\in\calG_u\setminus\calG'_u$. Therefore,
$$\sum\limits_{g\in \calG_u}\sum\limits_{q\in \calQ_g}B_{qt}y_q=\sum\limits_{g\in \calG'_u}\sum\limits_{q\in \calQ_g}B_{qt}y_q < |\calG'_u|=M'\leq M\leq N,$$
which is a contradiction. This proves that $\sum\limits_{g\in \calG'_u}\sum\limits_{q\in \calQ_g}B_{qt}y_q \ge |\calG'_u| = M'$.

We then prove that $\sum\limits_{g\in \calG_u\setminus \calG'_u}\sum\limits_{q\in \calQ_g}B_{qt}y_q \le J - |\calJ_u|$. We separate two cases:
\begin{enumerate}
\item If $M\leq|\calG_u|$, then $M'=M$, and we have, from the previous inequalities:
$$\sum\limits_{g\in \calG_u\setminus \calG'_u}\sum\limits_{q\in \calQ_g}B_{qt}y_q = \sum\limits_{g\in \calG_u}\sum\limits_{q\in \calQ_g}B_{qt}y_q - \sum\limits_{g\in \calG'_u}\sum\limits_{q\in \calQ_g}B_{qt}y_q \le N - M$$
\item Otherwise, we have $M'=|\calG_u|$ and $\calG_u\setminus \calG'_u = \emptyset$, so $\sum\limits_{g\in \calG_u\setminus \calG'_u}\sum\limits_{q\in \calQ_g}B_{qt}y_q=0$.
\end{enumerate}
In both cases, we have $\sum\limits_{g\in \calG_u\setminus \calG'_u}\sum\limits_{q\in \calQ_g}B_{qt}y_q \le N - M$, which implies that
$$\sum\limits_{g\in \calG_u\setminus \calG'_u}\sum\limits_{q\in \calQ_g}B_{qt}y_q \le J - |\calJ_u|.$$

Define $S = \left\{\left(g, \sum\limits_{q\in \calQ_g} B_{qt}y_q\right)  : g \in \calG_u\setminus \calG'_u\right\}$. Note that $S\in\calH$ because:
$$\sum_{(g,d)\in S}d=\sum_{g \in \calG_u\setminus\calG'_u}\sum_{q\in\calQ_g}B_{qt}y_q \le J - |\calJ_u|,$$
and, for any $g\in\calG_u$:
$$|\{(g', d) \in S : g' = g\}|=\begin{cases}
    1&\text{if $g\in\calG_u\setminus\calG'_u$}\\
    0&\text{if $g\in\calG'_u$}
\end{cases}$$
Therefore, the condition of the theorem gives 
$$\sum\limits_{(g, d) \in S} \delta_{ugd} =\sum_{g \in \calG_u\setminus \calG'_u}\delta_{ugd(g)}\le K_u,$$
where $d(g) = \sum\limits_{q\in \calQ_g} B_{qt}y_q$. We obtain:
\begin{equation}
    \label{proof1}
    \sum_{g\in \calG_u\setminus \calG'_u}\sum_{q\in \calQ_g}\sum_{d=0}^{J} \mathbbm 1[B_{qt} = d]\delta_{ugd}y_q = \sum_{g\in \calG_u\setminus \calG'_u} \delta_{ugd(g)} \le K_u.
\end{equation}

Meanwhile, 
\begin{equation}
    \label{proof2}
    \sum_{g\in \calG'_u}\sum_{q\in \calQ_g}\sum_{d=0}^{J} \mathbbm 1[B_{qt} = d]\delta_{ugd}y_q 
    \le \sum_{g\in \calG'_u}\sum_{q\in \calQ_g}\sum_{d=0}^{J} \mathbbm 1[B_{qt} = d]y_q
    = \sum_{g\in \calG'_u}\sum_{q\in \calQ_g} y_q
    = |G'_u|,
\end{equation}
where the inequality uses $\delta_{ugd} \le 1$, the first equality follows from the fact that $B_{qt}\leq J$ (Lemma~\ref{lem fire plan}) and the last equality follows from Constraints~\eqref{supp_plan}. Moreover,
\begin{align*}
    \sum_{j\in \calJ_u}\sum_{p\in \calP_j}\mathbbm 1\left[\sum_{g\in \calG_u}A_{pgt} = 0\right]z_p
    &\leq \sum_{j\in \calJ}\sum_{p\in \calP_j}\mathbbm 1\left[\sum_{g\in \calG_u}A_{pgt} = 0\right]z_p\\
    &= \sum_{j\in \calJ}\sum_{p\in \calP_j}z_p-\sum_{j\in \calJ}\sum_{p\in \calP_j}\sum_{g\in \calG_u}A_{pgt} z_p\\
    &= J - \sum_{g\in \calG_u}\sum_{j\in \calJ}\sum_{p\in \calP_j}A_{pgt}z_p \\
    &\le J - \sum_{g\in \calG_u}\sum_{q\in \calQ_g}B_{qt}y_q \\
    &= J - N,
\end{align*}
where the first inequality follows from the fact that $\calJ_u\subseteq\calJ$, the first equality uses the fact that $\sum\limits_{g\in \calG_u}A_{pgt}\in\{0,1\}$ (Lemma~\ref{lem crew route}), the second equality uses Constraints~\eqref{crew_route}, the second inequality uses Constraints~\eqref{link set partition}, and the last equality uses the definition of $N$. Furthermore, 
\begin{equation}
\label{proof3}
    \sum_{j\in \calJ_u}\sum_{p\in \calP_j}\mathbbm 1\left[\sum_{g\in \calG_u}A_{pgt} = 0\right]z_p \le \sum_{j\in \calJ_u}\sum_{p\in \calP_j}z_p = |\calJ_u|.
\end{equation}

Hence, we have from Equations~\eqref{proof1}--\eqref{proof3}:
\begin{align*}
    &\sum_{g\in \calG_u}\sum_{q\in \calQ_g}\sum_{d=0}^{J} \mathbbm 1[B_{qt} = d]\delta_{ugd}y_q + \sum_{j\in \calJ_u}\sum_{p\in \calP_j}\mathbbm 1\left[\sum_{g\in \calG_u}A_{pgt} = 0\right]z_p \\
    \qquad&= \sum_{g\in \calG_u\setminus\calG'_u}\sum_{q\in \calQ_g}\sum_{d=0}^{J} \mathbbm 1[B_{qt} = d]\delta_{ugd}y_q + \sum_{g\in \calG'_u}\sum_{q\in \calQ_g}\sum_{d=0}^{J} \mathbbm 1[B_{qt} = d]\delta_{ugd}y_q + \sum_{j\in \calJ_u}\sum_{p\in \calP_j}\mathbbm 1\left[\sum_{g\in \calG_u}A_{pgt} = 0\right]z_p \\
    \qquad&\le K_u + |\calG'_u| + \min(|\calJ_u|, J - N) \\
    \qquad&\le K_u + M + \min(|\calJ_u|, J - N) \\
    \qquad&= K_u + \max(0, N - J + |\calJ_u|) + \min(|\calJ_u|, J - N) \\
    \qquad&= K_u + |\calJ_u|
\end{align*}
This completes the proof.
\hfill\Halmos
        
\subsection{Proof of Proposition~\ref{prop:CGLP superset}.}
We show that any inequality constructed by inputs to Theorem~\ref{thm::GUB cover} can also be constructed by inputs to Theorem~\ref{thm:general_gub_aware_cut}. Let $t\in \calT$, $\calG_u \subseteq \calG$, $\calJ_u \subseteq \calJ$, and for all $g\in \calG_u$, let $D_{ug}\in \mathbb Z_+$ be such that
$$\sum_{g\in \calG_u} D_{ug} + |\calJ_u|> J.$$

Define $K_u = |\calG_u| - 1$, and for all $(g, d)\in {\mathcal{G}_u\times\{0, \ldots, J\}}$, define $\delta_{ugd} = \mathbbm 1[d \ge D_{ug}]$. We now check that $(t, \calG_u, \calJ_u, \{\delta_{ugd}\}_{\mathcal{G}_u\times\{0, \ldots, J\}}, K_u)$ satisfy the conditions of Theorem~\ref{thm:general_gub_aware_cut}.

By construction, $\delta_{ugd} \le 1$ for all $g, d$. Consider $S\in\calH$, i.e., $S \subseteq\mathcal{G}_u\times\{0, \ldots, J\}$ where $\sum\limits_{(g, d)\in S} d \le J - |\calJ_u|$ and for all $g\in \mathcal{G}_u$, $|\{(g', d) \in S : g' = g\}| \le 1$. We have:
\begin{align*}
    \sum\limits_{(g, d) \in S} \delta_{ugd} &= \sum\limits_{g\in \cal
    G_u}\sum\limits_{\{d \in \{0, \ldots, J\} : (g, d) \in S\}} \mathbbm 1[d \ge D_{ug}].
\end{align*}

This sum has as most $|\calG_u|$ terms because for all $g\in \mathcal{G}_u$, $|\{(g', d) \in S : g' = g\}| \le 1$, so $\sum\limits_{(g, d) \in S} \delta_{ugd}\leq|\calG_u|$. If it is equal to $|\calG_u|$, then each fire $g\in\calG_u$ is associated with exactly one value of $d(g)$ such that $(g,d(g))\in S$, and this value satisfies $d(g)\geq D_{ug}$. Therefore, $\sum\limits_{(g, d) \in S} d \ge \sum\limits_{g\in \calG_u} D_{ug} > J - |\calJ_u|$, leading to a contradiction. Hence, $\sum\limits_{(g, d) \in S} \delta_{ugd} \le |\calG_u| - 1=K_u$.

From this set $S$, we obtain the following A-GUB cut:
$$\sum_{g\in \calG_u}\sum_{q\in \calQ_g}\sum_{d=0}^{J} \mathbbm 1[B_{qt} = d]\delta_{ugd}y_q + \sum_{j\in \calJ_u}\sum_{p\in \calP_j}\mathbbm 1\left[\sum_{g\in \calG_u}A_{pgt} = 0\right]z_p \le |\calJ_u| + K_u = |\calJ_u| + |\calG_u| - 1.$$
Note, moreover, that
$$\sum_{g\in \calG_u}\sum_{q\in \calQ_g}\sum_{d=0}^{J} \mathbbm 1[B_{qt} = d]\delta_{ugd}y_q= \sum_{g\in \calG_u}\sum_{q\in \calQ_g} \sum_{d=0}^{J} (\mathbbm 1[B_{qt} = d]\cdot\mathbbm 1[d \ge D_{ug}])y_q=\sum_{g\in \calG_u}\sum_{q\in \calQ_g} \mathbbm 1[B_{qt} \ge D_{ug}]y_q$$
Hence, the A-GUB cut is exactly equivalent to the original GUB cut:
$$\sum_{g\in \calG_u}\sum_{q\in \calQ_g} \mathbbm 1[B_{qt} \ge D_{ug}]y_q + \sum_{j\in \calJ_u}\sum_{p\in \calP_j}\mathbbm 1\left[\sum_{g\in \calG_u}A_{pgt} = 0\right]z_p \le |\calG_u| + |\calJ_u| - 1.$$
Thus, any GUB cover cut can be written as an A-GUB cut, completing the proof.
\hfill\Halmos

\subsection{Proof of Proposition~\ref{prop:min_var}.}

We make use of the following lemma.

\begin{lemma} \label{lem:cut feasible}
    If $v(g,t)>0$, there exist $q_-\in \calQ_g$ with $y^*_{q_-} > 0$ and $B_{q_-,t} \le \left\lfloor\sum\limits_{q\in \calQ_g} B_{qt}y^*_q\right\rfloor$, and $q_+\in \calQ_g$ with $y^*_{q_+} > 0$ and $B_{q_+,t} > \left\lfloor\sum\limits_{q\in \calQ_g} B_{qt}y^*_q\right\rfloor$. Similarly, if $w(j,g,t)>0$, there exist $p_-\in \calP_j$ with $z^*_{p_-} > 0$ and $A_{p_-,g,t} \le \left\lfloor\sum\limits_{p\in \calP_j} A_{pgt}z^*_p\right\rfloor$, and $p_+\in \calP_j$ with $z^*_{p_+} > 0$ and $A_{p_+,g,t} > \left\lfloor\sum\limits_{p\in \calP_j} A_{pgt}z^*_p\right\rfloor$.
\end{lemma}

\paragraph{Proof of Lemma~\ref{lem:cut feasible}.}
Using a probabilistic interpretation, define the random variable $X_{gt}$ equal to $B_{qt}$ with probability $y^*_q$ for each $q\in \calQ_g$. This is a valid probability distribution because of Constraints~\eqref{supp_planMP} and~\eqref{fire nonneg constrMP}. The condition of the proposition becomes $\text{var}(X_{gt}) > 0$. This means that $X_{gt}$ takes on some value less than $\mathbb E[X_{gt}] = \sum\limits_{q\in \calQ_g} B_{qt}y^*_q$ with positive probability and some value greater than $\sum\limits_{q\in \calQ_g} B_{qt}y^*_q$ with positive probability. This implies the existence of some $q_-\in \calQ_g$ with $y^*_{q_-} > 0$ and $B_{q_-,t} \le \sum\limits_{q\in \calQ_g} B_{qt}y^*_q$ and some $q_+\in \calQ_g$ with $y^*_{q_+} > 0$ and $B_{q_+,t} > \sum\limits_{q\in \calQ_g} B_{qt}y^*_q$. Per the integrality of  $B_{q_-t}$, we have $B_{q_-,t} \le \left\lfloor\sum\limits_{q\in \calQ_g} B_{qt}y^*_q\right\rfloor$ and $B_{q_+,t} > \left\lfloor\sum\limits_{q\in \calQ_g} B_{qt}y^*_q\right\rfloor$. This proves the first result. We proceed similarly for the second one.
\hfill\Halmos

\paragraph{Proof of Proposition~\ref{prop:min_var}.}

We again use the probabilistic interpretation of the natural variables. Define the random variable $X_{gt}$ to equal $B_{qt}$ with probability $y^*_q$ for all $q\in \calQ_g$, and define the random variable $Y_{jgt}$ to equal $A_{pgt}$ with probability $z^*_p$ for all $p\in \calP_j$. These are valid probabilities because of Constraints~\eqref{supp_planMP},~\eqref{crew_routeMP},~\eqref{crew nonneg constrMP} and~\eqref{fire nonneg constrMP}. The condition of the proposition becomes $\text{var}(X_{gt}) = 0$ and $\text{var}(Y_{jgt}) = 0$. This is equivalent to saying $X_{gt}$ and $Y_{jgt}$ are constant, which is equivalent to the second condition of the proposition based on the probabilistic definitions of  $X_{gt}$ and $Y_{jgt}$.
\hfill\Halmos

\subsection{Details on upper-bounding heuristic}
\label{app:heuristic}

Algorithm~\ref{alg:heuristic} details the upper-bounding heuristic applied throughout the two-sided branch-and-price algorithm. Assume that a fire is assigned, in the linear relaxation, to plan $q_1$ assigning 10 crews in period 1 and 0 in period 2, and to plan $q_2$ assigning no crew in period 1 and 10 crews in period 2. Intuitively, we might want to assign 5 crews in periods 1 and 2 to that fire. The heuristic follows that logic, and then iteratively increases the number of allowed crews (e.g., 6, 7 crews, etc.).

\begin{algorithm} [h!]
\caption{\textsc{FireDemandHeuristic}.}\small
\label{alg:heuristic}
\begin{algorithmic}
\item \textbf{Initialization:} cost $c\gets\infty$, fractional solution $(\by^*,\bz^*)$, fire demands $\overline{D}_{gt}=\left\lceil\sum\limits_{q\in \calQ_g}B_{qt}y^*_q\right\rceil$, $\calP'\gets \emptyset$, $\calQ'\gets \emptyset$.
\item Iterate over Steps 1-3 until termination criterion (e.g., improvements get too small).
\begin{itemize}
    \item[] \textbf{Step 1.} Solve relaxation with \textsc{TwoSidedColumnGeneration} (Algorithm~\ref{alg:2CG_verbose}) and cutting planes (Section~\ref{section:cutting planes}), subject to $\sum\limits_{q\in \calQ_g }\mathbbm{1}[B_{qt} \le \overline{D}_{gt}]y_q = 1,\ \forall g\in\calG,\forall t\in\calT$; append columns to $\calP'$ and $\calQ'$.
    \item[] \textbf{Step 2.} Solve master problem with integrality constraints. Update cost $c$ and solution $\by, \bz$.
    \item[] \textbf{Step 3.} Increment $\overline{D}_{gt}\gets\overline{D}_{gt}+1,\ \forall g\in\calG,\forall t\in\calT$, and go to Step 1.
\end{itemize}
\end{algorithmic}
\end{algorithm}

\subsection{Proof of Theorem~\ref{thm:validity of branch-and-price-and-cut}.}
The two-sided column generation and cutting planes procedures (Step 1 of Algorithm~\ref{alg:branch-and-price-and-cut}) converge finitely because (i) any new cut renders the incumbent primal solution infeasible and any new variable renders the incumbent dual solution infeasible, so the algorithm does not cycle; and (ii) the number of variables and cuts under consideration are finite. It therefore suffices to show that the algorithm explores finitely many nodes in the branching tree. Since there are finitely many possible values of $\left\lfloor\sum\limits_{q\in \calQ_g} B_{qt}y^*_q\right\rfloor$ for all $g\in \calG$ (namely, $\{0, \ldots, J\}$) and finitely many possible values of $\left\lfloor\sum\limits_{p\in \calP_j} A_{pgt}z^*_p\right\rfloor$ for all $g\in \calG$, $j\in \calJ$ (namely, $\{0, 1\}$), there are finitely many possible branching rules. Furthermore, the maximum possible depth of the tree is upper-bounded by the number of branching rules and no branching rule can be duplicated, since any new branching eliminates a feasible fractional solution. This completes the proof.
\hfill\Halmos

\section{Details on fire spread model}

\subsection{Detailed Feature Description}
\label{app:features}

Table~\ref{tab:feature_details} describes the covariates used in the predictive models. These features were selected to capture wildfire dynamics, including its current state, its recent behavior, the environmental context, and the nature of the fire itself.

\begin{table}[p]
\centering
\caption{Detailed Description of Model Covariates.}
\label{tab:feature_details}
\small\renewcommand{\arraystretch}{1.1}
\begin{tabularx}{\textwidth}{lX}
\toprule\toprule
\textbf{Feature Group} & \textbf{Description and Rationale} \\
\cmidrule(r){1-2}
\multicolumn{2}{l}{\textit{Core Dynamic Features (SIT-209)}} \\
\midrule
Current Area & The total area burned in acres at the start of the day. \textbf{Rationale:} Primary measure of fire size, included in the state variable $S_{gt}$. \\
Growth Rate & The increase in burned area from the previous day. \textbf{Rationale:} Captures the fire's recent momentum, distinguishing rapidly accelerating fires from stable or decelerating ones. \\
Duration & The number of days since the fire was first reported. \textbf{Rationale:} Acts as a proxy for the fire's maturity and potential for resource fatigue or containment progress. \\
\midrule
\multicolumn{2}{l}{\textit{Temporal and Categorical Features (SIT-209)}} \\
\midrule
Month & Categorical feature for the month of the year. \textbf{Rationale:} Captures seasonal variations in climate, weather patterns, and fuel moisture that influence fire behavior. \\
Fire Cause & 4 one-hot encoded categories: Human, Lightning, Other, Undetermined. \textbf{Rationale:} The origin of the fire can influence initial spread characteristics and typical location (e.g., wildland-urban interface versus remote wilderness). \\
Suppression Method & 5 one-hot encoded categories describing the primary management strategy (Full Suppression, Confine, Monitor, Point/Zone
Protection, and Multiple Management Strategy). \textbf{Rationale:} Reflects the intended suppression posture, which correlates with resource allocation decisions. \\
\midrule
\multicolumn{2}{l}{\textit{Behavioral Features (SIT-209)}} \\
\midrule
Fuel Type & 13 one-hot encoded categories from the National Fire Danger Rating System (NFDRS) fuel models, covering grass, brush, timber, and logging slash. \textbf{Rationale:} Fuel is a primary driver of fire intensity and rate of spread. Different fuel types present unique suppression challenges. \\
Fire Behavior & 4 one-hot encoded categories: Minimal, Moderate, Active, Extreme. \textbf{Rationale:} Provides a high-level summary of the fire's current intensity and predictability, as reported by on-the-ground personnel. \\
Granular Behavior & 16 one-hot encoded attributes describing specific phenomena: e.g., Backing, Creeping, Crowning, Flanking, Running, Smoldering, Spotting, Torching, Uphill Runs, and Wind-Driven Runs. \textbf{Rationale:} Provide detailed, observational data on how the fire is spreading, offering more nuanced visibility into fire behavior. \\
\midrule
\multicolumn{2}{l}{\textit{ERA5 Features}} \\
\midrule
Meteorology & 10 daily variables at 10km resolution: u-wind, v-wind, max temperature, dewpoint temperature, total precipitation, surface net solar radiation, soil water, soil temperature, leaf area index, total evaporation. \textbf{Rationale:} Provides quantitative data on the local weather conditions that govern fire behavior.\\
\midrule
\multicolumn{2}{l}{\textit{LandFire Features}} \\
\midrule
Fuel Coverage & Percentage coverage of each fuel category within a 10km x 10km grid cell. Categories include the 13 Anderson fire behavior fuel models and 5 non-burnable types (Urban, Snow/Ice, Agriculture, Water, and Barren). \textbf{Rationale:} Provides spatially explicit fuel characterization that complements the coarser SIT-209 fuel covariates.\\
Terrain & Mean elevation, mean slope, and standard deviation of slope. \textbf{Rationale:} Terrain affects fire spread direction and rate.\\
Vegetation & Percent tree cover, percent herbaceous cover, and mean canopy base height. \textbf{Rationale:} Vegetation affects fire type (surface versus crown fire).\\
\midrule
\multicolumn{2}{l}{\textit{AlphaEarth Features}} \\
\midrule
Embeddings & 64 learned remote sensing embedding dimensions at 5km resolution derived from the AlphaEarth foundation model for Earth observation imagery. \textbf{Rationale:} These embeddings encode many different sources, including optical satellite images, radar, 3D laser mapping, and climate simulations that provide extra information about the fire's location.\\
\bottomrule\bottomrule
\end{tabularx}
\end{table}

\subsection{IHC equivalents from resource allocation data}
\label{app:personnel_scaling}

The predictive and prescriptive components use suppression effort at different levels of aggregation. The historical SIT-209 data report suppression effort across multiple resource categories and total incident personnel. The optimization model, in contrast, assigns IHC crews. This choice avoids higher-dimensional assignment variables, which would hinder tractability in the optimization problem---a complex problem due to the number of path-based variables and linking constraints between fire demand and crew supply---and further complicate the predictive problem---a challenging problem due to the noisy response and confounding. We therefore introduce an IHC-equivalent scale that maps total personnel into the crew units used by the optimization model.

Let \(P_{it}\) denote total incident personnel assigned to fire \(i\) on day \(t\). Let \(P^{\mathrm{IHC}}_{it}\) denote the personnel reported in the IHC category ``type 1''). Since an IHC crew contains approximately 20 personnel, we first define the reported number of IHC crews as $C^{\mathrm{IHC}}_{it}=P^{\mathrm{T1}}_{it}/20$. We then estimate the staffing-equivalent value of one IHC crew using a linear regression specification with fire fixed effects:
\[
P_{it}=\alpha_i+\kappa C^{\mathrm{IHC}}_{it}+\varepsilon_{it}
\]
The fixed effects absorb persistent differences across fires, so \(\kappa\) is identified from within-fire changes over time. Thus, \(\kappa\) measures the average change in total incident personnel associated with one additional reported IHC crew on the same fire.

This regression yields \(\widehat{\kappa}=63\) personnel. We use this estimate as a staffing-equivalent scale factor. This estimate should not be interpreted as saying that an IHC contains 63 firefighters; rather, it reflects that, in the historical data, the deployment of an additional IHC is typically accompanied by additional staffing and support resources on the same incident.

We therefore define the IHC-equivalent treatment variable as $T_{it}=P_{it}/\widehat{\kappa}$. We use this conversion as a deterministic linear rescaling in the predictive model to report the learned treatment-response function in IHC-equivalent units, and to preserve consistency between the treatment variable in the predictive model an the crew assignment decisions in the optimization model. Thus, a decision of $x$ IHC-equivalent crews in the optimization translates into a personnel level of $\widehat{\kappa}\cdot x$ in the data.

\subsection{Second-stage ensemble model}
\label{app:ensemble}

We apply four first-stage LightGBM models with varying feature sets (i) SIT-209; (ii) SIT-209 and ERA5; (iii) SIT-209, ERA5 and LandFire; and (iv) SIT-209, ERA5, LandFire and AlphaEarth embeddings. All four models share the same training/test split, the same GroupKFold cross-fitting procedure, and the same LGBM methodology. We then leverage these four models to build an ensemble treatment-response model with skip-zero averaging, as follows:
\begin{enumerate}
\item Each model $k\in\{1,2,3,4\}$ generates its treatment-response curve, denoted by $\widehat{Y}^{(k)}_{is}(t)$ for fire $i$ in period $s$, following the procedure described in Section~\ref{subsec:DML}.
\item Each curve is classified as zero effect if the total treatment range is below a threshold, i.e., if $\widehat{Y}^{(k)}_{is}(0)-\widehat{Y}^{(k)}_{is}(J)\leq\varepsilon$ where $\varepsilon = 1$ acre. A zero-effect classification indicates that the second-stage models use the treatment residual at tree splits only sparingly, resulting in a flat prediction across all crew allocations for that fire.
\item Define the active set $\calS_{is}$ as the set of models $k\in\{1,2,3,4\}$ with a non-zero effect for fire $i$ in period $s$. The ensemble prediction is then given as the average over non-zero models, with a fallback to the full average in case all models yield a zero effect:
\begin{align*}
	\widehat{Y}^{(\text{ens})}_{is}(t)=\begin{cases}
		\frac{1}{|\calS_{is}|}\sum_{k\in\calS_{is}}\widehat{Y}^{(k)}_{is}(t)&\text{if $\calS_{is}\neq\emptyset$}\\
		\frac{1}{4}\sum_{k=1}^4\widehat{Y}^{(k)}_{is}(t)&\text{otherwise}
        \end{cases}
\end{align*}
\end{enumerate}

Empirically, the zero-effect rate varies substantially across models, ranging from $3\%$ of fires in the smallest model with 46 SIT-209 covariates to $38\%$ in the largest model (144 covariates: 46 SIT-209 base, 10 ERA5, 24 LandFire, and 64 AlphaEarth embedding dimensions). Crucially, the zero-effect set is not the same across models: a fire flagged as flat by one model is often non-flat under another. The skip-zero ensembling approach described above exploits this complementarity: for any given fire and period, the ensemble uses only those models that successfully captured a treatment signal, discarding the flat predictions that would dilute the ensemble's treatment sensitivity. The fallback clause handles the rare case where all four models predict a flat treatment-response curve; in practice, this occurs for $0.1\%$ of fire-day observations in the test set.

\subsection{Constructing the fire spread model from the predictive model}
\label{app:firenetwork}

We translate the DML predictions into a time-state network, as follows:
\begin{enumerate}
    \item \textbf{State space discretization:} Recall that the state variable is modeled as a two-dimensional vector characterizing the current area burned and the momentum defined as the previous day's growth. The continuous state variable is discretized into a finite set of states $\calS_g$ using a piecewise uniform grid. This grid uses fine-grained bins for smaller areas that increase for larger areas (Table \ref{tab:discretization_bins}). This discretization manages the state space size while retaining granularity for the more common fires. We proceed similarly for momentum. We add terminal ``extinguished" states that are unique to the area at which the fire was extinguished. Altogether, we obtain billions of possible states with a  $51,031\times51,031$ discretization grid; in practice, time-state networks comprise thousands of states.

    \begin{table}[h]
    \centering
    \small
    \caption{Fire-area discretization bins}
    \label{tab:discretization_bins}
    \begin{tabular}{lrrr}
    \toprule
    Area range (acres)              & Step size (acres) & \# of bins \\
    \midrule
    1\,--\,100                      &  2                &        50  \\
    100\,--\,10{,}000               &  5                &     1{,}981  \\
    10{,}000\,--\,500{,}000         & 10                &    49{,}000  \\
    \midrule
    \textbf{Total (unique)}         &                   & \textbf{51{,}031} \\
    \bottomrule
    \end{tabular}
    \end{table}

    \item \textbf{Node generation:} For any fire, the algorithm iterates over the planning horizon $t \in \{1, \dots, T\}$. At each period $t$, it considers all reachable states from the previous period.

    \item \textbf{Arc creation:} For each active state $S_{g,t}$, the algorithm constructs a feature vector representing the fire's condition (e.g., current area, growth history, fuel type). It then calls the DML model to predict the next-day growth for a range of IHC-equivalent crew assignment levels, which define the values of the decision variable $x_{gt}$ in the dynamic model. The DML model yields the predicted next-day area and momentum. These quantities are mapped to the nearest values in the discretized state space, yielding the next state $S_{g,t+1}$. An arc is then created from state $S_{g,t}$ to state $S_{g,t+1}$, associated with the assignment level $x_{gt}$. Since multiple personnel assignments might lead to the same discretized state, we eliminate dominated transitions by only storing the arc corresponding to the minimum number of IHC-equivalent crews required.
\end{enumerate}

\subsection{A simple fire spread model}
\label{subsec:erinmodel}

To test the scalability of our optimization methodology in a controlled environment, we make use of the fire spread model from \cite{donovan2003integer} and \cite{wei2015chance}. In this model, fire damage grows proportionally to the active burning perimeter at each time step, and each crew decreases the perimeter by a fixed amount at each time period. We denote by $P$ the starting perimeter of the fire; by $R_t$ the ratio of its perimeter at time $t+1$ to its perimeter at time $t$, absent any suppression activities; and by $E_t$ the effectiveness of a crew at time $t$ (in units of perimeter suppressed). The decision variables are:
\begin{align*}
    s_t:\ &\text{number of crews suppressing the fire at time $t$}\\
    \ell_t:\ &\text{perimeter suppressed between times $t$ and $t+1$}\\
    p_t:\ &\text{perimeter of fire at time $t$}\\
    a_t:\ &\text{area burned by time $t$}
\end{align*}

The model minimizes total area burned by the end of the time horizon (Equation~\eqref{erin_model::obj}). Constraints~\eqref{erin_model::suppression}--\eqref{perimeter start} govern the growth of the fire perimeter, and Constraints~\eqref{area} and Constraints~\eqref{area start} translate the active fire perimeter to burned area. 
    \begin{mini!}
    {}{a_{T}}{\label{model::erin_model}}{} \label{erin_model::obj}
        \addConstraint{\ell_{t} \le E_{t}s_{t},\,\forall t\in \calT} \label{erin_model::suppression}
    \addConstraint{p_{t+1} \ge R_{t}\left(p_{t} - \frac{\ell_{t}}{2}\right) - \frac{\ell_{t}}{2},\, \forall t\in \calT} \label{perimeter growth}
    \addConstraint{p_{t} \ge 0, \, \forall t\in \calT} \label{perimeter positive}
    \addConstraint{p_{1} = P} \label{perimeter start}
    \addConstraint{a_{t+1} \ge a_{t} + \frac{p_{t} + p_{t+1}}{2},\, \forall t\in \calT} \label{area} 
    \addConstraint{a_{1} \ge 0} \label{area start}
    \addConstraint{s_t \in \mathbb Z,\, \forall{t}\in \calT.} 
    \end{mini!}

\section{Additional results}
\label{app:results}




\subsection{Predictive results}
\label{app:predictiveresults}

Tables~\ref{tab:nuisanceYdetails} and~\ref{tab:nuisanceTdetails} report detailed results of the nuisance models using different predictive models. These results confirm that LGBM yields the strongest predictive performance for predicting the outcome variable in the first stage; for predicting the treatment, best performance is generally obtained with LGBM and XGBoost. We use LGBM throughout the paper as a result.

\begin{table}[h!]
\centering
\caption{First-stage performance across feature configurations for predicting the outcome variable}
\label{tab:nuisanceYdetails}
\footnotesize
\begin{tabular}{ccccccrrrr}
\toprule
\multicolumn{5}{c}{Feature Sources} & \multirow{2}{*}{Model} & \multirow{2}{*}{RMSE} & \multirow{2}{*}{MAE} & \multirow{2}{*}{$R^2$} & \multirow{2}{*}{CV RMSE} \\
\cmidrule(lr){1-5}
\#Feat & SIT209 & ERA5 & LandFire & AlphaEarth &  &  &  &  &  \\
\midrule
 &  &  &  &  & LGBM & 536 & \textbf{379} & 0.268 & 590 \\
 &  &  &  &  & RF & \textbf{533} & 393 & \textbf{0.275} & \textbf{588} \\
 &  &  &  &  & XGB & 541 & 381 & 0.254 & 590 \\
46 & \checkmark &  &  &  & OLS & 651 & 456 & -0.079 & 628 \\
 &  &  &  &  & Lasso & 650 & 457 & -0.077 & 628 \\
 &  &  &  &  & Ridge & 650 & 457 & -0.076 & 628 \\
 &  &  &  &  & ENet & 650 & 458 & -0.075 & 628 \\
\midrule
 &  &  &  &  & LGBM & \textbf{534} & \textbf{375} & \textbf{0.272} & \textbf{588} \\
 &  &  &  &  & RF & 538 & 398 & 0.263 & 588 \\
 &  &  &  &  & XGB & 539 & 379 & 0.258 & 592 \\
56 & \checkmark & \checkmark &  &  & OLS & 647 & 450 & -0.065 & 627 \\
 &  &  &  &  & Lasso & 644 & 449 & -0.058 & 626 \\
 &  &  &  &  & Ridge & 633 & 451 & -0.022 & 625 \\
 &  &  &  &  & ENet & 630 & 450 & -0.012 & 625 \\
\midrule
 &  &  &  &  & LGBM & \textbf{530} & 370 & \textbf{0.284} & \textbf{585} \\
 &  &  &  &  & RF & 549 & 422 & 0.232 & 592 \\
 &  &  &  &  & XGB & 531 & \textbf{370} & 0.282 & 587 \\
80 & \checkmark & \checkmark & \checkmark &  & OLS & 650 & 452 & -0.077 & 627 \\
 &  &  &  &  & Lasso & 647 & 449 & -0.066 & 626 \\
 &  &  &  &  & Ridge & 635 & 450 & -0.028 & 624 \\
 &  &  &  &  & ENet & 632 & 450 & -0.018 & 624 \\
\midrule
 &  &  &  &  & LGBM & \textbf{526} & \textbf{367} & \textbf{0.295} & \textbf{582} \\
 &  &  &  &  & RF & 540 & 400 & 0.256 & 600 \\
 &  &  &  &  & XGB & 527 & 367 & 0.293 & 584 \\
144 & \checkmark & \checkmark & \checkmark & \checkmark & OLS & 612 & 418 & 0.044 & 630 \\
 &  &  &  &  & Lasso & 625 & 430 & 0.003 & 620 \\
 &  &  &  &  & Ridge & 609 & 424 & 0.055 & 619 \\
 &  &  &  &  & ENet & 608 & 424 & 0.058 & 619 \\
\bottomrule
\end{tabular}
\end{table}

\begin{table}[h!]
\centering
\caption{First-stage performance across feature configurations for predicting the treatment variable}
\label{tab:nuisanceTdetails}
\footnotesize
\begin{tabular}{ccccccrrrr}
\toprule
\multicolumn{5}{c}{Feature Sources} & \multirow{2}{*}{Model} & \multirow{2}{*}{RMSE} & \multirow{2}{*}{MAE} & \multirow{2}{*}{$R^2$} & \multirow{2}{*}{CV RMSE} \\
\cmidrule(lr){1-5}
\#Feat & SIT209 & ERA5 & LandFire & AlphaEarth &  &  &  &  &  \\
\midrule
 &  &  &  &  & LGBM & 6.10 & 2.33 & 0.462 & \textbf{5.90} \\
 &  &  &  &  & RF & 6.19 & 2.41 & 0.446 & 6.24 \\
 &  &  &  &  & XGB & \textbf{5.81} & \textbf{2.29} & \textbf{0.513} & 6.11 \\
46 & \checkmark &  &  &  & OLS & 8.22 & 3.17 & 0.025 & 6.76 \\
 &  &  &  &  & Lasso & 8.22 & 3.14 & 0.022 & 6.75 \\
 &  &  &  &  & Ridge & 8.17 & 3.19 & 0.034 & 6.76 \\
 &  &  &  &  & ENet & 8.17 & 3.17 & 0.034 & 6.76 \\
\midrule
 &  &  &  &  & LGBM & 4.95 & 2.06 & 0.646 & 5.76 \\
 &  &  &  &  & RF & 5.94 & 2.30 & 0.491 & 6.37 \\
 &  &  &  &  & XGB & \textbf{4.86} & \textbf{1.98} & \textbf{0.659} & \textbf{5.73} \\
56 & \checkmark & \checkmark &  &  & OLS & 8.13 & 3.08 & 0.045 & 6.81 \\
 &  &  &  &  & Lasso & 8.16 & 3.02 & 0.039 & 6.73 \\
 &  &  &  &  & Ridge & 7.94 & 3.10 & 0.090 & 6.79 \\
 &  &  &  &  & ENet & 7.95 & 3.08 & 0.086 & 6.79 \\
\midrule
 &  &  &  &  & LGBM & \textbf{4.90} & \textbf{2.19} & \textbf{0.653} & 6.13 \\
 &  &  &  &  & RF & 6.10 & 2.44 & 0.464 & 6.40 \\
 &  &  &  &  & XGB & 5.38 & 2.22 & 0.581 & \textbf{5.98} \\
80 & \checkmark & \checkmark & \checkmark &  & OLS & 8.16 & 3.24 & 0.040 & 6.73 \\
 &  &  &  &  & Lasso & 8.14 & 3.10 & 0.042 & 6.60 \\
 &  &  &  &  & Ridge & 7.95 & 3.24 & 0.086 & 6.70 \\
 &  &  &  &  & ENet & 7.97 & 3.22 & 0.082 & 6.68 \\
\midrule
 &  &  &  &  & LGBM & \textbf{4.94} & 2.17 & \textbf{0.649} & 5.97 \\
 &  &  &  &  & RF & 6.13 & 2.33 & 0.458 & 6.57 \\
 &  &  &  &  & XGB & 5.11 & \textbf{2.11} & 0.622 & \textbf{5.83} \\
144 & \checkmark & \checkmark & \checkmark & \checkmark & OLS & 7.90 & 3.40 & 0.099 & 6.92 \\
 &  &  &  &  & Lasso & 7.76 & 2.89 & 0.129 & 6.60 \\
 &  &  &  &  & Ridge & 7.05 & 3.10 & 0.284 & 6.79 \\
 &  &  &  &  & ENet & 6.60 & 2.92 & 0.369 & 6.73 \\
\bottomrule
\end{tabular}
\end{table}

\begin{table}[h!]
\centering
\caption{Second-stage DML performance across feature configurations}
\label{tab:secondStageDetails}
\footnotesize
\begin{tabular}{ccccccrrrr}
\toprule
\multicolumn{5}{c}{Feature Sources} & \multirow{2}{*}{Model} & \multirow{2}{*}{RMSE} & \multirow{2}{*}{MAE} & \multirow{2}{*}{$R^2$} & \multirow{2}{*}{CV RMSE} \\
\cmidrule(lr){1-5}
\#Feat & SIT209 & ERA5 & LandFire & AlphaEarth &  &  &  &  &  \\
\midrule
46 & \checkmark &  &  &  & LGBM & \textbf{539} & 380 & \textbf{0.260} & 603 \\
 &  &  &  &  & OLS & 540 & \textbf{380} & 0.257 & \textbf{592} \\
\midrule
56 & \checkmark & \checkmark &  &  & LGBM & \textbf{536} & \textbf{375} & \textbf{0.269} & 601 \\
 &  &  &  &  & OLS & 538 & 378 & 0.263 & \textbf{591} \\
\midrule
80 & \checkmark & \checkmark & \checkmark &  & LGBM & \textbf{533} & \textbf{372} & \textbf{0.277} & 596 \\
 &  &  &  &  & OLS & 535 & 372 & 0.271 & \textbf{588} \\
\midrule
144 & \checkmark & \checkmark & \checkmark & \checkmark & LGBM & 529 & 368 & 0.286 & 594 \\
 &  &  &  &  & OLS & \textbf{529} & \textbf{365} & \textbf{0.288} & \textbf{590} \\
\bottomrule
\end{tabular}
\end{table}

\begin{figure}[t]
    \centering
    \begin{subfigure}[t]{0.48\linewidth}
        \centering
        \includegraphics[width=\linewidth]{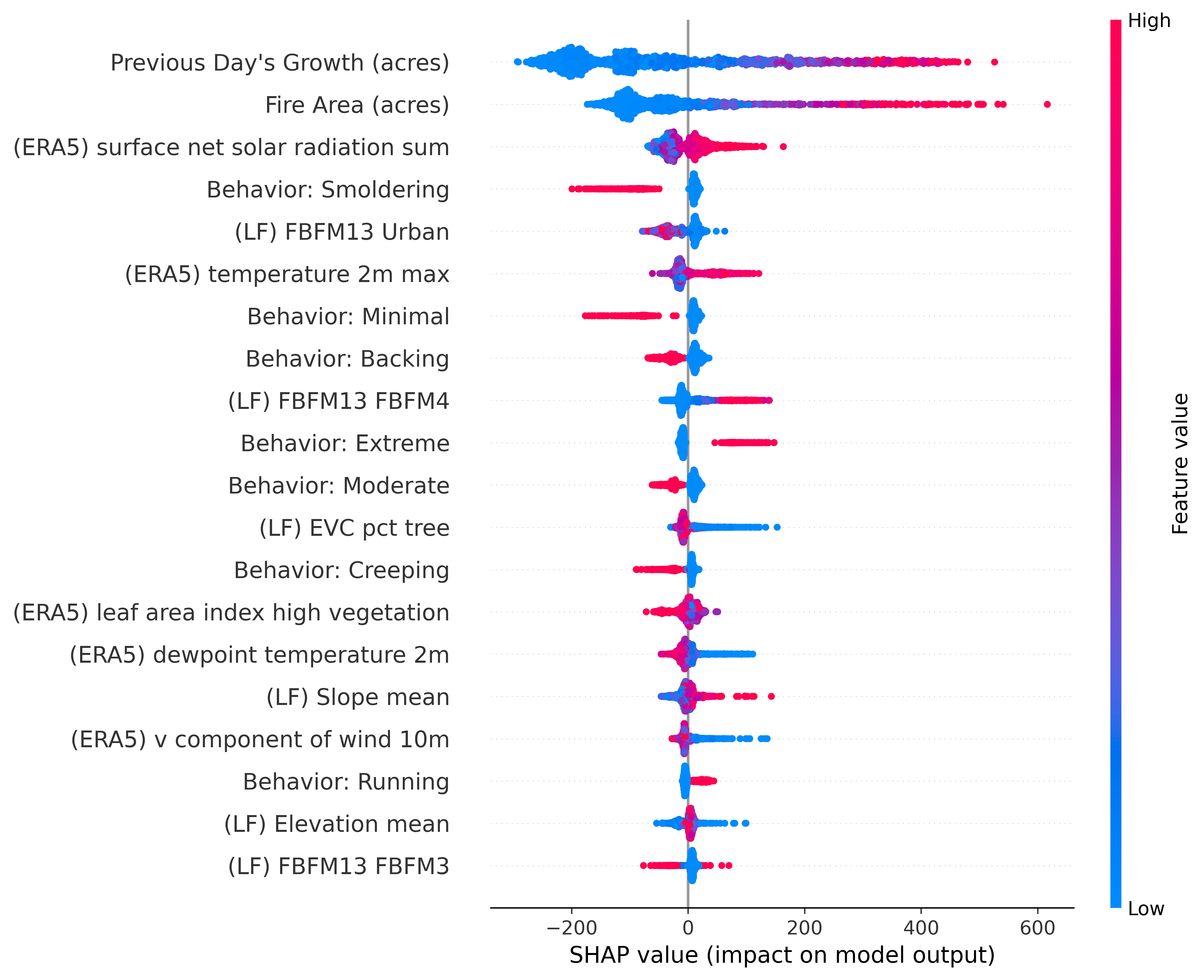}
        \caption{Model predicting outcome variable.}
        \label{fig:SHAPY}
    \end{subfigure}
    \hfill
    \begin{subfigure}[t]{0.48\linewidth}
        \centering
        \includegraphics[width=\linewidth]{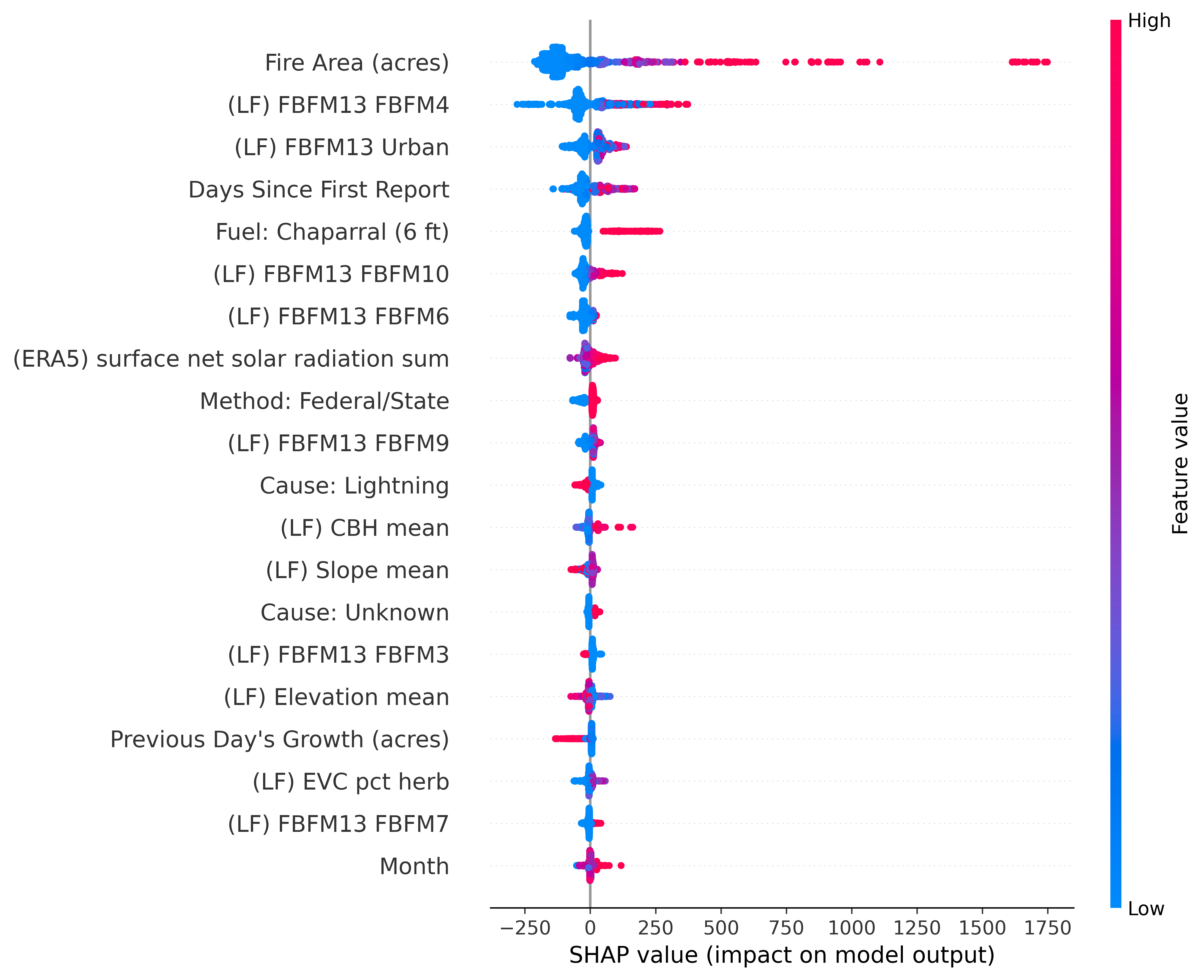}
        \caption{Model predicting treatment variable.}
        \label{fig:SHAPT}
    \end{subfigure}
    \caption{Full SHAP plots for the first-stage models in the DML methodology.}
    \label{fig:SHAP}
\end{figure}

\subsection{Impact of A-GUB cuts}
\label{app:cuts}

We compare the performance of each family of cuts from Section~\ref{section:cutting planes}. To isolate the impact of cutting planes, Figure~\ref{fig:cuts at root node} reports the linear relaxation bounds at the root node of the branching tree. The strengthened GUB cover cuts achieve a comparable bound as the original GUB cover cuts in fewer iterations. This result highlights benefits of the strengthening step but the strengthened GUB cover cuts still do not markedly tighten the linear relaxation over GUB cover cuts. In contrast, the A-GUB cuts from Theorem~\ref{thm:general_gub_aware_cut} result in significant improvements, with a more-than-twofold increase in the lower bound improvement in less-than-half the computational times in the largest instances. These results complement our theoretical results (Proposition~\ref{prop:CGLP superset}) with computational evidence of the benefits of the A-GUB cuts toward tightening the problem's linear relaxation.

\begin{figure}[h!]
    \centering
    \includegraphics[width=\textwidth]{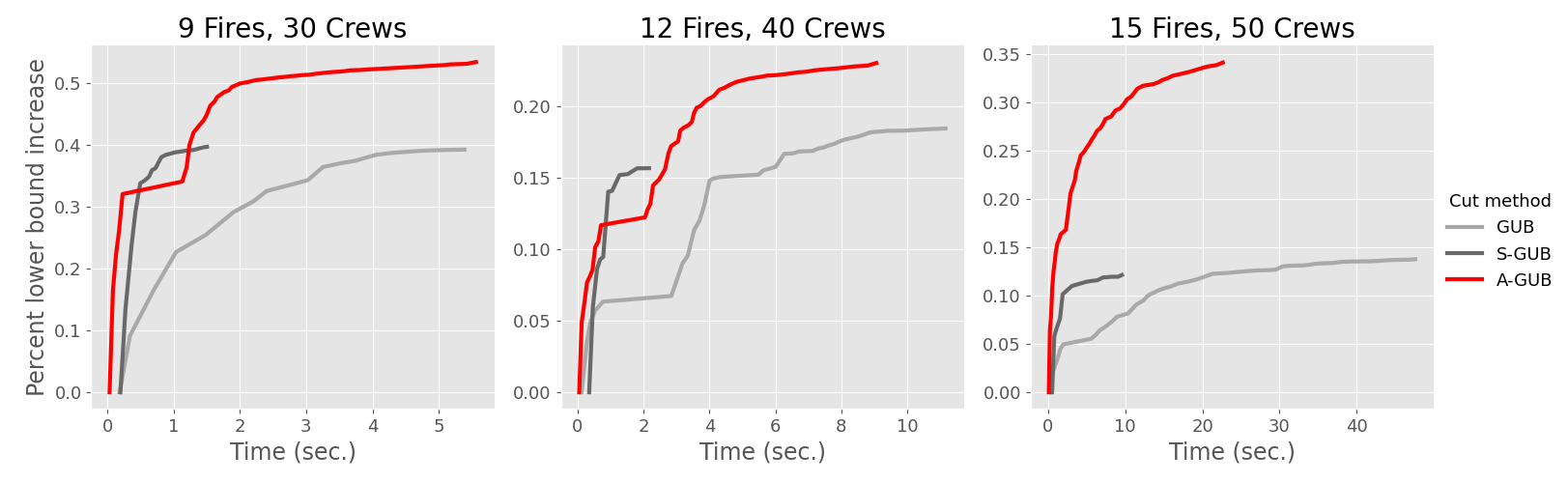}
    \caption{Lower bounds with GUB cover cuts, strengthened ones (S-GUB), and augmented ones (A-GUB).}\label{fig:cuts at root node}
    \vspace{-12pt}
\end{figure}

\subsection{Optimization results}
\label{app:optimizationresults}

\subsubsection*{Benchmarks.} Recall that the optimized solution is compared against four baselines: random assignment, a distance-based baseline, an area-based baseline, and an impact-based baseline. All baselines route crews via minimum-cost maximum-flow principles, subject to the same travel-time and rest constraints as the optimization problem. Only crews currently at base or a fire are eligible for deployment on each day; crews in transit cannot be reassigned. The arcs created by the DML model may require multiple crews to reach the next arc transition. The baselines determine the number of crews for the next arc transition for each fire, score it (based on distance, area or impact), assign crews to the highest-scoring transition among all the fires, then deduct those crews from the available supply and iterate until no crews remain available. Each score is divided by the number of crews required for that arc transition, to determine an ``average'' effect per crew. Then, the score is further divided by the average travel time (in 6-hour work periods) required to route that crew block to the fire, penalizing assignments that consume more travel time.

\subsubsection*{Optimization environment.} Each experiment is conditioned on the set of fires that actually occurred during the corresponding 14-day time window. The optimization model has visibility into fire ignition times and locations, covariate information, and the fire spread model from the DML methodology. As mentioned in the paper, our experiments therefore do not account for uncertainty in fire ignition, fire spread, or future covariate realizations. It can therefore be viewed as a best-case evaluation of what the optimization model can achieve under perfect information.

Our experiments also rely on initial locations for all crews. From the data, we estimate the number of IHC crews at each fire at the start of the planning horizon. Since we do not observe in the data \textit{which} crews (only their \textit{number}) are working on each fire, we simply assume crews are assigned to fires based on the proximity of their bases. Each crew is also randomly assigned a rest deadline between day 3 and day $T-2$. We consider a three-day rest requirement.

Each experiment covers a 14-day time horizon (e.g., June 1 - June 14). Travel times are discretized in 6-hour work-period increments: a crew traveling under 6 hours is assumed to arrive within the same period and is counted for suppression that day; travel between 6 and 12 hours means the crew arrives in the following period; and so on. Crews are dispatched at the start of each day.

\subsubsection*{Additional results.}
Tables~\ref{tab:results_2023} and \ref{tab:results_2024} report the results of each baseline and the optimization for each of the 45 intra-GACC experiments in 2023 and each of the 49 intra-GACC experiments in 2024. Generally, more and larger fires occur in the second half of the Summer and in the Western United States, but these tables also show significant variability in both time and place. As a sanity check, the optimization performs at least as well as each of the baselines for all 94 experiments. Among the baselines, the impact-based one most consistently achieves the lowest burned area, though there are exceptions. Moreover, the benefits come from consistent improvements over many experiments rather than outsized gains from a few instances.

Table~\ref{tab:bpc_results} includes the results of each baseline and the branch-and-price-and-cut algorithm for the 20 inter-GACC experiments across both summer seasons. Again, the optimization solution consistently outperforms all baselines, and the estimated 153,891 acreage improvement over the impact-based baseline is accumulated over consistent gains over time and space.

\begin{table}[ht]
  \centering
  \small
  \setlength{\tabcolsep}{4pt}
  \caption{Cumulative burned area (acres) by solution for six 14-day periods in Summer 2023.}
  \label{tab:results_2023}
  \begin{tabular}{llrrrrrrrr}
  \toprule
  Date & GACC & Crews & Fires & Zero & Random & Area & Distance & Impact & Optimized \\
  \midrule
  06 01 & EA &  1 &  11 & 38,710 & 38,590 & 38,710 & 38,710 & 38,710 & 38,155 \\
  06 01 & GB & 12 &   1 & 4,520 & 3,885 & 4,105 & 4,105 & 4,105 & 3,875 \\
  06 01 & NW & 13 &  12 & 49,055 & 47,346 & 47,600 & 47,334 & 47,349 & 46,037 \\
  06 01 & SA &  4 &   9 & 69,735 & 69,710 & 69,720 & 69,725 & 69,665 & 68,935 \\
  06 01 & SC & 27 &   1 & 4,775 & 4,510 & 4,510 & 4,510 & 4,510 & 4,510 \\
  06 01 & SW & 20 &   8 & 72,275 & 71,115 & 71,025 & 70,875 & 70,635 & 70,215 \\
  \midrule
  06 15 & EA &  1 &   4 & 15,275 & 15,270 & 15,275 & 15,275 & 15,275 & 14,805 \\
  06 15 & GB & 12 &   5 & 27,690 & 27,030 & 26,405 & 26,380 & 26,390 & 26,230 \\
  06 15 & NW & 13 &  10 & 83,265 & 82,460 & 81,215 & 81,230 & 80,305 & 78,880 \\
  06 15 & RM &  7 &   5 & 4,745 & 4,625 & 4,600 & 4,595 & 4,595 & 4,595 \\
  06 15 & SA &  4 &  10 & 43,375 & 43,172 & 43,178 & 43,108 & 43,143 & 42,213 \\
  06 15 & SC & 27 &   3 & 17,295 & 15,440 & 15,200 & 15,320 & 15,355 & 15,130 \\
  06 15 & SW & 20 &  17 & 158,775 & 157,435 & 156,895 & 156,400 & 156,245 & 155,315 \\
  \midrule
  07 01 & EA &  1 &   4 & 5,235 & 5,150 & 5,155 & 5,155 & 5,155 & 4,930 \\
  07 01 & GB & 12 &   6 & 10,635 & 9,915 & 9,708 & 9,723 & 9,713 & 9,626 \\
  07 01 & NR &  7 &   5 & 11,425 & 10,515 & 10,790 & 10,615 & 10,595 & 9,780 \\
  07 01 & NW & 13 &  11 & 30,567 & 28,324 & 27,920 & 27,684 & 27,955 & 27,422 \\
  07 01 & RM &  7 &   6 & 22,035 & 19,875 & 19,840 & 19,775 & 20,155 & 19,350 \\
  07 01 & SA &  4 &   9 & 40,520 & 40,260 & 40,315 & 40,285 & 40,485 & 39,420 \\
  07 01 & SC & 27 &   6 & 21,985 & 20,595 & 20,540 & 20,540 & 20,500 & 12,445 \\
  07 01 & SW & 20 &  24 & 186,755 & 185,680 & 185,875 & 184,755 & 184,490 & 181,775 \\
  \midrule
  07 15 & GB & 12 &  10 & 27,505 & 27,055 & 26,925 & 26,930 & 26,905 & 26,530 \\
  07 15 & NC & 21 &   1 & 2,485 & 1,880 & 1,885 & 1,885 & 1,885 & 1,880 \\
  07 15 & NR &  7 &  12 & 32,055 & 30,834 & 30,779 & 30,462 & 30,331 & 29,498 \\
  07 15 & NW & 13 &  15 & 101,585 & 99,900 & 99,310 & 98,880 & 99,120 & 96,375 \\
  07 15 & RM &  7 &   8 & 22,842 & 21,652 & 21,602 & 21,437 & 21,372 & 21,252 \\
  07 15 & SA &  4 &  11 & 42,265 & 42,230 & 42,220 & 42,255 & 42,240 & 41,375 \\
  07 15 & SC & 27 &  10 & 37,120 & 35,151 & 33,980 & 33,061 & 33,507 & 32,421 \\
  07 15 & SW & 20 &  50 & 344,585 & 342,160 & 341,825 & 341,265 & 339,615 & 334,405 \\
  \midrule
  08 01 & GB & 12 &   9 & 70,184 & 69,285 & 69,500 & 69,066 & 69,395 & 68,350 \\
  08 01 & NR &  7 &  31 & 109,393 & 108,016 & 108,388 & 106,446 & 105,178 & 100,931 \\
  08 01 & NW & 13 &  21 & 149,965 & 149,350 & 148,395 & 147,850 & 147,845 & 145,235 \\
  08 01 & RM &  7 &   8 & 27,775 & 26,906 & 26,571 & 26,466 & 26,551 & 26,001 \\
  08 01 & SA &  4 &  48 & 147,755 & 147,700 & 147,390 & 146,865 & 147,035 & 144,580 \\
  08 01 & SC & 27 &   8 & 148,420 & 147,230 & 146,950 & 147,030 & 146,705 & 141,535 \\
  08 01 & SW & 20 &  40 & 305,340 & 304,055 & 303,940 & 302,715 & 301,790 & 291,640 \\
  \midrule
  08 15 & EA &  1 &   2 & 5,835 & 5,830 & 5,820 & 5,820 & 5,820 & 5,810 \\
  08 15 & GB & 12 &  15 & 104,531 & 103,755 & 103,332 & 103,157 & 103,225 & 101,398 \\
  08 15 & NC & 21 &   7 & 50,710 & 49,405 & 48,950 & 49,020 & 48,860 & 47,245 \\
  08 15 & NR &  7 &  30 & 118,349 & 117,215 & 115,229 & 114,591 & 112,322 & 110,137 \\
  08 15 & NW & 13 &  22 & 166,556 & 164,517 & 165,216 & 163,823 & 163,706 & 162,296 \\
  08 15 & RM &  7 &  11 & 45,445 & 44,535 & 44,090 & 44,130 & 44,075 & 42,580 \\
  08 15 & SA &  4 &  62 & 200,347 & 200,272 & 199,352 & 199,384 & 199,117 & 195,302 \\
  08 15 & SC & 27 &   3 & 116,110 & 115,060 & 114,775 & 114,775 & 114,775 & 114,165 \\
  08 15 & SW & 20 &  16 & 138,425 & 136,765 & 137,330 & 136,245 & 135,905 & 131,460 \\
  \bottomrule
  \end{tabular}
\end{table}

\begin{table}[ht]
  \centering
  \small
  \setlength{\tabcolsep}{4pt}
  \caption{Cumulative burned area (acres) by solution for six 14-day periods in Summer 2024.}
  \label{tab:results_2024}
  \begin{tabular}{llrrrrrrrr}
  \toprule
  Date & GACC & Crews & Fires & Zero & Random & Area & Distance & Impact & Optimized \\
  \midrule
  06 01 & GB & 12 &   4 & 9,710 & 9,095 & 9,085 & 8,995 & 9,060 & 8,940 \\
  06 01 & NC & 21 &   1 & 24,090 & 23,120 & 23,170 & 23,170 & 23,170 & 23,000 \\
  06 01 & NW & 13 &   6 & 28,610 & 27,780 & 27,655 & 27,670 & 27,680 & 27,335 \\
  06 01 & RM &  7 &   3 & 5,980 & 5,830 & 5,815 & 5,815 & 5,815 & 5,620 \\
  06 01 & SA &  4 &  14 & 78,980 & 78,775 & 78,920 & 78,810 & 79,055 & 77,770 \\
  06 01 & SC & 27 &   5 & 32,055 & 30,730 & 30,225 & 30,300 & 30,235 & 29,110 \\
  06 01 & SW & 20 &  17 & 135,200 & 134,215 & 134,010 & 133,400 & 133,115 & 131,495 \\
  \midrule
  06 15 & GB & 12 &  10 & 53,235 & 52,261 & 52,265 & 51,267 & 51,297 & 51,042 \\
  06 15 & NC & 21 &   5 & 48,175 & 46,785 & 46,045 & 46,235 & 45,700 & 45,280 \\
  06 15 & NR &  7 &   1 & 2,745 & 2,570 & 2,570 & 2,570 & 2,570 & 2,515 \\
  06 15 & NW & 13 &  14 & 87,215 & 85,615 & 85,805 & 85,190 & 84,710 & 84,685 \\
  06 15 & RM &  7 &   8 & 44,885 & 43,760 & 43,485 & 43,240 & 42,870 & 41,845 \\
  06 15 & SA &  4 &  14 & 48,130 & 48,000 & 48,080 & 47,970 & 47,820 & 46,480 \\
  06 15 & SC & 27 &  10 & 84,240 & 83,110 & 82,410 & 80,300 & 80,405 & 80,135 \\
  06 15 & SW & 20 &  13 & 122,995 & 121,180 & 120,885 & 120,755 & 120,550 & 118,635 \\
  \midrule
  07 01 & EA &  1 &   1 & 6,770 & 6,760 & 6,765 & 6,765 & 6,765 & 6,745 \\
  07 01 & GB & 12 &  18 & 119,160 & 117,315 & 117,435 & 116,510 & 116,090 & 115,080 \\
  07 01 & NC & 21 &   7 & 44,185 & 42,395 & 42,130 & 41,710 & 41,690 & 41,115 \\
  07 01 & NR &  7 &  11 & 42,475 & 41,965 & 41,785 & 41,860 & 41,830 & 41,460 \\
  07 01 & NW & 13 &  25 & 196,380 & 195,494 & 195,000 & 194,810 & 193,554 & 191,170 \\
  07 01 & RM &  7 &   6 & 20,270 & 19,074 & 18,999 & 19,009 & 18,999 & 18,944 \\
  07 01 & SA &  4 &   9 & 35,905 & 35,850 & 35,875 & 35,870 & 35,870 & 34,800 \\
  07 01 & SC & 27 &  17 & 131,820 & 129,215 & 128,190 & 127,130 & 126,980 & 125,875 \\
  07 01 & SW & 20 &  22 & 118,100 & 115,975 & 114,965 & 115,210 & 115,400 & 110,820 \\
  \midrule
  07 15 & EA &  1 &   1 & 10,290 & 9,995 & 10,010 & 10,010 & 10,010 & 9,995 \\
  07 15 & GB & 12 &  41 & 173,265 & 172,780 & 172,470 & 171,085 & 169,809 & 168,224 \\
  07 15 & NC & 21 &  12 & 80,685 & 78,870 & 78,510 & 77,535 & 77,815 & 73,100 \\
  07 15 & NR &  7 &  31 & 180,815 & 180,275 & 180,415 & 178,035 & 179,580 & 176,150 \\
  07 15 & NW & 13 &  69 & 755,047 & 753,992 & 754,442 & 751,528 & 750,033 & 746,485 \\
  07 15 & RM &  7 &  12 & 70,440 & 69,215 & 68,885 & 68,810 & 68,900 & 66,705 \\
  07 15 & SA &  4 &   6 & 17,950 & 17,845 & 17,850 & 17,835 & 17,930 & 16,850 \\
  07 15 & SC & 27 &  16 & 189,060 & 186,900 & 186,445 & 185,595 & 185,360 & 182,245 \\
  07 15 & SW & 20 &  32 & 265,965 & 264,800 & 264,715 & 264,470 & 263,360 & 261,035 \\
  \midrule
  08 01 & GB & 12 &  43 & 279,662 & 278,591 & 278,362 & 277,422 & 276,066 & 273,421 \\
  08 01 & NC & 21 &   7 & 467,250 & 466,845 & 466,415 & 465,815 & 466,035 & 464,805 \\
  08 01 & NR &  7 &  38 & 143,070 & 142,563 & 142,610 & 141,267 & 140,972 & 137,911 \\
  08 01 & NW & 13 &  50 & 1,450,725 & 1,449,550 & 1,449,505 & 1,448,255 & 1,447,680 & 1,441,065 \\
  08 01 & RM &  7 &  13 & 121,650 & 120,890 & 120,785 & 120,390 & 120,225 & 119,415 \\
  08 01 & SA &  4 &  17 & 59,105 & 59,045 & 59,105 & 59,040 & 59,055 & 57,795 \\
  08 01 & SC & 27 &  18 & 238,527 & 237,332 & 236,537 & 235,602 & 234,782 & 232,807 \\
  08 01 & SW & 20 &  14 & 150,345 & 149,135 & 148,975 & 148,570 & 148,205 & 145,985 \\
  \midrule
  08 15 & GB & 12 &  37 & 439,021 & 435,981 & 438,281 & 436,691 & 435,321 & 431,216 \\
  08 15 & NC & 21 &   3 & 463,495 & 462,905 & 462,635 & 462,695 & 462,635 & 462,140 \\
  08 15 & NR &  7 &  38 & 96,324 & 92,262 & 94,374 & 90,957 & 90,256 & 84,246 \\
  08 15 & NW & 13 &  34 & 1,107,870 & 1,106,710 & 1,107,225 & 1,105,030 & 1,104,825 & 1,102,515 \\
  08 15 & RM &  7 &  10 & 241,055 & 240,780 & 240,525 & 240,600 & 240,500 & 237,880 \\
  08 15 & SA &  4 &  26 & 70,455 & 70,440 & 70,375 & 70,135 & 70,170 & 68,685 \\
  08 15 & SC & 27 &  10 & 136,265 & 135,057 & 134,745 & 133,527 & 133,777 & 130,806 \\
  08 15 & SW & 20 &  14 & 155,505 & 152,815 & 154,375 & 152,205 & 152,105 & 149,850 \\
  \bottomrule
  \end{tabular}
\end{table}

\begin{table}[ht]
  \centering
  \small
  \caption{Cumulative burned area (acres) with inter-GACC coordination in Summer 2023 and 2024.}
  \label{tab:bpc_results}
  \setlength{\tabcolsep}{4pt}
  \resizebox{\linewidth}{!}{
  \begin{tabular}{llrrrrrrrr}
  \toprule
  Date & GACC Group & Crews & Fires & Zero & Random & Area & Distance & Impact & Optimized \\
  \midrule
  \multicolumn{10}{l}{\textit{2023}} \\
  06 01 & All$^{a}$                      & 112 &  42 & 239,070 & 230,511 & 232,045 & 230,980 & 229,067 & 226,139 \\
  06 15 & All$^{a}$                      & 112 &  54 & 350,420 & 341,608 & 338,041 & 337,013 & 335,978 & 331,818 \\
  07 01 & East                     &  41 &  16 & 41,992 & 37,894 & 36,680 & 36,624 & 37,255 & 34,385 \\
  07 01 & West         &  71 &  55 & 287,165 & 280,822 & 280,655 & 277,895 & 277,395 & 262,927 \\
  07 15 & East                     &  41 &  28 & 136,125 & 131,067 & 130,740 & 128,882 & 128,223 & 125,283 \\
  07 15 & West         &  71 &  89 & 474,317 & 466,678 & 465,837 & 462,702 & 460,913 & 456,968 \\
  08 01 & East                     &  41 &  52 & 259,358 & 253,358 & 255,825 & 252,756 & 247,653 & 243,673 \\
  08 01 & West         &  71 & 113 & 699,474 & 695,370 & 694,464 & 690,626 & 688,286 & 666,720 \\
  08 15 & East                     &  41 &  59 & 335,615 & 329,533 & 330,090 & 326,617 & 323,118 & 318,896 \\
  08 15 & West         &  71 & 109 & 610,693 & 606,019 & 604,293 & 600,080 & 599,524 & 583,305 \\
  \midrule
  \multicolumn{10}{l}{\textit{2024}} \\
  06 01 & All$^{a}$                      & 112 &  50 & 314,625 & 307,925 & 305,500 & 305,095 & 304,050 & 300,220 \\
  06 15 & All$^{a}$                      & 112 &  75 & 491,620 & 480,932 & 479,690 & 472,662 & 474,287 & 469,067 \\
  07 01 & East                     &  41 &  43 & 283,040 & 279,421 & 277,530 & 276,975 & 275,109 & 273,095 \\
  07 01 & West         &  71 &  73 & 432,025 & 424,254 & 424,755 & 419,909 & 419,134 & 412,056 \\
  07 15 & East                     &  41 & 112 & 1,016,547 & 1,013,455 & 1,012,962 & 1,006,956 & 1,005,523 & 996,050 \\
  07 15 & West         &  71 & 108 & 726,970 & 718,502 & 721,565 & 718,000 & 712,809 & 705,004 \\
  08 01 & East                     &  41 &  95 & 2,061,045 & 2,058,136 & 2,057,780 & 2,052,008 & 2,052,322 & 2,042,372 \\
  08 01 & West         &  71 & 105 & 849,289 & 843,557 & 843,289 & 840,134 & 837,382 & 829,571 \\
  08 15 & East                     &  41 &  75 & 1,667,689 & 1,659,755 & 1,666,114 & 1,656,816 & 1,655,619 & 1,645,018 \\
  08 15 & West         &  71 &  97 & 1,042,301 & 1,034,330 & 1,037,551 & 1,032,363 & 1,029,533 & 1,016,722 \\
  \bottomrule
  \end{tabular}
  }
\end{table}

\end{APPENDICES}
    
\end{document}